\documentclass[11pt ]{article}
\usepackage[english]{babel}
\usepackage[latin1]{inputenc}
\usepackage{geometry}
\usepackage{graphicx}		
\usepackage[leqno]{amsmath}
\usepackage{amsfonts} 
\usepackage{amssymb}
\usepackage{amsthm}
\usepackage[T1]{fontenc}
\usepackage{hyperref}
\usepackage{stmaryrd}
\usepackage{dsfont}
\usepackage[all]{xy}
\usepackage{sectsty}
\usepackage{multicol}
\usepackage{eucal}
\usepackage[numbers]{natbib}
\usepackage{tikz}
\usepackage{yhmath}
\usepackage{makerobust}
\usepackage{fancyhdr}
\usepackage{eso-pic}
\usepackage{indentfirst}

\sectionfont{\centering}

\def\R{\mathbb{R}}
\def\N{\mathbb{N}}
\def\Z{\mathbb{Z}}
\def\Q{\mathbb{Q}}

\DeclareMathOperator{\Frac}{Frac}

\DeclareMathOperator{\id}{id}

\DeclareMathOperator{\Char}{Char}

\DeclareMathOperator{\supp}{Supp}

\DeclareMathOperator{\suppi}{Supp_\infty}

\newtheorem{theorem}{Theorem}[section]
\newtheorem{lemma}[theorem]{Lemma}
\newtheorem{prop}[theorem]{Proposition}
\newtheorem{cor}[theorem]{Corollary}
\newtheorem{definition}[theorem]{Definition}
\newtheorem{remark}[theorem]{Remark}
\theoremstyle{definition}
\newtheorem{example}[theorem]{Example}

\def\V{\mathcal{V}}
\def\I{\mathcal{I}}

\def\D{\mathcal{D}}
\def\Dk{\widehat{\mathcal{D}}^{(0)}_{\X, k}}

\def\Dkq{\widehat{\mathcal{D}}^{(0)}_{\X, k, \Q}}
\def\Dkp{\widehat{\mathcal{D}}^{(0)}_{\X', k}}
\def\Dkqp{\widehat{\mathcal{D}}^{(0)}_{\X', k, \Q}}
\def\X{\mathfrak{X}}
\def\m{\mathfrak{m}}
\def\Spf{\mathrm{Spf}\,}

\def\Spec{\mathrm{Spec}\,}

\def\O{\mathcal{O}}

\def\sp{\mathrm{sp}}
\def\Nb{\overline{N}_k}

\def\F{\mathcal{F}}
\def\M{\mathcal{M}}
\def\Nn{\mathcal{N}}
\def\L{\mathcal{L}}
\def\Ext{\mathcal{E}xt}

\def\CC{\mathrm{CC}}
\def\Irr{\mathrm{Irr}_v}

\def\Di{\mathcal{D}_{\X, \infty}}
\def\Dip{\mathcal{D}_{\X', \infty}}

\def\Fkr{\mathcal{F}_{\X,k , r}}
\def\Fkrp{\mathcal{F}_{\X',k , r}}
\def\Fi{\F_{\X, \infty}}
\def\Fip{\F_{\X' , \infty}}
\def\Fir{\F_{\X , \infty , r}}
\def\Firp{\F_{\X' , \infty , r}}

\def\Mk{\M_{\X , k}}
\def\Mkp{\M_{\X' , k}}
\def\Mp{\M_{\X'}}
\def\Mkr{\mathcal{M}_{\X , k,r}}
\def\Mr{\mathcal{M}_{\X , \infty , r}}

\def\Mkrp{\mathcal{M}_{\X' ,k,r}}
\def\Mrp{\mathcal{M}_{\X', \infty , r}}

\def\kp{k_{\X'}}
\def\pp{\phi_{\X'' , \X'}}
\def\G{\mathcal{BL}_\X}
\def\zr{\langle \X \rangle}

\def\sp{\mathrm{sp_{\X'}}}

\def\spp{\mathrm{sp_{\X''}}}
\def\Dzr{\D_{\zr}}

\def\Mzr{\M_{\zr}}

\def\Dcap{\wideparen{\D}_{\X_K}}

\def\Mp{\M_{\X'}}

\def\vqp{K\langle \partial \rangle^{(k , r)}}

\def\Arig{\mathbb{A}^{1 , \text{rig}}_K}

\title{Characteristic cycles for coadmissible $\D$-modules on smooth rigid analytic curves}
\author{Raoul Hallopeau}
\date{}

\begin{document}
\maketitle

\begin{abstract}
Let $\mathfrak{X}$ be a formal smooth curve over a complete discrete valuation ring of mixed characteristic and let $\mathfrak{X}_K$ be its generic fiber.
We consider respectively over $\mathfrak{X}$ and $\X_K$ the sheaves of differential operators $\mathcal{D}_{\mathfrak{X}, \infty}$ and $\wideparen{\D}_{\mathfrak{X}_K}$ with a rapid convergence condition.
In this article, we define a characteristic variety as a subset of the cotangent space $T^*\mathfrak{X}_K$ together with a characteristic cycle for coadmissible $\wideparen{\D}_{\mathfrak{X}_K}$-modules.
We deduce a notion of "sub-holonomicity" for coadmissible $\wideparen{\D}_{\mathfrak{X}_K}$-modules which is equivalent to being generically an integrable connection.
When $\mathfrak{X}$ is quasi-compact, we get an Artinian category of sub-holonomic $\wideparen{\D}_{\mathfrak{X}_K}$-modules which are weakly-holonomic.
Moreover, we prove that a coadmissible $\wideparen{\D}_{\mathfrak{X}_K}$-modules is sub-holonomic if and only if the corresponding coadmissible $\mathcal{D}_{\mathfrak{X}, \infty}$-module is.
\end{abstract}

\tableofcontents

\section{Introduction}\label{intro}

Let $\X$ be a smooth formal scheme over a complete discrete valuation ring $\V$ of mixed characteristic.
In \cite{huyghe}, Huyghe-Schmidt-Strauch have introduced a sheaf $\Di = \varprojlim_k \Dkq$ of rapidly converging differential operators obtained by adding congruence levels $k \in \N$ to Berthelot's arithmetic differential operators sheaves.
More precisely, if $U$ is an affine open subset of $\X$ equipped with local coordinates whose derivations are denoted by $\partial_1, \dots, \partial_d$, if $\varpi$ is a uniformizer of $\V$ and if $| \alpha| := \alpha_1 + \dots + \alpha_d \in \N$ for any $\alpha \in \N^d$, then
\[ \Dk(U) = \left\{ \sum_{\alpha \in \N^d} a_\alpha \cdot \partial_1^{\alpha_1} \dots \partial_d^{\alpha_d},~~ a_\alpha \in \O_\X(U) ~,  ~~|a_\alpha| \cdot |\varpi|^{-k | \alpha|} \underset{ | \alpha| \to \infty}{\longrightarrow} 0 \right\} .\]
Let us note $\Dkq := \Dk \otimes_\V K$ with $K$ the fraction ring of $\V$.
The algebras embedding $\widehat{\mathcal{D}}^{(0)}_{\X , k+1 , \Q}(U) \hookrightarrow \Dkq(U)$ induce transition homomorphisms $\widehat{\mathcal{D}}^{(0)}_{\X , k+1 , \Q} \to \Dkq$ for each congruence level $k$.
The sheaf $\Di := \varprojlim_k \Dkq$ is locally a Fr\'echet-Stein $K$-algebra:
\[ \Di(U) =\displaystyle \left\{ \sum_{\alpha \in \N^d} a_\alpha \cdot \partial_1^{\alpha_1} \dots \partial_d^{\alpha_d} :  \,  a_\alpha \in \O_{\X, \Q} (U)~~ \mathrm{st} ~~ \forall \eta>0,~  \lim_{|\alpha| \to \infty} |a_\alpha| \cdot \eta^{|\alpha|} =0 \right\}.\]
In this setting, the relevant $\Di$-modules, playing the role of coherent modules and coming from $p$-adic representations, are the coadmissible $\Di$-modules.
More precisely, a $\Di$-module is said to be coadmissible if this is isomorphic to a projective limit of coherent $\Dkq$-modules $\M_k$ such that $\Dkq \otimes_{\widehat{\mathcal{D}}^{(0)}_{\X, k+1 , \Q}} \M_{k+1} \simeq \M_k$.
The resulting abelian category of coadmissible $\Di$-modules is equivalent to the category of coadmissible $\Dzr$-modules over the Zariski-Riemann $\zr$ space associated with $\X$.
Let us note that Ardakov-Wadsley have developed in \cite{ardakov} a theory of coadmissible $\wideparen{\D}$-modules on smooth rigid analytic spaces.
Their construction coincides with that of Huyghe-Schmidt-Strauch for the generic fiber $\X_K$ of the smooth formal scheme $\X$.

In this article, we introduce a characteristic variety together with a characteristic cycle for coadmissible $\Dzr$-modules and coadmissible $\Dcap$-modules in the case of a smooth formal curve $\X$.
Such a variety is for example an important tool in the study of $\D$-modules over a smooth complex variety $X$, where holonomic $\D$-modules are classically characterized by having a characteristic variety of dimension less than or equal to the dimension of $X$.
However, unlike the complex case, the characteristic variety alone is not sufficient to yield a good notion of holonomicity in rigid analytic geometry.
For instance, one must add a Frobenius structure to Berthelot's arithmetic $\D$-modules to obtain finiteness results.
Moreover, the "holonomic" modules we define in this paper by asking the characteristic variety to be of dimension at most one are not be stable by all the classical six operations, like for example push-forward.
As a consequence, we call these modules sub-holonomic and not holonomic.
Ardakov-Bode-Wadsley introduced in \cite{ABW2} a category of weakly holonomic coadmissible $\wideparen{\D}$-modules using the classical cohomological characterization of holonomic modules.
However, this category is too large, as weakly holonomic $\wideparen{\D}$-modules are not necessarily of finite length.
Additionally, Bode introduced in \cite{bodehol} a notion of holonomicity for coadmissible $\wideparen{\D}$-modules modeled on Caro's construction for arithmetic $\D$-modules, without using characteristic varieties.
Nevertheless, defining a characteristic variety for coadmissible $\wideparen{\D}$-modules remains an important invariant.
On the smooth rigid analytic space $\X_K$, the differential operators considered are not of finite order, so the classical method from complex geometry does not apply.
A natural alternative approach, adopted here, uses microlocalization techniques applied in \cite{hallopeau3} to the sheaf $\Di$.
We restrict ourselves to the one-dimensional case where we obtain a characteristic cycle for coadmissible $\Dcap$-modules.

\vspace{0.4cm}

We now detail the content of this paper.
Assume that $\X$ is a connected, quasi-compact and smooth formal curve over $\V$ and let $\X_K$ be its generic fiber.
We introduced in \cite{hallopeau3} a characteristic variety $\Char(\M)$, using microlocalization sheaves of $\Di$, for any coadmissible $\Di$-module $\M$ which is a closed subset of the cotangent space $T^*\X$ of $\X$.
This characteristic variety satisfies Bernstein's inequality: the module $\M$ is not zero if and only if $\dim(\Char(\M)) \geq 1$.
Section \ref{section1} of this paper recalls the construction of this characteristic variety, as well as basic facts about coadmissible $\Di$-modules and weakly holonomic $\Di$-modules.
A coadmissible $\Di$-module is then called sub-holonomic if its characteristic variety  has dimension less or equal to one.
We also associated to sub-holonomic $\Di$-modules characteristic cycles induced by their characteristic varieties, from which it follows that such modules have finite length.
Furthermore, we also proved that a coadmissible $\Di$-module is sub-holonomic if and only if this is generically an integrable connection.
In section \ref{section2}, we adapt the construction of the microlocalization sheaves of $\Di$ to the case of any admissible blow-up $\phi: \X' \to \X$ of $\X$: $\Dip$ admits microlocaziation sheaves with similar properties.
Then we introduce in section \ref{section3}, exactly as in \cite{hallopeau3}, a characteristic variety for coadmissible $\Dip$-modules.
At the end, we prove that the direct image functor $\phi_*$ induces an equivalence of categories between sub-holonomic $\Dip$-modules and sub-holonomic $\Di$-modules.
We now consider the Zariski-Riemann space of $\X$ which is the projective limit $\zr = \varprojlim \X'$ of all admissible blow-ups $\X'$ of $\X$.
We deduce in section \ref{section5} a notion of sub-holonomicity for coadmissible $\Dzr$-modules over the Zariski-Riemann space $\zr$ for which we have associated finite multiplicities.
Finally, we extend in section \ref{section6} this notion for coadmissible $\Dcap$-modules via the specialization map $\text{sp} : \X_K \to \zr$ and we get the following.

\begin{theorem}
We can associate to any coadmissible $\Dcap$-module $\M$ a characteristic variety $\Char(\M)$ which is a closed subset of $T^*\X_K$ and which satisfies Bernstein's inequality: if $\M \neq 0$, then $\Char(\M)$ is equal to $T^*\X$ or equidimensional of dimension 1.
Moreover, if $\dim(\Char(\M)) \leq 1$, then $\M$ has an associated finite characteristic cycle.
\end{theorem}

A coadmissible $\Dcap$-module is then called sub-holonomic if $\dim(\Char(\M)) \leq 1$.
For example, the coadmissible $\Dcap$-module $\Dcap / P$ with $P \in \Dcap$ is  sub-holonomic if and only if $P$ is a finite differential operator.
Finally, we prove in section \ref{section6} the following properties of sub-holonomic $\Dcap$-modules.

\begin{theorem}\,
\begin{enumerate}
\item
Sub-holonomic $\Dcap$-modules form an Artinian subcategory of the weakly holonomic $\Dcap$-modules of Ardakov-Bode-Wadsley.
Moreover, a coadmissible $\Dcap$-module is sub-holonomic if and only if this is generically an integrable connection.
\item
Let $\text{sp}_\X : \X_K \to \X$ be the specialisation map.
The functor $(\text{sp}_\X)_*$ induces an equivalence of categories between sub-holonomic modules over respectively $\Dcap$ and $\Di$.
\end{enumerate}
\end{theorem}

Let us point out once again that this category of sub-holonomic $\Dcap$-modules is not stable under all six standard cohomological operations, which justifies the use of the term sub-holonomicity.
More precisely, we have explicit examples of direct images of integrable connection modules on a rigid analytic curve that are not even coadmissible.
These are the subject of the article \cite{bitoun} of Bitoun-Bode.
Let $X_K = \text{Sp}(K \langle x \rangle)$, $U_K = X_K \backslash \{0\}$, and $j : U_K \hookrightarrow X_K$ the natural inclusion.
Consider $P_\lambda = x \cdot \partial - \lambda$ for a scalar $\lambda \in K$ with $\partial = \frac{d}{dx}$.
The coadmissible $\wideparen{\D}_{U_K}$-module $\M_\lambda = \wideparen{\D}_{U_K} / P_\lambda \simeq \O_{U_K} \cdot x^\lambda$ is an integrable connection since $x$ is invertible on the open subset $U_K$.
Theorem 1.1 of \cite{bitoun} states that the pushforward module $j_*\M_\lambda$ is coadmissible if and only if $\lambda$ is of positive type.
However, scalars of type zero do exist.

\paragraph{Acknowledgement}
I would like to mention the JSPS summer program, thanks to which I was able to go to Japan in the summer of 2024 to work with Tomoyuki Abe.
I also thank him for all his discussions with me and the explanations in some of his articles on microlocalization.

\paragraph{Notations}

\begin{enumerate}
\item[$\bullet$]
$\V$ is a discrete valuation ring of mixed characteristic $(0,p)$ with maximal ideal $\m$ and with perfect residue field $\kappa$.
Note $\varpi$ a uniformizer of $\V$, $| \cdot |$ the normalised absolute value, ie $|\varpi| = \frac{1}{p}$, and $K =\Frac(\V)$ its fraction ring.
\item[$\bullet$]
$\X$ is a smooth, connected and quai-compact formal curve over $\V$ whose defining ideal is generated by $\varpi$.
Denote by $\X_K$ the rigid analytic $K$-space associated with $\X$ and by $X$ its special fiber which is a smooth, connected and quasi-compact $\kappa$-curve.
\item[$\bullet$]
$U$ is an affine open subset of $\X$ on which we have a local coordinate $y$.
We denote by $\partial$ the corresponding derivation.
\end{enumerate}

\section{Sub-holonomic $\Di$-modules}\label{section1}

We briefly recall in this first section the definition of the sheaf $\Di = \varprojlim_k \Dkq$ of rapidly converging differential operators over the formal smooth curve $\X$ introduced by Huyghe-Schmidt-Strauch in \cite{huyghe} as well as the definition of sub-holonomic $\Di$-modules as done in \cite{hallopeau3}.

\subsection{Microlocalization of $\Di$}

Let $\X$ be a smooth, connected and quasi-compact formal curve over $\V$.
The sheaf $\Dkq$ of differential operators with congruence level $k \in \N$ can be described locally by the following.
For any affine open subset $U$ of $\X$ on which we have an \'etale coordinate with derivation $\partial$, the sections of $\Dkq$ on $U$ are given by
\[ \Dkq(U) = \left\{ \sum_{n \in \N} a_n \cdot (\varpi^k \partial)^n , ~~ a_n \in \O_{\X, \Q}(U) , ~~| a_n |\underset{n \to \infty}{\longrightarrow} 0 \right\} \, .\]
Since the formal curve $\X$ is smooth and connected, the spectral norm $| \cdot |$ of the affinoïd $K$-algebra $\O_{\X , \Q}(U)$ is a non-archimedean absolute value.
The norm of a differential operator $P = \sum_{n \in \N} a_n \cdot (\varpi^k \partial)^n$ of $\Dkq(U)$ is defined by $| P |_k := \max_{n \in \N} \{ |a_n| \}$.
It was proved in \cite[proposition 2.6]{hallopeau1} that this norm $| \cdot |_k$ is multiplicative and complete.
In other words, $(\Dkq(U), |\cdot|_k)$ is a Banach $K$-algebra. Moreover, for any integer $r \in \{ 0 , \dots , k \}$, $\Dkq(U)$ is a sub-algebra of $\widehat{\D}_{\X , r , \Q}^{(0)}(U)$.
In particular, the norm $|\cdot |_r$ is also defined on $\Dkq(U)$ and one can check that $|\cdot |_r \leq |\cdot |_k$.
The order of a differential operator $P= \sum_{n \in \N} a_n \cdot (\varpi^k \partial)^n$ of $\Dkq(U)$ is the integer $\Nb(P) := \max\{ n , ~ |a_n| = |P|_k \} \in \N$.
If $|P|_k = 1$, then $\Nb(P)$ is exactly the order of the finite differential operator $\bar{P} = ( P \mod \varpi)$.
We consider the transition maps $\widehat{\D}_{\X , k+1 , \Q}^{(0)} \hookrightarrow \Dkq$ induced by the local embeddings $\widehat{\D}_{\X , k+1 , \Q}^{(0)}(U) \hookrightarrow \Dkq(U)$.
By definition, $\Di := \varprojlim_k \Dkq$. Sections of this sheaf are locally Fréchet-Stein $K$-algebras with a rapidly converging condition:
\[ \Di(U) =\displaystyle \left\{ \sum_{n \in \N} a_n \cdot \partial^n :  \,  a_n \in \O_{\X, \Q} (U)~~ \mathrm{st} ~~ \forall \eta>0,~  \lim_{n \to \infty} |a_n| \cdot \eta^{n} =0 \right\}.\]
This is an analogue for a smooth formel scheme of the sheaf $\wideparen{\D}$ of Ardakov-Wadsley as introduced in \cite{ardakov}.
For Fr\'echet-Stein algebras, coadmissibility plays the role of coherence. For example, any finitely presented $\Dkq$-module is coadmissible.
\begin{definition}
A $\Di$-module is said to be coadmissible if it is isomorphic to a projective limit of coherent $\Dkq$-modules $\M_k$ such that $\Dkq \otimes_{\widehat{\mathcal{D}}^{(0)}_{\X, k+1 , \Q}} \M_{k+1} \simeq \M_k$.
We denote by $\mathcal{C}_\X$ the abelian category of coadmissible $\Di$-modules.
\end{definition}

Let us recall that a Banach $K$-algebra $A$ endowed with a multiplicative norm $| \cdot |$ is said to be quasi-abelian if there exists a constant $\gamma \in [0,1[$ such that for any pair $(a, b)$ of elements in $A$, $| ab -ba | = | [a,b] | \leq \gamma \cdot |ab| = \gamma \cdot  |a| \cdot |b|$.
Then, for a finite number of quasi-abelian norms $| \cdot |_1 , \dots,  | \cdot |_m$ on $A$ and for a multiplicative subset $S$ of $A$, one can construct a localization $\varphi : A \to B:= A  \langle | \cdot |_1 , \dots,  | \cdot |_m, S \rangle$ of $A$ given by the following universal property.
Let $(D , \|\cdot \|_D)$ be a Banach $K$-algebra together with a morphism $f : A \to D$ such that $f(S) \subset D^\times$. Assume that there exists $c > 0$ such that
\[ \forall (s,a) \in S \times A, ~~ \| f(s)^{-1} \cdot f(a) \|_D \leq c \cdot \max_{1 \leq i \leq m} \left\{ |s|_i^{-1} \cdot |a|_i \right\} . \]
Then there is a unique continuous homomorphism of $K$-algebras $\tilde{f} : B \to D$ such that $\tilde{f} \circ \varphi = f$.
This localization is a Banach $K$-algebra in which elements of $S$ are all invertible.
One can check \cite[appendix]{zab} for more details.
We assume from now on that $k > 0$.
We proved in \cite[proposition 3.5]{hallopeau3} that the Banach $K$-algebra $(\Dkq(U), |\cdot|_r)$ is quasi-abelian for any integer $r \in \{1 , \dots , k \}$:
\[ \forall P, Q \in \Dkq(U), ~ |PQ-QP|_r \leq p^{-r} \cdot |P|_r \cdot |Q|_r . \]
For $k \geq r \geq 1$, the microlocalization sheaf $\Fkr$ introduced in \cite[section 3.3]{hallopeau3} is then defined locally by
\[ (\Fkr(U) , \| \cdot \|_{k, r}) := \Dkq(U) \left\langle |\cdot|_r , \dots ,  |\cdot|_k \, ; \{ \partial^n, ~ n \in \N \} \right\rangle .\]
The action of $\partial^{-1}$ on a section $f$ of $\O_{\X , \Q}(U)$ is given by
\[ \partial^{-1} f = \sum_{n=0}^{+\infty} (-1)^n \cdot \partial^n(f) \cdot \partial^{-(n+1)} \in \F_{k , r}(U) .\]
The following result corresponds to \cite[proposition 3.16]{hallopeau3}.

\begin{prop}\label{prop1.1}
Any element $S$ of $\Fkr(U)$ can be uniquely written in the form
\[ S = \sum_{n = 0}^\infty a_n \cdot (\varpi^k\partial)^n + \sum_{ n =1}^\infty  a_{-n} \cdot (\varpi^r\partial)^{-n} \]
with $a_n \to 0$ when $n \to \pm \infty$. Moreover, $\| S \|_{k , r} = \max_{n \in \Z} \{ |a_n| \}$.
\end{prop}

For $k \geq r \geq 1$, we consider the transition homomorphism of sheaves $\F_{\X, k+1 , r} \hookrightarrow  \Fkr$ induced by the local embeddings $\F_{\X, k+1 , r}(U) \hookrightarrow \Fkr(U)$ of $K$-algebras. 
Let us define the sheaf $\Fir := \varprojlim_{k \geq r} \Fkr$. Locally, $\Fir(U)$ is a Fréchet-Stein algebra given by
\[ \Fir(U) =\displaystyle \left\{ P + \sum_{n = 1}^\infty  a_{-n} \cdot (\varpi^r\partial)^{-n} :  \, P \in \Di(U) , ~ a_{-n} \in \O_{\X, \Q} (U),~~  \lim_{n \to \infty} |a_{-n}|  =0 \right\} .\]
We proved in \cite[theorem 4.7]{hallopeau3} the following.
\begin{theorem}
Let $P \in \Di(U)$ and $V \subset U$ an open subset. Then $P$ is invertible in $\Fir(V)$ if and only if
\begin{enumerate}
\item
$P = \sum_{n=0}^d a_n \cdot \partial^n$ is finite of some order $d$ in $\Di(U)$ ,
\item
$a_d$ is invertible in $\O_{\X , \Q}(V)$,
\item
$\forall n\in \{ 0 , \dots , d-1 \}, ~~  |a_n|  < |a_d| \cdot p^{r(d-n)}$.
\end{enumerate}
\end{theorem}

For integers $k > r \geq 1$, we have injective maps of $K$-algebras $\Fkr(U) \hookrightarrow \F_{\X , k , r+1}(U)$ inducing a sheaf homomorphism $\Fkr \hookrightarrow \F_{\X , k , r+1}$ such that the following diagram commutes:
\[ \xymatrix@R=3pc @C=3pc{ \F_{k+1 , r} \ar@{^{(}->}[r]\ar@{^{(}->}[d] & \F_{k , r} \ar@{^{(}->}[d] \\ \F_{k+1 , r+1} \ar@{^{(}->}[r] & \F_{k , r+1} .} \]
Passing to the projective limit over $k$, we get an injective sheaf homomorphism $\Fir \to \F_{\X , \infty , r+1}$.
The microlocalization sheaf $\Fi$ of $\Di$ is then the injective limit of the sheaves $\Fir$, ie
\[ \Fi := \varinjlim_{r \geq 1} \Fir .\]

\subsection{Sub-holonomic $\Di$-modules}\label{section1.2}

Let $\M_\X = \varprojlim_k \Mk$ be a coadmissible $\Di$-module. It means that for any congruence level $k \in \N$, the $\Dkq$-module  $\Mk$ is coherent and the $\widehat{\D}^{(0)}_{\X , k+1 , \Q}$-linear transition map $\M_{\X , k+1} \to \Mk$ induce a $\Dkq$-linear isomorphism $\Mk \simeq \Dkq \otimes_{\widehat{\D}^{(0)}_{\X , k+1 , \Q}} \M_{\X , k+1}$.
We introduce the coherent $\Fkr$-module $\Mkr := \Fkr \otimes_{\Dkq} \Mk$ for any integers $k \geq r \geq 1$. There are transition homomorphisms $\M_{\X, k+1 , r} \to \Mkr$ induced by $\M_{\X , k+1} \to \Mk$.
We have proved in \cite[section 5.2]{hallopeau3} that $\Mr := \varprojlim_{k \geq r} \Mkr$ is a coadmissible $\Fir$-module.
As a consequence, the sequence of supports $(\supp \Mkr)_k$ is increasing and we set $\supp (\Mr) := \overline{ \bigcup_{k \geq r} (\supp \Mkr)} \subset \X$.
It is worth noting that the usual support of a coadmissible module is not necessarily closed.
Moreover, one has $\supp (\M_{\X , \infty , r+1}) \subset \supp (\Mr)$ and we introduce the following "infinite" support of $\M_\X$:
\[ \suppi (\M_\X) := \bigcap_{r \geq 1} \supp (\Mr) \subset \X . \]
Since the formal curve $\X$ is Noetherian, there exists an integer $r_0 \geq 1$ such that for any $r \geq r_0$, $\suppi (\M_\X) = \supp(\Mr)$. 
Moreover, when $\dim(\suppi (\M_\X)) = 0$, there exists a congruence level $k_0 \geq k$ such that for any $k \geq k_0$,
\[ \suppi (\M_\X) = \supp(\Mr) = \supp(\Mkr) .\]
We have introduced in \cite[section 5.3]{hallopeau3} a characteristic variety $\Char(\M_\X)$ for any coadmissible $\Di$-module $\M_\X$ which is a closed subset of the cotangent space $T^*\X$ of $\X$.
It satisfies the Bernstein's inequality: the module $\M_\X$ is non-zero if and only if the irreducible components of $\Char(\M_\X)$ all have dimension at least one.

\begin{definition}
A coadmissible $\Di$-module $\M_\X$ is called sub-holonomic if $\dim \Char(\M_\X) \leq 1$, or equivalently if $\dim \suppi(\M_\X) = 0$.
We denote by $\mathcal{SH}_\X$ the abelian category of sub-holonomic $\Di$-modules.
\end{definition}

Sub-holonomic coadmissible $\Di$-modules are weakly holonomic according to the definition of Ardakov-Bode-Wadsley \cite{ABW2}.
We briefly recall it.
Let $\M_\X = \varprojlim_{k \geq 0} \Mk$ be a coadmissible $\Di$-module.
For any $n \in \N$, the right $\Di$-module $\Ext^n_{\Di}(\M_\X , \Di) := \varprojlim_{k \geq 1} \Ext^n_{\Dkq}(\Mk , \Dkq)$ is coadmissible.
The module $\M_\X$ is then said to be weakly holonomic if for all integer $n \neq \dim \X = 1$, $\Ext^n_{\Di}(\M_\X , \Di) = 0$.
This is equivalent to ask all the coherent $\Dkq$-modules $\Mk$ to be holonomic in the sense that all the characteristic varieties $\Char(\M_k) \subset T^*X$ have dimension at most one.
If the coadmissible $\Di$-module $\M_\X$ is non-zero and sub-holonomic, then its characteristic variety $\Char(\M_\X)$ has dimension one.
More precisely, this variety is composed of a finite number of vertical irreducible components and an horizontal irreducible component given by the zero-section of the canonical projection $\pi_\X : T^*\X \to \X$.
See for example the following figure \ref{figurechar}.
In fact, these vertical components are exactly the vertical lines passing through the points of the infinite support $\suppi (\M_\X)$ of $\M_\X$.
We now relate sub-holonomicity to integrable connections.

\begin{prop}\label{propconnections}
Let $\M$ be an integrable connection in the sense that $\M$ is a coherent $\O_{\X , \Q}$-module together with an integrable connection $\nabla$.
Then $\M$ is a coadmissible $\Di$-module.
In other words, integrable connections are exactly the coadmissible $\Di$-modules which are also $\O_{\X , \Q}$-coherent.
\end{prop}
\begin{proof}
We can assume that $\X$ is affine.
Since $\dim(\X) = 1$, $\O_{\X , \Q}(\X)$ is a principal ideal domain. As a consequence, $\M$ is free of finite rank over $\O_{\X , \Q}$.
Let $m_1, \dots , m_s$ be a basis for $\M$ composed of global sections.
The connection $\nabla$ induces an action of $\partial$ on $\M$ satisfying $\partial(f\cdot m) = \partial(f) \cdot m + f \cdot \partial(m)$ for any $f \in \O_{\X , \Q}$ and $m \in \M$.
In particular, $\partial$ is determined by some matrix $A \in \text{M}_{s\times s}(\O_{\X , \Q}(\X))$: $\partial(\underline{m}) = A \cdot \underline{m}$ where $\underline{m} = (m_1 , \dots , m_s)$.
We now check that this action of $\partial$ on $\M$ extends to a left $\Dkq$-module structure for $k$ sufficiently large, ie $P(m_i)$ is well defined inside $\M$ for any differential operator $P \in \Dkq(\X)$ and any integer $i \in \{ 1 , \dots , s\}$.
We denote by $(P_n)_n$ the sequence of partial sums of $P$.
We have $\partial^2(\underline{m}) = \partial(A \cdot \underline{m} ) = \partial(A) \cdot \underline{m} + A \cdot \partial(\underline{m}) = (\partial(A) + A) \cdot \underline{m}$.
More generally, we prove by recursion that $\partial^n(\underline{m}) = S_n(A) \cdot \underline{m}$ for some polynomial $S_n(A) \in \Z[ A , \partial(A) , \dots \partial^n(A)]_{\leq n}$ of degree at most $n$ without constant term.
Let $| \cdot |_\infty$ be the supremum norm over respectively $\text{M}_{s\times s}(\O_{\X , \Q}(\X))$ and $\M^s = \oplus_{i = 1}^s \O_{\X , \Q}(\X) \cdot m_i$.
All the norms are equivalent on $\M^s$ and on $\text{M}_{s\times s}(\O_{\X , \Q}(\X))$ seen as $\O_{\X , \Q}(\X)$-module.
Thus, it suffices to check the convergence for the supremum norm.
Since $|\partial^\ell(A)|_\infty \leq |A|_\infty$ for all $\ell \in \N$, we have $|S_n(A)|_\infty \leq |A|_\infty^n$.
Let $k_0 \in \Z$ be such that $|A|_\infty = |\varpi|^{k_0}$.
It follows that for any positive integer $k \geq - k_0$, $|\varpi|^{n k} \cdot |\partial^n(A)|_\infty \leq 1$.
As a consequence, $| (P_{n + \ell} - P_n)(\underline{m})|_\infty \leq |P_{n + \ell} - P_n|_k$ and $(P_n(\underline{m}))_n$ is a Cauchy sequence in $(\M^s , | \cdot |_\infty)$.
Let us denote its limit by $P(\underline{m}) = ( P(m_1) , \dots , P(m_s))$.
This proves that $\M$ is a $\Dkq$-modules for the structure coming from the connection as soon as $k \geq k_0$ (let us point out the fact that $k_0$ is determined by the choice of $m_1 , \dots , m_s$).
It remains to check that this left $\Dkq$-module structure is independent from the choice of the basis of $\M$.
Let $m'_1 , \dots , m'_s$ be another basis. Then $\underline{m}' = B\cdot \underline{m}$ for some invertible matrix $B \in \text{M}_{s\times s}(\O_{\X , \Q}(\X))$.
The map $\phi : \M \simeq \bigoplus_{i = 1}^s \O_{\X , \Q}\cdot m_i \to \M \simeq \bigoplus_{i = 1}^s \O_{\X , \Q}\cdot m_i'$ induced by $\underline{m} \mapsto B^{-1}\cdot \underline{m}'$ is an isomorphism of $\O_{\X , \Q}$-modules.
We have $\partial(\underline{m}') = \partial(B \cdot \underline{m}) =(\partial(B) + A) \cdot \underline{m}$ and
\[ \partial(B^{-1}\cdot \underline{m}') = \partial(B^{-1}) \cdot \underline{m}' + B^{-1} \cdot \partial(\underline{m}')) = (\partial(B^{-1})B  + B^{-1}\partial(B) + A)\cdot \underline{m} .\]
Since $\partial(B^{-1})B  + B^{-1}\partial(B) = \partial(B^{-1} B) = \partial(1) = 0$, we get $\partial(B^{-1}\cdot \underline{m}') = \partial(\underline{m}) = A \cdot \underline{m}$.
By linearity, it follows that $\phi(\partial \cdot m) = \partial \cdot \phi(m)$ for any $m \in \M$.
In other words, the map $\phi$ is $\partial$-linear and thus $\D_{\X , \Q}$-linear.
Finally, we check by arguments similar to the preceding ones that this extends to $\Dkq$ for $k$ large enough (depending on the norm of $B$).
To resume, we have proved that the connection $\M$ is naturally a $\Dkq$-module for $k$ large enough.
Since $\M$ is $\O_{\X , \Q}$-coherent, this is also coherent as $\Dkq$-module.
It follows that $\M$ is a coadmissible $\Di$-module.
\end{proof}

\begin{remark}
This proposition remains true in higher dimension: any coherent $\O_{\X , \Q}$-module equipped with a connection is a coadmissible $\Di$-module for the structure coming from the connection.
The proof is globally the same, replacing a basis of $\M$ by a system of generators.
This result was also proved by Ardakov-Wadsley in \cite{ardakov2} for coadmissible $\wideparen{\D}$-modules.
\end{remark}

From now on, we identify integrable connections with $\O_{\X, \Q}$-coherent coadmissible $\Di$-modules.
A coadmissible $\Di$-module $\M_\X$ is sub-holonomic if and only if there exists a dense open subset $U$ of $\X$ such that $\M_{|U}$ is an integrable connection.
One can choose for $U$ the open subset $\X \backslash \suppi(\M_\X)$.
It was also proved in \cite[section 6.1]{hallopeau3} that a coadmissible $\Di$-module $\M_\X$ is an integrable connection if and only if $\Char(\M_\X) \subset \X$ (identified with the zero-section of $T^*\X$).
Moreover, because $\X$ has dimension one, a connection is automatically locally a free $\O_{\X , \Q}$-module of finite rank.

\begin{example}
Assume that $\X$ is affine with a local coordinate.
\begin{enumerate}
\item
Let $P\in \Di(\X)$. The coadmissible module $\Di/ P$ is sub-holonomic if and only if $P = \sum_{n = 0}^d a_n \cdot \partial^n$ is a finite order differential operator.
The characteristic variety $\Char(\Di / P)$ is then given in figure \ref{figurechar}, where $x_1, \dots, x_s$ denote the zeroes in $\X$ of the dominant coefficient $a_d$ of $P$.
\item
A $\Di$-module $\M$ is said to be algebraic is there exists a coherent $\D_{\X , \Q}^{(0)}$-module $\M^{alg}$ such that $\M \simeq \Di \otimes_{\D_{\X , \Q}^{(0)}} \M^{alg}$.
Then any locally algebraic $\Di$-module which is weakly holonomic is sub-holonomic.
\end{enumerate}
\end{example}

\begin{remark}
The category of sub-holonomic $\Di$-modules contains algebraic $\Di$-modules which are weakly holonomic.
We conjecture the converse: sub-holonomic $\Di$-modules are exactly the weakly holonomic $\Di$-modules that are locally algebraic.
\end{remark}

Let us mention that this category of coadmissible sub-holonomic $\Di$-modules does not satisfy all the six cohomological operations.
Since the direct image of a coadmissible $\wideparen{\D}$-module, or worse, an integrable connection, is not necessarily coadmissible, we have explicit examples of direct images of sub-holonomic $\wideparen{\D}$-modules that are not even coadmissible.
Such examples are the subject of Bitoun-Bode's article \cite{bitoun}.

\begin{example}
Let $X_K = \mathrm{Sp}(K \langle x \rangle)$ be the rigid analytic unit disc over $K$.
We consider the open subset $U_K = X_K \backslash \{0\}$ of $X_K$ obtained by removing the origin.
Let us denote by $\partial$ the derivation associated with $x$ and consider the differential operator $P_\lambda = x\cdot\partial - \lambda \in \wideparen{\D}_{X_K}(X_K)$ with $\lambda \in K$.
The $\wideparen{\D}_{U_K}$-module $\M_\lambda = \wideparen{\D}_{U_K} / P_\lambda \simeq \O_{U_K} \cdot x^\lambda$ is coadmissible.
In fact, this is an integrable connection over $U_K$ where $x$ is invertible.
If $j : U_K \hookrightarrow X_K$ is the natural embedding, then \cite[theorem 1.1]{bitoun} tells us that the $\wideparen{\D}_{X_K}$-module $j_*\M_{U_K}$ is not coadmissible when $\lambda$ is not of positive type.
There are scalars $\lambda$ which are not of positive type.
The following example is due to Le Bras.
Let us note $k_1 = p$ and $k_{n+1} = p^{2k_n}$ for $n > 1$.
The element $\lambda := \sum_{i=1}^\infty p^{k_i}$ of $\Q_p$ is of type zero.
\end{example}

\begin{figure}\label{fig1}
\begin{center}
\caption{$\Char(\Di /P)$ for $P = \sum_{n = 0}^d a_n \cdot \partial^n$}
\begin{tikzpicture}\label{figurechar}

\draw[thick][->] (-2,0) -- (8,0);
\draw (8.2,0) node[right] {$x$};
\draw [thick][->] (0,-1) -- (0,3.5);
\draw (0,3.7) node[above] {$\xi$};
\draw (0,0) node[below right] {$0$};

\draw[red][thick] (-1.5,0) -- (7,0);

\draw[red][thick] (-1, -1) -- (-1,3) ;
\draw (-1,0) node[below right]{$x_1$} ;

\draw[red][thick] (1.8, -1) -- (1.8,3) ;
\draw (1.8,0) node[below right]{$x_2$} ;

\draw[red][thick] (3.5 , -1) -- (3.5,3) ;
\draw (3.5,0) node[below right]{$x_i$} ;

\draw[red][thick] (6, -1) -- (6,3) ;
\draw (6,0) node[below right]{$x_s$} ;

\draw (7,3) node{$T^*\X$} ;

\end{tikzpicture}
\end{center}
\end{figure}

For this reason, the coadmissible $\Di$-modules whose characteristic varieties are of minimal dimension are only called sub-holonomic and not holonomic.
Nevertheless, we have associated with these coadmissible $\Di$-modules a finite characteristic cycle in \cite[section 6.2]{hallopeau3}.
In particular, sub-holonomic modules are of finite length. We end this section by quickly summarizing this.
Let $\M_\X = \varprojlim_k \Mk$ be a non-zero sub-holonomic $\Di$-module and $U = \X \backslash \suppi(\M_\X)$.
We define the horizontal multiplicity $m_0(\M_\X)$ of $\M_\X$ to be the rank of the locally free $\O_{U , \Q}$-module $(\M_\X)_{|U}$.
In particular, the zero section of $T^*\X$ is an irreducible component of the characteristic variety $\Char(\M_\X)$ if and only if $m_0(\M_\X) \geq 1$.
For a short exact sequence of sub-holonomic $\Di$-modules $0 \to \Nn_\X \to \M_\X \to \L_\X \to 0$, we have $m_0(\M_\X) = m_0(\Nn_\X) + m_0(\L_\X)$.
Finally, we denote by $\vqp$ the commutative sub-algebra of $\F_{k , r}(U)$ composed of the elements with coefficients in $K$ :
\[ \vqp := \left\{ \sum_{n = 0}^\infty \lambda_n \cdot (\varpi^k\partial)^n + \sum_{ n =1}^\infty  \lambda_{-n} \cdot (\varpi^r\partial)^{-n} \in \F_{k , r}(U)~: ~ \lambda_n \in K, ~ \lim_{n \to \pm \infty} \lambda_n = 0  \right\} . \]
We proved in \cite[section 6.2]{hallopeau3}, using the arguments of \cite[section 1.3]{adriano}, that for any point $x \in \suppi(\M_\X)$, the modules $(\Mkr)_x$ are free of the same finite rank over $\vqp$ as soon as $k \geq r \geq 1$ are large enough.
We define the multiplicity $m_x(\M_\X) = m_C(\M_\X)$ of the vertical component $C$ of $\Char(\M_\X)$ passing though $x$ to be this common rank.
We denote by $\Irr(\M_\X)$ the set of vertical irreducible components of the characteristic variety $\Char(\M_\X)$.
The characteristic cycle of $\M_\X$ is then given by the finite formal sum
\[ \CC(\M_\X) :=m_0(\M_\X)\cdot \X + \sum_{C \in \Irr(\M_\X)} m_C(\M_\X) \cdot C \]
where we always identify $\X$ with the zero section of its cotangent space $T^*\X$.
If $\M_\X = 0$, then all the multiplicities of $\M_\X$ are equal to zero and $\CC(\M) = 0$.
The following result is demonstrated in \cite[section 6.2]{hallopeau3}.

\begin{prop}
Let $0 \to \Nn_\X \to \M_\X \to \L_\X \to 0$ be an exact sequence of sub-holonomic coadmissible $\Di$-modules.
Then $\CC(\M_\X) = \CC(\Nn_\X) + \CC(\L_\X)$. Moreover, a sub-holonomic coadmissible $\Di$-module $\M_\X$ is zero if and only if $\CC(\M_\X) = 0$.
As a consequence, any sub-holonomic coadmissible $\Di$-module $\M_\X$ is of finite length less or equal to $\ell(\M_\X) := m_0(\M_\X) + \sum_{x \in \suppi(\M_\X)} m_x(\M_\X) \in \N$.
\end{prop}

To resume, we obtain (only in the dimension one case) an Artinian category of sub-holonomic coadmissible $\Di$-modules which are also weakly holonomic.
More precisely, this category is the largest one such that its objects are of finite length in the sense that the module $\Di / P$ is sub-holonomic if and only if it is of finite length.
We will generalize in next sections the notion of sub-holonomicity for coadmissible $\Dcap$-modules (where $\X_K$ is the rigid analytic curve associated to $\X$) to get a full subcategory of the one defined by Ardakov-Bode-Wadsley.

\section{Microlocalization for an admissible blow-up $\X'$ of $\X$}\label{section2}

Let $\phi : \X' \to \X$ be an admissible blow-up of $\X$ defined by a coherent sheaf of open ideals $\I$ of $\O_\X$ containing a non-negative power of the uniformizer $\varpi$.
Let us point out the fact that the ideal $\I$ is not determined by the blow-up $\phi : \X' \to \X$.
We denote by $k_\I$ the minimal integer $\ell$ such that $\varpi^\ell \in \I$ and we put $\kp := \min\left\{ k_\I : \X'~\mathrm{is~the~blowing~up~of~\I~on~\X} \right\}$.
Huyghe-Schmidt-Strauch have introduced in \cite{huyghe} a sheaf $\Dkqp$ of differential operators on the blow-up $\X'$ for any congruence level $k \geq \kp$.
If $V$ if an affine open subset in $\phi^{-1}(U)$ with $U$ an affine open subset of $\X$ endowed with a local coordinate, then
\[ \Dkqp(V) = \left\{ \sum_{n \in \N} a_n \cdot (\varpi^k \partial)^n : ~~ a_n \in \O_{\X',\Q}(V), ~~ |a_n| \underset{ n \to \infty}{\longrightarrow} 0 \right\} .\]
We will introduce in this section microlocalization sheaves $\Firp$ for $\Dip$ as soon as $r > \kp$ using the ones of $\Di$.

\begin{example}\label{example2.1}
Assume that $\X = \Spf(\V\langle x \rangle)$ and that $\phi : \X' \to \X$ is the admissible blow-up of $\X$ defined by the ideal $I = (x , \varpi)$ of $\V\langle x \rangle$.
A standard open covering of $\X'$ is then given by $U_1 = \Spf \left( \V \left\langle \frac{x}{\varpi} \right\rangle \right)$ and $U_2 = \Spf \left( \V \left\langle \frac{\varpi}{x} , x \right\rangle \right)$.
In this example, $\kp = 1$. One has $\partial\left(\frac{x}{\varpi}\right) = \frac{1}{\varpi} \notin \V\langle x \rangle$ but $\varpi^k \cdot \partial\left(\frac{x}{\varpi}\right) = \varpi^{k-1} \in \V\langle x \rangle$ for any congruence level $k \geq 1$.
Thus, $\Dkqp(U_1)$ is a $K$-algebra as soon as $k \geq 1$. The same occurs for $\Dkqp(U_2)$.
When $k = 1$, $[ \varpi \partial , \frac{x}{\varpi} ]= \partial(x) = 1$. In particular, $\widehat{\mathcal{D}}^{(0)}_{\X', 1, \Q}(U_1)$ is not quasi-abelian but $\widehat{\mathcal{D}}^{(0)}_{\X', k, \Q}(U_1)$ is for any congruence level $k \geq 2$. Indeed,
\[ \left| \left[ \varpi^k \partial , \frac{x}{\varpi} \right] \right| = | \varpi^{k-1} | = p^{-(k-1)} < |\varpi^k \partial | \cdot \left|\frac{x}{\varpi}\right| = 1. \]
Thus, we can microlocalize the sheaf $\Dkqp$ of differential operators only for a congruence level $k \geq 2$.
More generally, if $\X'$ is any admissible blow-up of $\X$, then $\Dkqp$ will be quasi-abelian for $k > \kp$.
\end{example}

\subsection{Microlocalization of $\Dkqp$}

Let fix an ideal sheaf $\I$ of $\O_\X$ such that $\X'$ is the blowing-up of $\I$ on $\X$ and such that $\varpi^{\kp} \in \I$.
We consider a system $\{f_1 , \dots, f_n\}$ of generators of $\I(U)$, where $U = \Spf (A)$ is some affine open subset of $\X$.
Note $A_i := A \langle t_1 , \dots t_n \rangle / I_i$ and $\bar{A}_i := A_i/ (\varpi-torsion)$ with $I_i$ the ideal of $A$ generated by the elements $f_j - f_i \cdot t_j$ for all $j \in \{1, \dots , n\}$.
Then, an open covering of $\phi^{-1}(U)$ is given by the affine subsets $U_i := \Spf (\bar{A}_i) $ for $i \in \{ 1, \dots , n\}$.
We have
\[ \Dkqp(U_i) = \left\{ \sum_{n \in \N} a_\alpha \cdot (\varpi^k \partial)^n : ~~ a_n \in A_i\otimes_\V K, ~~ |a_n| \underset{ n \to \infty}{\longrightarrow} 0 \right\} .\]
Recall that $A$ is an integral domain because the formal curve $\X$ is both smooth and connected. Then $A_i \otimes_\V K$ is also an integral domain.
We consider on $A_i \otimes_\V K$ the spectral norm, denoted again by $| \cdot |$, which is an ultrametrique absolute value.
We define a multiplicative and complete norm $|\cdot|_k$ on $\Dkqp(U_i)$ by $|P|_k := \max_{n \in \N}\{ |a_n|\}$ for any $P = \sum_{n \in \N} a_n \cdot (\varpi^k \partial)^n \in \Dkqp(U_i)$.
It follows from the relation $f_i \cdot t_j = f_j$ that $\partial(f_i \cdot t_j) = \partial(f_i)\cdot t_j + f_i \cdot \partial(t_j)$ and
\[ \partial(t_j) = \frac{1}{f_i}\cdot \underbrace{(\partial(f_j)-\partial(f_i) t_j)}_{b_j \in A_i} . \]
For a congruence level $k > \kp$, one has $\varpi^k \partial(t_j) = \varpi^{k-\kp} \cdot \frac{\varpi^{\kp}}{f_i} \cdot b_j$.
Since $\varpi^{\kp} \in A$, we get $\frac{\varpi^{\kp}}{f_i} b_j \in A_i$. As a consequence, $|\varpi^k \partial(t_j) | \leq |\varpi^{k-\kp}| = p^{k-\kp} < 1$ and $| [\varpi^k \partial , t_j| = p^{k-\kp} < |\varpi^k \partial | \cdot |t_j| = 1$.
The same proof as \cite[proposition 3.5]{hallopeau3} then implies that the $K$-algebra $(\Dkqp(U_i), |\cdot|_k)$ is quasi-abelian with optimal constant $p^{k-\kp}$.
More generally, for a given congruence level $k > \kp$ and for any integer $r \in \{ \kp + 1 , \dots , k\}$, the $K$-algebra $(\Dkqp(U_i), |\cdot|_r)$ is quasi-abelian.
Nevertheless, the $K$-algebra $\widehat{\mathcal{D}}^{(0)}_{\X', \kp, \Q}(U_i)$ is not necessary quasi-abelian, see example \ref{example2.1}.

\begin{definition}
For integers $k \geq r \geq \kp + 1$, we define the microlocalization $K$-algebra
\[ (\Fkrp(U_i) , \| \cdot \|_{k, r}) := \Dkqp(U_i) \left\langle |\cdot|_r , \dots ,  |\cdot|_k ; \{ \partial^n, ~ n \in \N \} \right\rangle .\]
\end{definition}

By construction, $(\Fkrp(U_i) , \| \cdot \|_{k, r})$ is a Banach $K$-algebra admitting $\Dkqp(U_i)$ as sub-algebra.
Moreover, one can check that the norm $\| \cdot \|_{k, r}$ of $\Fkrp(U_i)$ extends the norm $|\cdot|_k$ of $\Dkqp(U_i)$.
A proof similar to the one of \cite[proposition 3.16]{hallopeau3} provides the following result.

\begin{prop}\label{prop2.3}
Any element $S$ of $\Fkrp(U_i)$ can be uniquely written in the form
\[ S = \sum_{n \in \N} a_n \cdot (\varpi^k \partial)^k + \sum_{n \in \N^*} a_{-n} \cdot (\varpi^r \partial)^{-n} \]
with $a_n \in A_i \otimes_\V K$ such that $a_n \to 0$ when $n \to \pm \infty$. Moreover, $\| S \|_{k , r} = \max_{n \in \Z} \{ |a_n| \}$.
\end{prop}

To resume, we have defined the microlocalization algebras $\Fkrp(U_i)$ for any standard open covering $\{ U_i \}$ of $\phi^{-1}(U)$ (for a fix ideal $\I$ inducing the blow-up $\phi : \X' \to \X$).
Let us denote by $\mathcal{U}$ the set of open affine subsets of $\X$ endowed with local coordinates and by $\mathcal{U}'$ the set of open subsets containing in the inverse images of elements of $\mathcal{U}$ and coming from standard open covering of $\X'$ with respect to $\I$.
Since $\mathcal{U}'$ is a basis of open subsets of $\X'$, we get a sheaf $\Fkrp$ on $\X'$ whose sections on open subsets $V$ in $\mathcal{U}'$ are exactly the microlocalization Banach $K$-algebras $\Fkrp(V)$.
Let $V$ be an open subset in $\mathcal{U}'$ and $k \geq r > \kp$.
Thanks to \cite[proposition 3.14]{hallopeau3}, we know that an element $P = \sum_{n \in \Z} a_n \cdot (\varpi^k \partial)^n$ of $\mathcal{F}_{\X' , k , k}(V)$ is invertible  if and only if $P$ has a unique coefficient $a_q$ of maximal spectral norm and if this coefficient is invertible in $\O_{\X' , \Q}(V)$.
In this case, we have
\[ P^{-1} = (\varpi^k\partial)^{-q} \cdot  \sum_{\ell \in \N} \left(-\sum_{n \in \Z\backslash \{0\}} \frac{a_{n+q}}{a_q} \cdot (\varpi^k \partial)^n \right)^\ell  \cdot a_q^{-1} .\]
We are now interested in invertible elements of the sub-algebra $\Fkrp(V)$ of $\mathcal{F}_{\X' , k , k}(V)$.
Knowing these elements will be useful to prove the equivalence of sub-holonomicity between the formal curves $\X'$ and $\X$.
Let $P = \sum_{n \geq 0} a_n \cdot (\varpi^k \partial)^n + \sum_{n < 0} a_n \cdot (\varpi^r \partial)^n$ be invertible in $\Fkrp(V) \subset \mathcal{F}_{\X' , k , k}(V)$ and let us note $\alpha_n = a_n$ for $n \geq 0$ and $\alpha_n = \varpi^{n(r-k)}\cdot a_n$ for $n < 0$.
Then $P = \sum_{n \in \Z} \alpha_n \cdot (\varpi^k \partial)^n$.
When $n < 0$, we have $|\alpha_n| = p^{n(k-r)} \cdot |a_n| \leq |a_n| \underset{n\to -\infty}{\longrightarrow} 0$.
We deduce from the fact $P$ is inversible in $\mathcal{F}_{\X' , k , k}(V)$ that $P$ has a unique coefficient $\alpha_q$ of maximal spectral norm and that its inverse $P^{-1} \in \mathcal{F}_{\X', k , k}(V)$ is given by the previous formula, replacing the coefficients $a_n$ by the $\alpha_n$.
At the end, we have to check when $P^{-1}$ defines an element of $\Fkrp(V)$, ie that the series $\sum_{\ell \in \N} \left(-\sum_{n \in \Z\backslash \{0\}} \frac{\alpha_{n+q}}{\alpha_q} \cdot (\varpi^k \partial)^n \right)^\ell$ converges in the Banach $K$-algebra $(\Fkrp(V), \| \cdot \|_{k , r})$.
Since the norm $\| \cdot \|_{k , r}$ is sub-multiplicative, a sufficient condition is that $\| Q \|_{k , r} < 1$ where $Q = \sum_{n \in \Z\backslash \{0\}} \frac{\alpha_{n+q}}{\alpha_q} \cdot (\varpi^k \partial)^n$.
Recall that the norm $\| \cdot \|_{k , r}$ is multiplicative on elements which contain only positive powers of the derivation $\partial$ or only negative powers of $\partial$.
Then we prove exactly as in \cite[proposition 3.21]{hallopeau3} that the converse holds.
In other words, the series $\sum_{\ell \in \N} \left(-Q \right)^\ell$ converges in $\Fkrp(V)$ if and only if $\| Q\|_{k , r} < 1$.
Let us simply point out the fact that this is not totally obvious because the norm $\| \cdot \|_{k, r}$ is not multiplicative as soon as $k > r$.
Using proposition \ref{prop2.3} expressing the norm in terms of coefficients, we obtain the following proposition which is just a simple generalisation of \cite[proposition 3.21]{hallopeau3}.

\begin{prop}\label{lemma2.4}
Let $V \in \mathcal{U}'$, $P = \sum_{n \geq 0} a_n \cdot (\varpi^k \partial)^n + \sum_{n < 0} a_n \cdot (\varpi^r \partial)^n \in \Fkrp(V)$ and $W \subset V$ an open subset. We note
\[ \alpha_n := 
\begin{cases}
a_n \hspace{1.85cm} \text{if} ~~ n \geq 0 \\
\varpi^{n(r-k)}\cdot a_n ~~ \text{if} ~~ n < 0 .
\end{cases}\]
Then $P \in \Fkrp(W) ^\times$ if and only if
\begin{enumerate}
\item
there is a unique element $\alpha_n$ of maximal spectral norm, said $\alpha_q$ ;
\item
$\alpha_q \in \O_{\X' , \Q}(W)^\times$ ;
\item
$|\alpha_{q-n}| \cdot p^{n(k-r)} < |\alpha_q|$ for all $n > 0$.
\end{enumerate}
\end{prop}

\begin{remark}
When $P = \sum_{n \geq 0} a_n \cdot (\varpi^k \partial)^n \in \Dkqp(V)$, the third condition simply says that $|a_n| \cdot p^{(k-r)(q-n)} < |a_q|$ for all $n \in \{ 0 , \dots , q-1 \}$.
For instance, this condition is satisfied when $|a_n| \leq |a_q|$.
\end{remark}

The microlocalization sheaves $\Fkrp$ are defined locally by localization of the non-commutative Banach $K$-algebras $\Dkqp(V)$ for $V\in \mathcal{U}'$.
We can also recover the sheaf $\Fkrp$ as the pull-back of $\Fkr$ via the admissible blow-up $\phi : \X' \to \X$.
In particular, the preceding definition of $\Fkrp$ does not depend on the chosen ideal $\I$ defining the blow-up $\X' \to \X$.
Indeed, let $i$ be a non-negative integer and $X_i := \X \times_\V \Spec(\V / \varpi^{i + 1})$ which is a $(\V / \varpi^{i + 1})$-scheme.
Then $\F_{X_i , k , r} := \Fkr^\circ \otimes_\V (\V / \varpi^{i + 1})$ is a sheaf of $(\V / \varpi^{i + 1})$-algebras over the classical scheme $X_i$.
We denote by $\mathcal{U}'_i$ the set of open affine subsets of $X_i' := \X' \times_\V \Spec(\V / \varpi^{i + 1})$ whose image by $\phi$ lies in some affine open subset of $X_i$.
It defines a basis of open subsets for the scheme $X_i'$.
Thanks to the proposition \ref{prop2.3}, one can easily observe that the germs of $\F_{X_i' , k , r} := \Fkrp^\circ \otimes_\V (\V / \varpi^{i + 1})$ and of $\phi^*\F_{X_i , k , r}$ coincide.
Thus, we get an isomorphism $\F_{X_i' , k , r} \simeq \phi^*\F_{X_i , k , r}$ of sheaves of $(\V / \varpi^{i + 1})$-algebras.
We denote by $\F_{\X', k, r}^\circ$ the sub-sheaf of $\Fkrp$ consisting of elements whose norm $\| \cdot \|_{k , r}$ is less or equal to one.
Then $\F_{\X', k, r}^\circ(U)$ is a complete $\V$-algebra for the $\varpi$-adic topology such that $\Fkrp = \F_{\X', k, r}^\circ \otimes_\V K$.
We define $\phi^! \Fkr$ as the projective limit over $i \in \N$ of the sheaves $\phi^*\F_{X_i , k , r} \simeq \F_{X_i' , k , r}$.
Then, as an $\O_{\X' , \Q}$-module, the sections of the sheaf $\phi^! \Fkr$ on any open affine subset $V$ of $\mathcal{U}'$ are given by the ones of $\Fkrp(V)$:
\[ \phi^! \Fkr(V) = \left\{ S = \sum_{n \in \N} a_n \cdot (\varpi^k \partial)^k + \sum_{n \in \N^*} a_{-n} \cdot (\varpi^r \partial)^{-n}, ~~ a_n \in \O_{\X' , \Q}(V), ~~ |a_n| \underset{ n \to \pm \infty}{\longrightarrow} 0 \right\} .\]
We equip this module with the norm $\| S \|_{k , r} = \max_{\alpha\in \Z^d} \{ |a_\alpha| \}$.
The topology given by this norm is induced by the $\varpi$-adic topology.
For any integers $k \geq r \geq \kp$ and $n \in \N$, we dispose of the natural action of the derivations $(\varpi^k \partial)^n$ on elements of $\O_{\X' , \Q}(V)$ given by the product law in the $K$-algebra $\Dkqp(V)$.
As a consequence, $\phi^! \Fkr$ is naturally a $\Dkqp$-module with $\partial^n \cdot \partial^m = \partial^{n+m}$ for any pair $(n , m)\in \Z^2$.
It remains to equip $\phi^! \Fkr(V)$ with an action of the derivations $\partial^{-1}$ on elements of $\O_{\X' , \Q}(V)$.
The natural candidate for this action is $\partial^{-1} f := \sum_{n=0}^{+\infty} (-1)^n \partial^n(f) \cdot \partial^{-(n+1)}$.
We have to check that this series converges inside $\phi^! \Fkr(V)$. One has $\partial^{-1} f = \varpi^r \cdot \sum_{n=0}^{+\infty} (-1)^n (\varpi^r\partial)^n(f) \cdot (\varpi^r\partial)^{-(n+1)}$.
In other words, the sequence $((\varpi^r\partial)^n(f))_n$ has to converge toward zero which is the case since $|(\varpi^r\partial)^n(f)| = p^{-r n} \cdot |\partial^n(f)| \leq p^{-r n} \cdot |\partial(f)|$.
Finally, the fact that $\Dkqp(V)$ is quasi-abelian when $k \geq r > \kp$ implies that the product of an operator with negative powers of $\partial$ with another operator composed of non-negative powers of $\partial$ converges for the norm defined before.
To conclude, this endow $\phi^! \Fkr(V)$ with a structure of $K$-algebra coinciding with the one of $\Fkrp$ when $k \geq r > \kp$.
Thus, we recover the microlocalization sheaf $\Fkrp$ as $\phi^! \Fkr$ for any integers $k \geq r > \kp$.
We now identitfy the sheaves $\Fkrp$ and $\phi^! \Fkr$. We deduce from this identification the following proposition.

\begin{prop}\label{prop2.5}
Let $\phi : \X' \to \X$ be an admissible blow-up of $\X$ and $k \geq r > \kp$.
We have $R^j \phi_*\Fkrp = 0$ for all $j \geq 1$.
As a consequence, there are natural sheaves isomorphism $\phi_* \Fkrp \simeq \Fkr$.
\end{prop}
\begin{proof}
The proof is similar to the one of \cite[lemma 2.3.7]{huyghe}, replacing $\Dkq$ by $\Fkr$ and $\Dkqp$ by $\Fkrp$.
Indeed, the proof is based on the fact that over $X_i' = \X' \times_\V \Spec(\V / \varpi^{i + 1})$, the sheaf $\D_{X'_i , k}^{(0)} := \widehat{\D}^{(0)}_{\X' , k} \otimes_\V (\V /\varpi^{i+1})$ is locally a free $\O_{X'_i}$-module.
The same is true for $\F_{X_i' , k , r} = \F_{\X' , k , r}^\circ \otimes_\V (\V / \varpi^{i + 1})$.
\end{proof}

For any admissible blow-up $\phi : \X' \to \X$ of $\X$, we now identify the sheaves $\phi_* \Fkrp$ and $\Fkr$ via this natural isomorphism.
We recall that $\mathcal{U}'_i$ is the set of open affine subsets of $X_i'$ whose image by $\phi$ lies in some affine open subset of $X_i$.

\begin{prop}\label{prop2.4}
Let $\phi : \X' \to \X$ be an admissible blow-up and let $k \geq r > \kp$ be integers. 
\begin{enumerate}
\item
For all $i \in \N$, the sheaves $\F_{X_i' , k , r}$ are coherent and for any affine open subset $U_i$ of $X_i'$, the algebras $\F_{X_i' , k , r}(U_i)$ are two-sided Noetherian.
\item
For any open subset $V$ of $\mathcal{U}'$, the $K$-algebra $\Fkrp(V)$ is left and right Noetherian.
Moreover, the sheaves $\Fkrp$ are coherent.
\item
The maps $\F_{\X' , k+1 , r} \hookrightarrow \Fkrp$ and $\Dkqp \hookrightarrow \Fkrp$ are left and right flat.
\end{enumerate}
\end{prop}
\begin{proof}~
\begin{enumerate}
\item
We demonstrate as in \cite[proposition 2.2.2]{huyghe} that the sheaves $\F_{X_i' , k , r}$ are coherent.
For an affine open subset $U_i$ in $\mathcal{U}'_i$, we check by using the filtration induced by the order in $\partial$ that the $(\V / \varpi^{i +1})$-algebra $\F_{X_i' , k , r}(U_i)$ is Noetherian.
The sheaf $\F_{X'_i , k , r}$ can be written as an inductive limit of coherent sheaves filtered by the order of the derivation $\partial$.
We deduce that for any open subsets $U'_i \subset U_i$ in $\mathcal{U}'_i$, we have
\[ \F_{X_i' , k , r}(U'_i) \simeq \O_{X_i'}(U'_i) \otimes_{\O_{X_i'}(U_i)} \F_{X_i' , k , r}(U_i). \]
In particular, $\F_{X_i' , k , r}(U'_i)$ is flat over $\F_{X_i' , k , r}(U_i)$.
The coherence of $\F_{X_i' , k , r}$ then comes from \cite[proposition 3.1.3]{berthelot1}.
Finally, we prove by using the same arguments as in \cite[proposition 2.2.2]{huyghe}, covering $U_i$ by a finite union of open subsets in $\mathcal{U}'_i$, that the algebra $\F_{X_i' , k , r}(U_i)$ is two-sided Noetherian.
\item
The first statement is the analogue for the admissible blow-up $\X'$ of \cite[proposition 3.20]{hallopeau3}.
The proof is exactly the same.
Indeed, on any affine open subset $V$ of $\mathcal{U}'$, the description of $\Fkrp(V)$ is similar to the description of $\Fkr(U)$.
Then we use the filtration given by the complete sub-multiplicative norm $\| \cdot \|_{k , r}$ and we easily check that the associated graded ring is Noetherian.
Finally, since $\F_{X_i , k , r}$ is a coherent sheaf which is also a quasi-coherent $\O_{X_i}$-module for any $i$, it follows from \cite[section 3.3.3]{berthelot1} that $\Fkrp$ is coherent.

\item
Let us fix $k \geq r > \kp$.
We obtain as in \cite[proposition 3.20]{hallopeau3} that the sheaves homomorphism $\F_{\X' , k+1 , r} \to \Fkrp$ is left and right flat.
It remains to prove that the homomorphism $\Dkqp \to \Fkrp$ is left and right flat.
Following the proof of \cite[lemma 1.3.4]{adriano}, it suffices to show that $\mathrm{Tor}^1_{\Dkqp}(\Fkrp , \bullet) = 0$.
But $\F_{\X' , r , r}$ is flat over $\Dkqp$ thanks to \cite[proposition 3.10]{hallopeau3}. Indeed, the idea is the following.
If $\F_{\X' , r , r}^\circ$ is the subsheaf of $\F_{\X' , r , r}$ consisting of operators with coefficients in $\O_{\X'}$, then $\F_{\X' , r , r} = \F_{\X' , r , r}^\circ \otimes_\V K$.
Moreover, $\F_{\X' , r , r}^\circ \otimes_\V \kappa$ is the classical localization of the sheaf of commutative $\kappa$-algebras $\D_{X , r}^{(0)} = \widehat{\D}_{\X , r}^{(0)} \otimes_\V \kappa$ with respect to the multiplicative part of the powers of $\varpi^r\partial$.
In particular, the morphism $(\Dkp \otimes_\V \kappa) \to (\F_{\X' , r , r}^\circ \otimes_\V \kappa)$ is left and right flat.
Taking $\varpi$-adic completion preserves flatness.
As a consequence, $\mathrm{Tor}^1_{\Dkqp}(\Fkrp , \bullet) = 0$ as soon as $\mathrm{Tor}^1_{\Dkqp}(\F_{\X' , r , r} / \Fkrp , \bullet) = 0$.
Since $\F_{\X' , r , r} / \Fkrp \simeq \widehat{\D}^{(0)}_{\X' , r , \Q} / \Dkqp$ and $\Dkqp \to \widehat{\D}^{(0)}_{\X' , r , \Q}$ is flat by \cite[proposition 2.2.16]{huyghe}, we get the flatness of $\Fkrp$ over $\Dkqp$.
\end{enumerate}
\end{proof}

Thanks to proposition \ref{prop2.4}, we know that the sheaves $\Fkrp$ satisfy the hypothesis of \cite[auxiliary result 2.2.7]{huyghe} for the affine open basis $\mathcal{U'}$ of $\X'$.
As a consequence, we obtain exactly by the same arguments theorems $A$ and $B$ for the microlocalization sheaves $\Fkrp$ over open subsets of $\mathcal{U'}$.
Moreover, we also deduce that coherent $\Fkr$-modules are locally, for elements of $\mathcal{U'}$, of finite presentation.

\begin{cor}
Let $U'$ be an open subset of $\X'$ in $\mathcal{U'}$.
Then theorems A and B are satisfied for coherent $\Fkrp(U')$-modules.
Moreover, coherent $\Fkrp(U')$-modules are of finite presentation.
\end{cor}

\subsection{Microlocalization of $\Dip$}

For any $k \geq \kp$, consider the transition homomorphism $\widehat{\mathcal{D}}^{(0)}_{\X', k+1, \Q} \hookrightarrow \Dkqp$ induced by the local inclusions $\widehat{\mathcal{D}}^{(0)}_{\X', k+1, \Q}(V) \hookrightarrow \Dkqp(V)$ for all open affine subset $V$ in $\mathcal{U}'$.
Then $\Dip$ is defined as the projective limit of the sheaves $\Dkqp$ : $\Dip := \varprojlim_{k \geq \kp} \Dkqp$.
For an affine open subset $V$ in $\mathcal{U}'$,
\[ \Dip(V) = \left\{ \sum_{n \in \N} a_n \cdot \partial^n :  \,  a_n \in \O_{\X', \Q} (V)~~ \text{such that} ~~ \forall \eta>0,~  \lim_{n \to \infty} |a_n| \cdot \eta^n =0 \right\} \]
is a Fréchet-Stein $K$-algebra.
The canonical isomorphisms $\phi_* \Dkqp = \Dkq$ for $k \geq \kp$ obtained in \cite[lemma 2.3.7]{huyghe} give rise to a canonical isomorphism $\phi_* \Dip = \Di$.
Indeed, for any open subset $V$ of $\mathcal{U}'$, one has
\[ \phi_* \Dip(V)  = \Dip(\phi^{-1}(V)) = \varprojlim_{k \geq \kp} \Dkqp(\phi^{-1}(V)) = \varprojlim_{k \geq \kp} \Dkq(V) = \Di(V) . \]
Similarly, we consider for all congruence level $k \geq r > \kp$ the transition homomorphism $\mathcal{F}_{\X',k+1 , r} \hookrightarrow \Fkrp$ induced by the local inclusions over open subsets of $\mathcal{U}'$. We define
\[ \Firp := \varprojlim_{k \geq r} \Fkrp .\]
The canonical isomorphisms $\phi_* \Fkrp = \Fkr$ for all congruence level $k \geq r \geq \kp + 1$ also give rise to a canonical isomorphism $\phi_* \Firp = \Fir$.

\begin{prop}
The sections of $\Firp$ over an open subset $V$ of $\mathcal{U}'$ is the following Fréchet-Stein algebra:
\begin{align*}
\Firp(V) & =  \displaystyle \left\{ \sum_{n \in \N} a_n \cdot \partial^n+ \sum_{n \in \N^*}  a_{-n} \cdot (\varpi^r \partial)^{-n} : \right. ~~ a_n \in \O_{\X', \Q} (V) ~~ \text{such that}  \\
& \hspace{3cm} \left.  \forall \mu \in \R, ~~ \lim_{n \to \infty} |a_n| \cdot \varpi^{- \mu n} = 0 ~~ \text{and} ~~   \lim_{n \to - \infty} |a_n|  =0 \right\}.
\end{align*}
\end{prop} 
\begin{proof}
This is a consequence of proposition \ref{prop2.4} and the proof is similar to the one of \cite[proposition 4.1]{hallopeau3}.
\end{proof}

Let $k > r > \kp$. The local inclusions $\Fkrp(V) \hookrightarrow \F_{\X' , k , r+1}(V)$ for any open subset $V$ in $\mathcal{U}'$ give rise to a transition homomorphism $\Fkrp \hookrightarrow \F_{\X' , k , r+1}$ such that the following diagram commutes:
\[ \xymatrix@R=3pc @C=3pc{ \F_{\X' , k+1 , r} \ar@{^{(}->}[r]\ar@{^{(}->}[d] & \Fkrp\ar@{^{(}->}[d] \\ \F_{\X', k+1 , r+1} \ar@{^{(}->}[r] & \F_{\X' , k , r+1} .} \]
Thus, passing to the projective limit over $k$ provides a sheaves homomorphism $\Firp \hookrightarrow \F_{\X' , \infty , r+1}$ which is injective.
Since $\X$ is suppose to be quasi-compact, so is the admissible blow-up $\X'$.
Then \cite[\href{https://stacks.math.columbia.edu/tag/009F}{Tag 009F}]{stacks-project} implies that the presheaf $\Fip := \varinjlim_{r > \kp} \Firp$ is in fact a sheaf of $K$-algebras over $\X'$.
For any open subset $V$ in $\mathcal{U}'$, we have
\begin{align*}
\Fip(V) & = \left\{ \sum_{n \in \Z} a_n \cdot \partial^n : ~~a_n \in \O_{\X' , \Q}(V), ~~  \forall \mu > 0, ~ |a_n| \cdot \mu^n \underset{n \to \infty} {\longrightarrow} 0 , \right. \\
& \hspace{7cm} \left. \exists R > 1, ~~ |a_n| \cdot R^{-n} \underset{n \to -\infty} {\longrightarrow} 0 \right\}.
\end{align*}
We consider on $\Fip(V)$ the finest topology for which the embeddings $\Fir(V) \hookrightarrow \Fi(V)$ are continuous for all integers $r > \kp$. This is the locally convex topology.
The following result corresponds to \cite[theorem 4.9]{hallopeau3} for the admissible blow-up $\phi : \X' \to \X$. The proof is similar.

\begin{theorem}\label{theorem2.8}
Let $V$ be an affine open subset in $\mathcal{U}'$ and $V' \subset V$ another open subset.
A differential operator $P$ of $\Dip(V)$ is invertible in the microlocalization algebra $\Fip(V')$ if and only if
\begin{enumerate}
\item
$P = \sum_{n=0}^q a_n \cdot \partial^n$ is finite of order $q\in \N$ in $\Dip(V)$,
\item
$a_q$ is invertible in $\O_{\X' , \Q}(V')$.
\end{enumerate}
\end{theorem}

\section{Sub-holonomic $\Dip$-modules}\label{section3}

We fix an admissible blow-up $\phi : \X' \to \X$ of the smooth formal curve $\X$.
Since $\X$ is quasi-compact, the blow-up $\X'$ is also quasi-compact.
As a consequence, $\X'$ is a Noetherian topological space.
Denote by $\tau_\X$ the relative tangent sheaf of $\X$ over $\V$ and let us introduce $\tau_{\X'} := \phi^* \tau_\X$.
Thanks to \cite[lemma 2.2.1]{huyghe}, we know that $\tau_{\X'}$ is a locally free $\O_{\X'}$-module of rank $\dim (\X) = 1$.
Let $\mathrm{Sym}(\tau_{\X'})$ be the sheaf of graded symmetric algebras generated by the sheaf $\tau_{\X'}$.

\begin{definition}
We define the cotangent space $T^* \X'$ of $\X'$ as the $\varpi$-adic completion of the $\V$-scheme $\text{Spec} \left( \mathrm{Sym}(\tau_{\X'}) \right)$. Let us denote by $\pi'$ the canonical projection $T^* \X' \to \X'$.
\end{definition}

We identify $\X'$ with the zero section of of its cotangent space $T^*\X'$.
In this section, we associate to any coadmissible $\Dip$-module a characteristic variety as it was done for coadmissible $\Di$-modules in \cite{hallopeau3}.

\subsection{Microlocalization of coadmissible $\Dip$-modules}

Let $\M_\X = \varprojlim_k \Mk$ be a coadmissible $\Di$-module and $\phi : \X' \to \X$ an admissible blow-up of $\X$.
Recall that by definition, the $\Dkq$-module $\Mk$ is coherent and the $\widehat{\D}^{(0)}_{\X , k+1 , \Q}$-linear transition homomorphisms $\M_{\X , k+1} \to \Mk$ induce $\Dkq$-linear isomorphisms $\Mk \simeq \Dkq \otimes_{\widehat{\D}^{(0)}_{\X , k+1 , \Q}} \M_{\X , k+1}$.
For any congruence level $k \geq \kp$, we define as in \cite{huyghe} the coherent $\Dkqp$-module $\phi^! \Mk := \Dkqp \otimes_{\phi^{-1} \Dkq} \phi^{-1} \Mk$.
The functors $\phi_*$ and $\phi^!$ induce an equivalence of categories between coherent $\Dkqp$-modules and coherent $\Dkq$-modules for which they are quasi-inverse.
Moreover, the $\Dip$-module $\phi^! \M_\X := \varprojlim_{k \geq \kp} \phi^! \Mk$ is coadmissible.
Conversely, if $\M_{\X'}$ is a coadmissible $\Dip$-module, then $\phi_* \M_{\X'}$ is a coadmissible $\Di$-module.
Thus, we have an equivalence of categories, established in \cite[proposition 3.1.12]{huyghe}, between coadmissible $\Dip$-modules and coadmissible $\Di$-modules via these functors.
Let $\M_{\X'} = \varprojlim_{k \geq \kp} \M_{\X' , k}$ be a coadmissible $\Dip$-module.
We introduce for any integers $k \geq r > \kp$ the coherent $\Fkrp$-module $\Mkrp := \Fkrp \otimes_{\Dkqp} \M_{\X' , k}$.
Since $\F_{X_i' , k , r} \simeq \phi^*\F_{X_i , k , r}$, the sheaf $\F_{X_i' , k , r}$ can be uniquely endowed with a structure of right $\phi^{-1}\F_{X_i , k , r}$-module via this isomorphism.
Passing to the $\varpi$-adic completion, we get a structure of right $\phi^{-1} \Fkr$-module on $\Fkrp$.
In particular, the microlocalization sheaf $\Fkrp$ is also a right $\phi^{-1} \Dkq$-module.
If $\M_\X = \varprojlim_k \Mk$ is a coadmissible $\Di$-module and $\M_{\X'} = \phi^! \M_\X$, then one can verify that
\[ \Mkrp \simeq \Fkrp \otimes_{\phi^{-1} \Dkq} \phi^{-1} \Mk \simeq \Fkrp \otimes_{\phi^{-1} \Fkr} \phi^{-1} \Mkr . \]
More generally, we define an inverse image functor
\[ \phi^! \Mkr :=  \Fkrp \otimes_{\phi^{-1} \Fkr} \phi^{-1} \Mkr \]
from coherent $\Fkr$-modules to coherent $\Fkrp$-modules.
This functor is compatible with the formation of $\Mkp = \phi^! \Mk$ starting from a coherent $\Dk$-module $\Mk$.
Indeed, as we have just said, there is a natural isomorphism $\Mkrp = \Fkrp \otimes_{\Dkqp} \phi^! \Mk \simeq \phi^! \Mkr$.
We will prove in this section, using arguments of \cite{huyghe}, that the functors $\phi^!$ and $\phi_*$ induce an equivalence of categories between coherent $\Fkr$-modules and coherent $\Fkrp$-modules.
The transition sheaves homomorphisms $\widehat{\mathcal{D}}^{(0)}_{\X', k+1, \Q} \hookrightarrow \Dkqp$, $\F_{\X' , k+1 , r} \hookrightarrow \Fkrp$ and $\M_{\X' , k+1} \to \M_{\X' , k}$ induce a transition map $\M_{\X' , k+1 , r} \to \Mkrp$.
We define the coadmissible $\Firp$-module
\[ \Mrp := \varprojlim_{k \geq r} \Mkrp .\]
The functors $\phi^!$ for coherent $\Fkr$-modules extends to a functor from coadmissible $\Fir$-modules to coadmissible $\Firp$-modules: if $\Mr = \varprojlim_{k \geq r} \Mkr$ with $r > \kp$, then we define
\[ \phi^! \Mr := \varprojlim_{k \geq r} \phi^! \Mkr .\]
For any congruence level $k \geq r+1$, the homomorphisms $j_{k , r} : \Fkrp \to \F_{\X' , k , r+1}$ induce $\Fkrp$-linear homorphisms
\[ j_{k,r} \otimes \id : \Mkrp \simeq \Fkrp \otimes_{\Dkqp} \M_{\X' , k} \to \M_{\X' , k ,r+1} \simeq \F_{\X' , k,r+1} \otimes_{\Dkqp} \M_{\X' , k} \]
that commute with the transition homomorphisms over $k$.
Passing to the projective limit over $k$ provides a continuous $\Firp$-linear sheaves homomorphism $\Mrp \to \M_{\X' , \infty , r+1}$.

\begin{prop}\label{prop2.7}
Let $ \Mr = \varprojlim_k \Mkr$ be a coadmissible $\Fir$-module with $r > \kp$ and $\Mrp = \phi^! \Mr$.
Then for any congruence level $k \geq r$, there are two natural isomorphisms $\phi_* \Mkrp = \Mkr$ and $\phi_* \Mrp = \Mr$.
\end{prop}
\begin{proof}
To prove $\phi_* \Mkrp = \Mkr$, we adapt the proof of \cite[theorem 2.3.8]{huyghe} to our case.
We can assume that $\X$ is affine with a local coordinate such that $\Mkr$ is of finite presentation.
It follows that $\Mkrp =\phi^! \Mkr =  \Fkrp \otimes_{\phi^{-1} \Fkr} \phi^{-1} \Mkr$ is also of finite presentation.
As a consequence, this module admits a short exact sequence of the form $0 \to \Nn_{\X' , k, r} \to (\Fkrp)^\ell \to \Mkrp \to 0$ for some coherent $\Fkrp$-module $\Nn_{\X' , k,r}$.
We deduce from this that $R^j \phi_* \Mkrp = 0$ for $j > 0$.
Indeed, since $\X'$ has dimension one, the assertion is true for $j > 1$ by proposition \ref{prop2.5}.
Then, the short exact sequence above implies that $R^1 \phi_* \Mkrp = R^2 \phi_* \Nn_{\X' , k , r} = 0$.
In other words, the functor $\phi_*$ is exact for any $\Fkrp$-module coming from a coherent $\Fkr$-module.
We now consider the map
\[ \phi^{-1} \Mkr \to \Mkrp = \phi^! \Mkr =  \Fkrp \otimes_{\phi^{-1} \Fkr} \phi^{-1} \Mkr\]
given by $m \to 1 \otimes m$ which is linear with respect to the adjunction map $\phi^{-1} \Fkr \to \Fkrp$ coming from proposition \ref{prop2.5}.
Thus, by adjunction we get a natural homomorphism $\varphi_{\Mkr} : \Mkr \to \phi_* \Mkrp$ which is an isomorphism when $\Mkr = \Fkr$.
We end the proof by establishing that this map is always an isomorphism. It can be checked locally.
Let $(\Fkr)^\ell \to (\Fkr)^m \to \Mkr \to 0$ be a finite presentation of the coherent $\Fkr$-module $\Mkr$.
This provides a presentation $ (\Fkrp)^\ell \to (\Fkrp)^m \to \Mkrp \to 0$ for the module $\Mkrp$.
Since the functor $\phi_*$ is exact, proposition \ref{prop2.5} implies that the module $\phi_* \Mkrp$ admits the finite presentation $(\Fkr)^\ell \to (\Fkr)^m \to \phi_* \Mkrp \to 0$.
The map $\varphi_{\Mkr}$ induces the following commutative diagram:
\[ \xymatrix{
(\Fkr)^\ell \ar[r] \ar[d]_{\simeq} & (\Fkr)^m \ar[r] \ar[d]_{\simeq} & \Mkr \ar[r] \ar[d]_{\varphi_{\Mkr}}&  0 \\
(\Fkr)^\ell \ar[r] & (\Fkr)^m \ar[r] & \phi_* \Mkrp \ar[r] & 0 .
} \]
It follows that $\varphi_{\Mkr} : \Mkr \to \phi_* \Mkrp$ is an isomorphism.
Finally, for any open subset $V$ of $\mathcal{U}'$, one has
\[ \phi_* \Mrp(V)  = \Mrp(\phi^{-1}(V)) = \varprojlim_{k \geq r} \Mkrp(\phi^{-1}(V)) = \varprojlim_{k \geq r} \Mkr(V) = \Mr(V) . \]
Since $\mathcal{U}'$ is a basis of open subsets of $\X'$, we also get $\phi_* \Mrp = \Mr$.
\end{proof}

In fact, we have proved that $\phi_*$ induces an exact functor from the category of coherent $\Fkrp$-modules to the category of coherent $\Fkr$-modules.
We have the more precise following result.

\begin{cor}\label{cor3.2}
The exact functors $\phi_*$ and $\phi^!$ induce a quasi-inverse equivalence between the category of coherent $\Fkr$-modules and the category of coherent $\Fkrp$-modules.
\end{cor}
\begin{proof}
Thanks to what we just said before and proposition \ref{prop2.7}, we only have to prove that the natural map $\phi^!\phi_* \Mkrp \to \Mkrp$ induced by $P \otimes m \to P \cdot m$ is an isomorphism.
We can assume that $\X$ is affine. Since $\phi_*$ is exact and $\Mkrp$ is coherent, we can reduce to the case of $\Mkrp = \Fkrp$.
This case is clear by definition of $\phi^!$.
Thus, the functors $\phi_*$ and $\phi^!$ are quasi-inverse equivalence between the categories of the corollary.
Since the functor $\phi_*$ is exact, so is $\phi^!$.
\end{proof}

In particular, this corollary implies that the sheaves $\Fkr$ and $\Fkrp$ have the same global sections in the sense that there is a natural isomorphism  $\Fkr(\X) \simeq \Fkrp(\X')$ of $K$-algebras.
In particular, these sheaves have the same global invertible elements.
This equivalence of categories will be useful to prove the equivalence between sub-holonomicity over $\X$ and over the admissible blow-up $\X'$.

\subsection{Weakly holonomic $\Dip$-modules}\label{section3.2}

Let $\phi : \X' \to \X$ be an admissible blow-up and $\Mp = \varprojlim_k \Mkp$ a coadmissible $\Dip$-module.
One can check, exactly as in \cite[section 2.3]{hallopeau3}, that for any integer $d \in \N$, the right $\Dip$-module $\Ext^d_{\Dip}(\Mp , \Dip) := \varprojlim_k \Ext^d_{\Dkqp}(\Mkp , \Dkqp)$ is coadmissible. 

\begin{definition}
A coadmissible $\Dip$-modules $\Mp$ is said to be weakly holonomic in the sense of Ardakov-Bode-Wadsley if for all $d \neq \dim\X = 1$, $\Ext^d_{\Dip}(\M , \Dip) = 0$.
\end{definition}

We recall that the exact functor $\phi_* : \mathcal{C}oh(\Dkqp) \overset{\sim}{\longrightarrow} \mathcal{C}oh(\Dkq)$ between left coherent modules induces an equivalence of category by \cite[theorem 2.3.8]{huyghe}.
An analogue result holds for right coherent modules.
As a consequence, this functor gives rise to a bijective map :
\begin{align*}
 \mathrm{Hom}_{\Dkqp}(\Mkp , \Dkqp) & \simeq  \mathrm{Hom}_{\phi_*  \Dkqp}(\phi_*  \Mkp , \phi_*  \Dkqp) \\
& \simeq \mathrm{Hom}_{\Dkq}(\Mk , \Dkq) \, .
\end{align*}
It follows that $\phi_* \left( \mathcal{H}om_{\Dkqp}(\Mkp , \Dkqp) \right) \simeq \mathcal{H}om_{\Dkq}(\Mk , \Dkq)$.
Indeed, this is enough to check it on global sections when $\X$ is affine ; this case comes from the previous isomorphism thanks to theorem A for coherent $\Dkqp$-modules and coherent $\Dkq$-modules.
Moreover, because the functor $\phi_*$ is exact for coherent modules and thanks to theorem B, this implies that for all integer $d \in \N$, there is a natural isomorphism $\phi_* \left( \Ext^d_{\Dkqp}(\Mkp , \Dkqp) \right) \simeq \Ext^d_{\Dkq}(\Mk , \Dkq)$.
Then, we obtain as before that
\begin{align*}
\phi_* \left( \Ext^d_{\Dip}(\Mp , \Dip) \right) & \simeq \varprojlim_k \phi_* \left( \Ext^d_{\Dkqp}(\Mkp , \Dkqp) \right) \\
& \simeq  \varprojlim_k \Ext^d_{\Dkq}(\Mk , \Dkq) \\
& \simeq \Ext^d_{\Di}(\M_\X , \Di) \, .
\end{align*}
To resume, we have proved the following.

\begin{prop}\label{propwh}
The equivalence of categories $\phi_* : \mathcal{C}_{\X'} \overset{\sim}{\longrightarrow} \mathcal{C}_\X$ given in \cite[proposition 3.1.12]{huyghe} induces an equivalence between weakly holonomic coadmissible modules over $\X'$ and $\X$.
In particular, the category of weakly holonomic $\Dip$-modules is abelian thanks to \cite[corollary 2.14]{hallopeau3}.
\end{prop}

Recall that we define the tangent sheaf of $\X'$ by $\tau_{\X'} = \phi^* \tau_\X$ , where $\tau_\X$ is the relative tangent sheaf of $\X$ over $\V$.
This is a locally free $\O_{\X'}$-module of rank one.
We consider its dual $\omega_{\X'}$ and let us note $\omega_{\X' , \Q} := \omega_{\X'} \otimes_\V K$.
We deduce from \cite[proposition 2.1.1, part I]{virrion} that the functor $\bullet  \otimes_{\O_{\X , \Q}} \omega_{\X' , \Q}^{-1}$ induces an equivalence between the category of right coherent $\Dkqp$-modules and the category of left coherent $\Dkqp$-modules.
The dual of a coadmissible $\Dip$-module $\Mp  = \varprojlim_k \Mkp$ is defined by the left coadmissible $\Dip$-module
\[ \Mp^\vee := \varprojlim_k \left( \Ext^1_{\Dkqp}(\Mkp , \Dkqp) \otimes_{\O_{\X' , \Q}} \omega_{\X' , \Q}^{-1} \right) .\]
This definition makes sense thanks to \cite[lemma 2.2, part I]{virrion} proving that this module is really coadmissible.
Next result is a direct consequence of proposition \ref{propwh} and of \cite[proposition 2.17]{hallopeau3}.

\begin{prop}\label{propdual}
Let $\Mp  = \varprojlim_k \Mkp$ be a weakly holonomic coadmissible $\Dip$-module.
Then its dual $\Mp^\vee$ is weakly holonomic and we have an isomorphism $(\Mp^\vee)^\vee \simeq \Mp$.
\end{prop}

To end this section, we assume for a while that $\X$ is affine with an \'etale coordinate.
Let $\M_k$ be a coherent $\Dkq$-module and $V \subset U$ be open subsets of $\X$.
The isomorphism $\M_k(V) \simeq \Dkq(V) \otimes_{\Dkq(U)} \M_k(U)$ established in \cite{huyghe} induces a linear isomorphism:
\[ \text{Hom}_{\Dkq(V)}(\M_k(V) , \Dkq(V)) \simeq \Dkq(V) \otimes_{\Dkq(U)} \text{Hom}_{\Dkq(U)}(\M_k(U) , \Dkq(U)).\]
As a consequence, one can check exactly as in the article \cite{ABW2} that the presheaves $U \mapsto \text{Hom}_{\Dkq(U)}(\M_k(U) , \Dkq(U))$ and $U \mapsto \text{Ext}_{\Dkq(U)}(\M(U) , \Dkq(U))$ define sheaves of right coherent $\Dkq$-module over the formal curve $\X$.
Let us denote this sheaf by $\text{Ext}^d_{\Dkq}(\Mk , \Dkq)$.
Moreover, it follows from theorem A and B for coherent $\Dkq$-modules that this sheaf coincides with $\Ext^d_{\Dkq}(\Mk , \Dkq)$. In other words, we have $\text{Ext}^d_{\Dkq}(\Mk , \Dkq) \simeq \Ext^d_{\Dkq}(\Mk , \Dkq)$ as right $\Dkq$-modules.
Similarly, the same works for coherent $\Dkqp$-modules.
Thus, we deduce that for any coadmissible $\Dip$-module $\Mp \simeq \varprojlim_{k \geq \kp} \Mkp$, there is a natural isomorphism
\begin{equation}\label{eqext}
\Ext^d_{\Dip}(\Mp , \Dip) \simeq \varprojlim_{k \geq \kp} \text{Ext}^d_{\Dkqp}(\Mkp , \Dkqp) .
\end{equation}

\subsection{Sub-holonomic $\Dip$-modules}\label{section3.3}

We always consider an admissible blow-up $\phi : \X' \to \X$ of $\X$ and let $\M_{\X'} = \varprojlim_{k \geq \kp} \M_{\X' , k}$ be a coadmissible $\Dip$-module.
We recall that for any integers $k \geq r > \kp$, the $\Fkrp$-module $\Mkrp = \Fkrp \otimes_{\Dkqp} \M_{\X' , k}$ is coherent and $\Mrp = \varprojlim_{k \geq r} \Mkrp$ is a coadmissible $\Firp$-module.
Let us set
\[ \supp (\Mrp) := \overline{ \cup_{k \geq r} \supp \Mkrp} .\]
Here the considered union is increasing.
Let $k > r \geq \kp$. We also verify, as in \cite[section 5.3]{hallopeau3}, that $\M_{\X' , k , r+1} \simeq \F_{\X' , k , r+1} \otimes_{\Fkrp} \Mkrp$ as $\F_{\X' , k , r+1}$-modules.
As a consequence, we get $\supp (\M_{\X' , k , r+1}) \subset \supp (\Mkrp)$ and $\supp (\M_{\X', \infty ,  r+1}) \subset \supp (\Mrp)$.
Let us note
\[ \suppi(\Mp) := \bigcap_{r > \kp} \supp(\Mrp) . \]
The infinite support $\suppi(\Mp)$ is a closed subset of $\X'$.
If $\dim(\suppi \Mp) = 0$, then $\suppi(\Mp)$ consists of a finite set of closed points in $\X'$.

\begin{prop}\label{prop3.4}
Let $\Mp = \varprojlim_{k \geq \kp} \Mkp$ be a coadmissible $\Dip$-module.
Then there exists a rank $r_0 > \kp$ such that for any $r \geq r_0$, $\suppi (\Mp) = \supp (\Mrp)$.
Moreover, if $\dim(\suppi \Mp) = 0$, then for all $r > r_0$, there exists a congruence level $k_r \geq r$ such that for any $k \geq k_r$, $\supp(\Mrp) = \supp(\Mkrp)$.
\end{prop}
\begin{proof}
Since $\X'$ is a Noetherian topological space, the decreasing sequence of support $(\supp (\Mrp))_{\X' \in \G'}$ is stationary.
Thus, $\suppi(\Mp) = \supp (\Mrp)$ for $r$ large enough.
We now assume that $\dim(\suppi(\Mp)) = 0$. Then $\suppi(\Mp) = \supp(\Mrp)$ is a finite union of points.
Recall that the closed subset $\supp(\Mrp)$ of $\X'$ is the closure of the increasing union over $k$ of the supports $\supp (\Mkrp)$.
It follows that the sequence $(\supp \Mkrp)_{k \geq r}$ is stationary because $\supp(\Mrp)$ is finite.
Then, there exists a congruence level $k \geq r$ such that $\supp (\Mkrp) = \cup_{k \geq r} \supp (\Mkrp)$.
As a consequence, the union $\cup_{k \geq r} \supp (\Mkrp)$ is closed and $\supp (\Mrp) = \overline{ \cup_{k \geq r} \supp (\Mkrp)} = \supp(\Mkrp)$.
\end{proof}

We define as in \cite[section 5.3]{hallopeau3} a characteristic variety $\Char(\Mp)$ for any coadmissible $\Dip$ module $\Mp$ as a closed subset of the cotangent space $T^*\X'$ of $\X'$.
Let just say few words about this construction. We can associate to the microlocalization sheaves $\Fkrp$ analogue sheaves $\Fkrp^*$ at the level of the cotangent space $T^* \X' := \widehat{\Spec} \left( \mathrm{Sym}(\tau_{\X'}) \right)$ of $\X'$.
We recall that $\tau_{\X'} = \phi^* \tau_\X$ is a locally free $\O_{\X' ,\Q}$-module of rank one.
Let us note $\pi' : T^*\X' \to \X'$ the canonical projection which is an open map and $s : \X' \to T^*\X'$ the zero section.
For $k \geq r > \kp$ and any open subset $V$ of $T^*\X'$, we define the sheaf $\Fkrp^*$ by
\[
\Fkrp^*(V) := 
\begin{cases}
\Fkrp( \pi'(V)) ~~~~~~ \mathrm{if}~~V \cap s(\X') = \emptyset  \\
\Dkqp(\pi'(V)) \hspace{0.62cm} \mathrm{if}~~V \cap s(\X') \neq \emptyset.
\end{cases}
\]
We also form the sheaf $\Firp^* := \varprojlim_{k > r} \Fkrp^*$ on $T^* \X'$.
We associate to any coherent $\Dkqp$-module $\Mkp$ the coherent $\Fkrp^*$-module
\[ \Mkrp^* := \Fkrp^* \otimes_{\pi^{'-1} (\Dkqp)} \pi^{'-1} (\Mkp) .\]
Finally, we associate to a coadmissible $\Dip$-modules $\Mp = \varprojlim_{k \geq \kp} \Mkp$ the coadmissible modules $\Mrp^* := \varprojlim_{k \geq r} \Mkrp^*$.
Again, one can prove that the sequence of supports $(\supp(\Mrp^*))_r$ is decreasing.

\begin{definition}
The characteristic variety of a coadmissible $\Dip$-module $\Mp$ is defined as the following closed subset of $T^*\X'$:
\[ \Char(\Mp) := \bigcap_{r > \kp} \supp(\Mrp^*) \subset T^*\X' . \] 
\end{definition}

Thanks to the theorem \ref{theorem2.8} and working on irreducible components of $T^*\X'$, we prove as in \cite[proposition 5.19]{hallopeau3} the Bernstein inequality for this characteristic variety: if $\Mp \neq 0$, then all irreducible components of $\Char(\Mp)$ have dimension at least one.
If the dimension of $\Char(\Mp)$ is one, then $\Char(\Mp)$ is composed potentially of the horizontal irreducible component $\X'$ identified with the zero section of $T^*\X'$ and of vertical irreducible components whose x-axis are exactly the elements of the infinite support $\suppi (\Mp)$ of $\Mp$.

\begin{definition}
A coadmissible $\Dip$-module $\Mp$ is said to be sub-holonomic if the dimension of its characteristic variety $\Char(\Mp)$ is less or equal to one, or equivalently if the dimension of its infinite support $\suppi (\Mp)$ is zero.
\end{definition}

\begin{example}
Theorem \ref{theorem2.8} implies that any coadmissible $\Dip$-module locally of the form $\Dip / P$, with $P$ a finite differential operator, is sub-holonomic.
\end{example}

Proposition \ref{prop3.4} and the fact that $\Fkrp$ is flat over $\Dkqp$ (proposition \ref{prop2.4}) imply the following result establishing that the category of sub-holonomic $\Dip$-modules is abelian.

\begin{prop}\label{prop3.4b}
Let $0 \to \Nn_{\X'} \to \Mp \to \L_{\X'} \to 0$ be an exact sequence of coadmissible $\Dip$-modules. Then $\suppi(\Mp) = \suppi(\Nn_{\X'}) \cup \suppi(\L_{\X'})$ and $\Char(\Mp) = \Char(\Nn_{\X'}) \cup \Char(\L_{\X'})$.
As a consequence, the module $\Mp$ is sub-holonomic if and only if $\Nn_{\X'}$ and $\L_{\X'}$ are sub-holonomic.
In particular, the category of sub-holonomic coadmissible $\Dip$-modules is abelian.
\end{prop}

We denote by $\mathcal{SH}_{\X'}$ the abelian category of sub-hononomic $\Dip$-modules.
The structural sheaf $\O_{\X' , \Q} \simeq \phi^! \O_{\X , \Q}$ of $\X'$ tensorized by $K$ is a coadmissible $\Dip$-module thanks to \cite[proposition 3.3.1]{huyghe}. 
More generally, we know by proposition \ref{propconnections} saying that any coherent $\O_{\X , \Q}$-module endowed with an integrable connection is naturally a coadmissible $\Dip$-module.
We will prove in next section that sub-holonomic $\Dip$-modules are equivalent to sub-holonomic $\Di$-modules via the exact functor $\phi_*$ (where $\phi$ always denotes the admissible blow-up $\X' \to \X$) and that $\phi(\suppi(\Mp)) = \suppi(\phi_*\Mp)$ for any coadmissible $\Dip$-module $\Mp$.
This fact implies that for any integrable connection $\Mp$ over $\X'$, $\suppi(\Mp) = \emptyset$.
Indeed, since $\phi_* \O_{\X' , \Q} \simeq \O_{\X , \Q}$, the pushforward $\phi_* \Mp$ of $\Mp$ is a coherent $\O_{\X , \Q}$-module and thus an integrable connection.
But we know by \cite[proposition 6.8]{hallopeau3} that $\suppi(\phi_* \Mp) = \emptyset$.
It follows that $\phi(\suppi(\Mp)) = \suppi(\phi_* \Mp) = \emptyset$ and that $\suppi(\Mp) = \emptyset$ by surjectivity of $\phi$.
This condition is equivalent to have $\Char(\Mp) \subset s(\X') \subset T^*\X'$.
As a consequence, any integrable connection is sub-holonomic.
More generally, we will prove in next section that a coadmissible $\Dip$-module $\Mp$ is an integral connection if and only if $\suppi (\Mp) = \emptyset$.
The following lemma says that a coherent $\Dkqp$-module $\Mkp$ is a coherent $\O_{\X' , \Q}$-module as soon as $\supp(\Fkrp \otimes_{\Dkqp} \Mkp) = \emptyset$ for some integer $r \in \{ \kp , \dots , k \}$.
This result will be useful for proving that any sub-holonomic $\Dcap$-module is generically a locally free $\O_{\X_K}$-module of finite rank.

\begin{lemma}\label{lemma3.8}
Let $\Mkp$ be a coherent $\Dkqp$-module.
It there exists an open subset $V$ of $\X'$ such that $\supp (\Fkrp\otimes_{\Dkqp} \Mkp)\cap V = \emptyset$, then $(\Mkp)_{|V}$ is a coherent $\O_{V , \Q}$-module.
\end{lemma}
\begin{proof}
We can assume that $\X' = V$ is affine endowed with a derivation coming from a local coordinate of $\X$. 
Let us suppose that the support $\supp (\Fkrp\otimes_{\Dkqp} \Mkp)$ is empty.
The $\Dkqp$-module $\Mkp$ is coherent so generated by global sections $e_1 , \dots e_\ell \in \Mkp(\X')$.
Since the coherent $\Fkrp$-module $\Fkrp\otimes_{\Dkqp} \Mkp$ is also generated by the section $e_1 , \dots , e_\ell$, we have $\supp(\Mkrp) = \cup_{i=1} ^\ell \supp(\Fkrp\cdot e_i) = \emptyset$.
Thus, we can reduce to the case $\Mkp = \Dkqp \cdot e \simeq \Dkqp / \I$
where $\I$ is a non-zero coherent ideal of $\Dkqp$.
The condition $\supp( \Fkrp / \Fkrp\cdot \I) = \emptyset$ implies that $\Fkrp\cdot \I = \Fkrp$.
We deduce from \cite[lemma 5.5]{hallopeau3} the existence of a differential operator $P = \sum_{n=0}^\infty a_n \cdot (\varpi^k \partial)^n \in \I(\X')$ admitting a unique coefficient $a_s$ of maximal spectral norm such that $a_s$ is invertible in $\O_{\X' , \Q}(\X')$.
To prove the lemma, it suffices to show that $\Dkqp /P$ is coherent over $\O_{\X', \Q}$ because we have a surjective map $\Dkqp / P \twoheadrightarrow \Mkp$.
In fact, we will prove a more precise result: $\Dkqp / P$ is a free $\O_{\X' , \Q}$-module of rank $s$.
Let $X' := \X' \otimes_\V \Spec(\kappa)$ be the special fiber of $\X'$ and $\D_{X' , k}^{(0)} := \Dkp \otimes_\V \kappa$.
We denote by $\xi_k$ the image of $\partial_k$ in $\D_{X' , k}^{(0)}(X')$.
Since $k \geq 1$, $\D_{X' , k}^{(0)}(X') = \O_{X'}(X')[\xi_k]$ is a commutative ring over $\O_{X'}(X')$.
We can assume that $|P|_k = 1$.
The coherent $\Dkp$-module $\Dkp / P$ is complete for the $\varpi$-adic topology and $\varpi$-torsion free.
Because $\Dkqp / P \simeq(\Dkp / P) \otimes_\V K$, it suffices to prove that $\overline{\Dkp / P} = \D_{X' , k}^{(0)} /  \bar{P}$ is a free $\O_{X'}$-module of rank $s$.
We have $\bar{P} = (P \mod \varpi) = \bar{a}_s \cdot \xi_k^s \in \D_{X' , k}^{(0)}(X')$ with $\bar{a}_s$ invertible in $\O_{X'}(X')$.
Since $\D_{X' , k}^{(0)}(X') = \O_{X'}(X')[\xi_k]$, it is clear that $\D_{X' , k}^{(0)} /  \bar{P} \simeq \O_{X'} \oplus \O_{X'} \cdot \xi_k \oplus \dots \oplus O_{X'} \cdot \xi_k ^{s-1}$ is free of rank $s$ over the sheaf $\O_{X'}$.
\end{proof}

Considering the smooth points of $\X'$, we get the following characterisation of sub-holonomic $\Dip$-modules.
It corresponds to \cite[propositions 6.6 and 6.8]{hallopeau3}.

\begin{lemma}\label{lemma 3.12}
Let $\Mp$ be a coadmissible $\Dip$-module and $V$ be the open subset composed of all the smooth points of $\X'$ inside $\X' \backslash \suppi(\Mp)$.
Then $(\Mp)_{V}$ is locally a free $\O_{V, \Q}$-module of finite rank.
In other words, any sub-holonomic $\Dip$-modules is generically an integrable connection.
\end{lemma}

Generically means that a sub-holonomic $\Dip$-module is locally a free $\O_{\X' , \Q}$-module of finite rank over some open dense subset of $\X'$.
We will prove the converse in next section.

\begin{definition}
Let $\Mp$ be a sub-holonomic $\Dip$-module and $W_{\Mp}$ the maximal open subset of $\X'$ such that $(\Mp)_{|W_{\Mp}}$ is locally a free $\O_{\X' , \Q}$-module of finite rank.
We define the horizontal multiplicity $m_0(\Mp) \in \N$ of $\Mp$ to be the maximal rank of the locally free $\O_{\X' , \Q}$-module $(\Mp)_{|W_{\Mp}}$.
\end{definition}

For a sub-holonomic $\Dip$-module $\Mp$, we also call $m_0(\Mp)$ the multiplicity of the horizontal component of its characteristic variety $\Char(\Mp)$.
Bernstein's inequality implies that the characteristic variety $\Char(\Mp)$ contains the horizontal component if and only if $m_0(\Mp) > 0$.
We immediately deduce from the the fact that $W_{\Mp}$ is dense in $\X'$ that the multiplicity $m_0$ is additive on short exact sequence : if $0 \to \Mp \to \Nn_{\X'} \to \L_{\X'} \to $ is an exact sequence of sub-holonomic $\Dip$-modules, then $m_0(\Nn_{\X'}) = m_0(\Mp) + m_0(\L_{\X'})$.
Since $\X'$ is not necessarily irreducible, let us emphazise that this is not yet clear that the rank of the locally free $\O_{\X', \Q}$-module $(\Mp)_{|W_{\Mp}}$ is everywhere the same.
In particular, the horizontal multiplicity of $\Mp$ is potentially not the same as the one of $\phi_* \Mp$.
However, we will prove later that these multiplicities over $\X'$ and $\X$ are really the same, ie that $m_0(\Mp) = m_0(\phi_* (\Mp))$.

\subsection{Equivalence of sub-holonomicity over $\X'$ and $\X$}\label{section3.4}

Let again $\phi : \X' \to \X$ be an admissible blow-up of $\X$, $\M_\X = \varprojlim_k \M_k$ a coadmissible $\Di$-module and $\Mp = \phi^! \M_\X = \varprojlim_k \Mkp$ the corresponding coadmissible $\Dip$-module.
We prove in this subsection the equality $\phi(\suppi(\Mp)) = \suppi(\M_\X)$ between infinite supports and the fact that $\Mp$ is sub-holonomic if and only if $\M_\X$ is.
In order to prove these results, we associate vertical multiplicities to the sub-holonomic module $\Mp$, using the ones of $\M_\X$ at the points of the support $\suppi(\M_\X)$ as defined in \cite[section 6.2]{hallopeau3}.
Let us start by two lemmas. We demonstrate in the first one that it is equivalent to ask the supports $\supp(\Mkr)$ and $\supp(\Mkrp)$ to be finite.

\begin{lemma}\label{lemma3.9}
Let $k \geq r > \kp$ and $\M_k$ be a coherent $\Dkq$-module. We note $\Mkr = \Fkr \otimes_{\Dkq} \M_k$ and $\Mkrp = \phi^! \Mkr$.
Then $\dim(\supp(\Mkrp)) = 0$ if and only if $\dim(\supp(\Mkr)) = 0$. In this case, $\phi(\supp(\Mkrp)) = \supp(\Mkr)$.
\end{lemma}
\begin{proof}
First, we assume that $\dim(\supp(\Mkrp)) = 0$, ie the support $\supp(\Mkrp)$ is finite.
Let us write $\supp ( \Mkrp) = \{ x_1 , \dots , x_s\}$.
If $i_{x_\ell}$ is the inclusion $\{ x_\ell \} \hookrightarrow \X'$, we have
\[ \Mkrp \simeq \bigoplus_{\ell = 1}^s (i_{x_\ell})_* (\Mkrp)_{x_\ell}. \]
Thanks to the proposition \ref{prop2.7}, we know that $\phi_* \Mkrp \simeq \Mkr$.
It follows that we have an isomorphism of $\Fkr$-modules
\[ \Mkr \simeq \bigoplus_{\ell = 1}^s \phi_* (i_{x_\ell})_* (\Mkrp)_{x_\ell} \simeq \bigoplus_{\ell = 1}^s  (j_{\phi(x_\ell)})_* (\Mkrp)_{x_\ell} \]
where $j_{\phi(x_\ell)} : \{ \phi(x_\ell) \} \hookrightarrow \X$ denotes the natural inclusion.
The structure of $\Fkr$-module for the right-hand term comes from the canonical homomorphisms of $K$-algebras $(\Fkr)_{\phi(x_\ell)} \simeq (\phi_* \Fkrp)_{\phi(x_\ell)} \to (\Fkrp)_{x_\ell}$ for all $\ell \in \{ 1 , \dots , s \}$.
As a consequence, we deduce that
\[ \phi(\supp(\Mkrp)) = \supp(\Mkr) = \{ \phi(x_1) , \dots , \phi(x_s) \} . \]
Let us suppose that $\dim(\supp(\Mkr)) = 0$. We can assume that $\X$ is affine with an \'etale coordinate such that $\supp(\Mkr) = \{ x \}$.
Since $\Char(\Mk) \subset \supp(\Mkr) = \{ x \}$, the coherent $\Dkq$-module $\Mk$ is holonomic.
As a consequence, there exists a non-zero coherent ideal $\I_k$ of $\Dkq$ such that $\Mk \simeq \Dkq / \I_k$.
One can check \cite{hallopeau1} for more details about holonomic $\Dkq$-modules.
It follows that $\Mkr \simeq \Fkr / \I$ with $\I = \Fkr \cdot \I_k$.
Up to reducing $\X$, the condition $\supp(\Mkr) = \{ x \}$ implies that there exists an element $P \in \I(\X)$ such that $P$ is invertible in $\Fkr(\X \backslash \{ x \})^\times$.
Otherwise, the ideal $\I$ should not be equal to $\Fkr$ outside $x$ and the support $\supp(\Mkr) $ should contain another point.
The surjective homomorphism $\Fkr / P \twoheadrightarrow \Mkr$ induces a surjective map $\Fkrp / P \twoheadrightarrow \Mkrp$, identifying $P$ with an element of $\Fkrp(\X') \simeq \Fkr(\X)$.
Thus, we are reduced to prove that the support of the coherent $\Fkrp$-module $\Fkrp / P$ is finite.
Let us write $P = \sum_{n = 0}^\infty a_n \cdot (\varpi^k \partial)^n + \sum_{n=1}^\infty a_{-n} \cdot (\varpi^r \partial)^{-n} \in \Fkr(\X)$, $W = \X \backslash \{ x \}$ and note
\[ \alpha_n = 
\begin{cases}
a_n \hspace{1.85cm} \text{if} ~~ n \geq 0 \\
\varpi^{n(r-k)}\cdot a_n ~~ \text{if} ~~ n < 0 .
\end{cases}\]
Since $P \in \Fkr(W)^\times$, we know by proposition \ref{lemma2.4} that $P$ verifies the following conditions:
\begin{enumerate}
\item[$(i)$]
there is a unique element $\alpha_n$ of maximal spectral norm, say $\alpha_q$ ;
\item[$(ii)$]
$\alpha_q \in \O_{\X , \Q}(W)^\times$, ie $a_q$ potentially admits a zero only at $x$ ;
\item[$(iii)$]
$|\alpha_{q-n}| \cdot p^{n(k-r)} < |\alpha_q|$ for all $n > 0$.
\end{enumerate}
Recall that by \cite[proposition 2.3.1]{huyghe}, there is an isomorphism between the global sections of $\O_{\X, \Q}$ and of $\O_{\X' , \Q}$.
Thus, we identify the coefficients $a_n \in \O_{\X , \Q}(\X)$ of $P$ with elements of $O_{\X' , \Q}(\X)$.
Since $\Fkrp(\X') \simeq \Fkr(\X)$ and the norm $\| \cdot \|_{k,r}$ does not depend on the choice of the chosen open subset, it follows that $P$ satisfies conditions $(i)$ and $(iii)$ in $\Fkrp(V)$ for any open subset $V$ in $\mathcal{U}'$.
Moreover, the function $a_q \in \O_{\X' , \Q}(V)$ has only a finite number numbers of zeros over $V$, see for example \cite[proposition 1.3.2]{garnier}.
It follows again by lemma \ref{lemma2.4} that for any affine open subset $V$ in $\mathcal{U}'$, $P$ is invertible in $\Fkrp(V)$ up to removing a finite number of closed points of $V$ (the zeros of $a_q$ in $V$).
In other words, the support $\supp(\Fkrp / P)\cap V$ coincides with the zeros of $a_q$ in $V$.
Finally, the finiteness of the support $\supp(\Fkrp / P)$ comes from the compactness of the admissible blow-up $\X'$.
\end{proof}

Next lemma simply says that we can deduce the relation $\phi(\suppi (\Mp)) = \suppi(\M_\X)$ from such equalities for large enough integers $k \geq r > \kp$.

\begin{lemma}\label{lemma3.10}
Let $\M_\X$ be a coadmissible $\Di$-module and $\Mp = \phi^! \M_\X$ the corresponding coadmissible $\Dip$-module.
Assume that for large enough integers $k \geq r > \kp$ we have $\phi(\supp(\Mkrp)) = \supp(\Mkr)$.
Then $ \phi(\suppi (\Mp)) = \suppi(\M_\X)$.
\end{lemma}
\begin{proof}
Write $\M_\X = \varprojlim_{k \geq 0} \Mk$ and $\Mp = \phi^! \M_\X = \varprojlim_{k \geq \kp} \Mkp$.
Thanks to the proposition \ref{prop3.4}, there exists an integer $r_0 > \kp$ such that for any $r \geq r_0$,
\[ \suppi (\Mp) = \supp (\Mrp)  ~\, \text{and} ~   \suppi (\M_\X) = \supp (\Mr) . \]
Let us recall that by definition, $\supp (\Mrp) = \overline{ \cup_{k \geq r} \supp (\Mkrp)}$ and $\supp (\Mr) = \overline{ \cup_{k \geq r} \supp (\Mkr)}$.
Moreover, the sequences $(\supp (\Mkrp))_k$ and $(\supp (\Mkr))_k$ are increasing.
Thanks to the assumption, we can find some $r \geq r_0$ and $k_r \geq r$ such that for all congruence level $k \geq k_r$, we have $\phi(\supp(\Mkrp)) = \supp(\Mkr)$.
It follows that
\[\phi\left(\bigcup_{k \geq r}\supp(\Mkrp)\right) = \bigcup_{k \geq r} \supp(\Mkr) . \]
Since the admissible blow-up $\phi : \X' \to \X$ is a closed map, we deduce that
\begin{align*}
\phi(\supp(\Mrp)) & = \phi\left(\overline{\bigcup_{k \geq r}\supp(\Mkrp)}\right) = \overline{\phi\left(\bigcup_{k \geq r}\supp(\Mkrp)\right)} \\
& = \overline{\bigcup_{k \geq r} \supp(\Mkr)}  = \supp(\Mr) .
\end{align*}
Thus, $\phi(\suppi (\Mp)) = \phi(\supp (\Mrp)) = \supp (\Mr) = \suppi (\M_\X)$.
\end{proof}

These two lemmas immediately imply in next corollary that if the coadmissible $\Dip$-module $\Mp$ is sub-holonomic, then the coadmissible $\Di$-module $\M_\X$ is also sub-holonomic.
The other direction needs a little more work.
Indeed, let us assume that $\M_\X = \varprojlim_{k\geq0} \Mk$ is sub-holonomic.
Then all the supports $\supp(\Mkr)$ are the same and finite for $k$ and $r$ large enough.
We only know by lemma \ref{lemma3.9} that the supports $\supp(\Mkrp)$ are finite.
But at this point we do not know the sizes of these supports: they can tend toward infinity, which implies that the dimension of $\suppi(\Mp)$ is at least one.
We need finite vertical multiplicities for sub-holonomic $\Dip$-modules in order to control these sizes.

\begin{cor}\label{cor3.11}
Let $\Mp$ be a coadmissible $\Dip$-module.
If $\Mp$ is sub-holonomic, then $\M_\X = \phi_* \Mp$ is sub-holonomic and $\phi(\suppi(\Mp)) = \suppi(\M_\X)$.
As a consequence, $\Mp$ is weakly holonomic as soon as $\Mp$ is sub-holonomic.
\end{cor}
\begin{proof}
We suppose that the coadmissible $\Dip$-module $\Mp$ is sub-holonomic.
In other words, the support $\suppi(\Mp)$ is finite. Let us write $\suppi(\Mp) = \{ x'_1 , \dots , x'_s \}$.
By proposition \ref{prop3.4}, we know that there exists a rank $r_0 > \kp$ such that for any $r \geq r_0$ and for $k \geq r$ large enough, $\suppi(\Mp) = \supp(\Mkrp)$.
In particular, the supports $\supp(\Mkrp)$ are finite.
Lemma \ref{lemma3.9} implies that $\phi(\supp(\Mkrp)) = \supp(\Mkr)$ and lemma \ref{lemma3.10} implies that $\suppi(\M_\X) = \phi(\suppi(\Mp))$.
Thus, the coadmissible $\Di$-module $\M_\X$ is also coadmissible.
Since we have an equivalence of categories between weakly holonomic $\Di$-modules and weakly holonomic $\Dip$-modules and any sub-holonomic $\Di$-module is weakly holonomic, it follows that $\Mp$ is also weakly holonomic.
\end{proof}

As before, we denote by $\vqp$ the commutative $K$-algebra composed of the elements of $\Fkrp(V)$ with coefficients in $K$:
\[ \vqp = \left\{ \sum_{n\in \N} \lambda_n\cdot (\varpi^k\partial)^n + \sum_{n\in \N^*} \lambda_{-n}\cdot (\varpi^r\partial)^{-n} \in \F_{k , r}(V)~: ~ \lambda_n \in K, ~ \lim_{n \to \pm \infty} \lambda_n = 0  \right\} . \]
Let $\Mp$ be a sub-holonomic coadmissible $\Dip$-module.
Then $\M_\X = \phi_* \Mp$ is sub-holonomic and we know by corollary \ref{cor3.11} that $\phi(\suppi (\Mp)) = \suppi(\M_\X)$.
We fix a closed point $x \in \suppi(\M_\X)$. Denote by $x_1' , \dots , x_s'$ the points of $\suppi (\Mp)$ above $x$, ie such that $\phi(x_i') = x$.
The proof of lemma \ref{lemma3.9} implies that for any $k \geq r \geq \kp +1$ large enough, we have an isomorphism of $(\Fkr)_x$-modules
\begin{equation}\label{eq1}
 (\Mkr)_x \simeq \bigoplus_{i=1}^s (\Mkrp)_{x'_i} .
\end{equation}
In particular, this is also an isomorphism of $\vqp$-modules.
For $k \geq r$ large enough, we have proved in \cite[section 6.2]{hallopeau3} that $(\Mkr)_x$ is a free $\vqp$-module of finite rank $m_x(\M_\X) \in \N$.
Since the commutative ring $\vqp$ is a principal ideal domain, the modules $(\Mkrp)_{x'_i}$ are also free over $\vqp$ of finite rank denoted by $rk_{x'_i}(\Mkrp)$.
Indeed, they are sub-modules of the free module $ (\Mkr)_x$.
We verify exactly as in \cite[section 6.2]{hallopeau3}  that these ranks do not depend on $k \geq r > \kp$.
The multiplicity $m_{x_i'}(\Mp)$ of $\Mp$ at $x_i'$ is then this common rank.
We deduce from isomorphism \ref{eq1} the following equality between multiplicities:
\begin{equation}\label{eq2}
m_x(\M_\X) = m_{x_1'}(\Mp) + \dots + m_{x_s'}(\Mp) .
\end{equation}
Moreover, we obtain as in \cite[section 6.2]{hallopeau3} that these vertical multiplicities are additive on short exact sequence of sub-holonomic coadmissible $\Dip$-modules.
Let us note
\[ m(\Mp) := \sum_{x' \in \suppi(\Mp)}m_{x}(\Mp) ~~ \text{and} ~~ m(\M_\X) := \sum_{x \in \suppi(\M_\X)}m_{x}(\M_\X) \in \N . \]
Relation \ref{eq2} implies that for any admissible blow-up $\phi : \X' \to \X$ of $\X$ and for any sub-holonomic $\Dip$-module $\Mp$, we have $m(\Mp) = m(\phi_*\M_\X)$.
We now prove the converse of corollary \ref{cor3.11} : if a coadmissible $\Di$-module $\M_\X$ is sub-holonomic, then the coadmissible module $\Mp =\phi^! \M_\X$ is also sub-holonomic.

\begin{theorem}\label{theorem3.14}
Let $\phi : \X' \to \X$ be an admissible blow-up. Then a coadmissible $\Di$-module $\M_\X$ is sub-holonomic if and only if the corresponding coadmissible $\Dip$-module $\Mp = \phi^! \M_\X$ is.
Moreover, the support $\suppi(\Mp)$ contains at most $m(\M_\X)$ points and we have multiplicities for the module $\Mp$ at the points of $\suppi(\Mp)$.
\end{theorem}
\begin{proof}
Thanks to corollary \ref{cor3.11}, we know that $\M_\X$ is sub-holonomic as soon as $\Mp$ is sub-holonomic. We now assume that $\M_\X$ is sub-holonomic.
There exists an integer $r_0 > \kp$ such that for any $r \geq r_0$, $\suppi(\M_\X) = \supp(\Mr)$ and $\suppi(\Mp) = \supp(\Mrp)$.
Since the coadmissible $\Di$-module $\M_\X$ is holonomic, there exists a congruence level $k_0 \geq r_0$ such that for any $k \geq k_0$, $\suppi(\M_\X) = \supp(\Mr) = \supp(\Mkr)$.
Let us write $\suppi(\M_\X) = \{ x_1 , \dots , x_s \}$. Recall that $m(\M_\X) := m_{x_1}(\M_\X) + \dots + m_{x_s}(\M_\X)$.
Thanks to lemma \ref{lemma3.9}, we know that $\supp(\Mkrp)$ is finite and $\phi(\supp(\Mkrp)) = \supp(\Mkr) = \{ x_1 , \dots , x_s \}$.
As explained in the paragraph above, we can associate to $\Mkrp$ positive ranks at the points of the support $\supp(\Mkrp)$.
Moreover, equation \ref{eq1} implies that if $\phi^{-1}(x_i) \cap \supp(\Mkrp) = \{ x'_1 , \dots x'_\ell \}$, then
\[ m_{x_i}(\M_\X) = rk_{x_i}(\Mkr) = rk_{x'_1}(\Mkrp) + \dots + rk_{x'_\ell}(\Mkrp) . \]
As a consequence, there is at most $m_{x_i}(\M_\X)$ points in the support $\supp(\Mkrp)$ above the point $x_i$.
Thus, the support $\supp(\Mkrp)$ contains at most $m(\M_\X)$ points.
Since the sequence of supports $(\supp(\Mkrp))_{k \geq k_0}$ is increasing, we deduce that this is stationary.
It follows that $\supp(\Mrp) = \supp(\Mkrp)$ for $k$ large enough.
In particular, the $\supp(\Mrp)$ is finite and the coadmissible $\Dip$-module $\Mp$ is sub-holonomic.
\end{proof}

We know that if both the coadmissible modules $\Mp$ and $\M_\X$ are sub-holonomic, then $\phi(\suppi(\Mp)) = \supp(\M_\X)$.
We now prove this equality in the general case, ie when these modules are not necessarily sub-holonomic.

\begin{prop}\label{prop3.13}
Let $\phi : \X' \to \X$ be an admissible blow-up, $\M_\X$ a coadmissible $\Di$-module and $\Mp = \phi^! \M_\X$ the corresponding coadmissible $\Dip$-module.
Then
\[ \phi(\suppi(\Mp)) = \supp(\M_\X) . \]
\end{prop}
\begin{proof}
We can assume that the coadmissible modules $\Mp$ and $\M_\X$ are not sub-holonomic.
Since the formal curve is irreducible, we have $\suppi(\M_\X) = \X$.
Le us fix an integer $r_0 > \kp$ such that for any $r \geq r_0$, $\suppi(\M_\X) = \supp(\Mr) = \X$ and $\suppi(\Mp) = \supp(\Mrp)$.
We recall that $\supp (\Mr) = \overline{ \cup_{k \geq r} \supp (\Mkr)}$.
Thus, there are two possible cases: $\supp(\Mkr)$ if finite for all congruence level $k \geq r$ (but the size of these supports is increasing toward infinity), or $\supp(\Mkr) = \X$ for $k$ large enough.
In the first case, we deduce from lemma \ref{lemma3.9} that $\phi(\supp(\Mkrp)) = \supp(\Mkr)$ and lemma \ref{lemma3.10} implies that $\phi(\suppi(\Mp)) = \suppi(\M_\X)$.
It remains to consider the situation $\supp(\Mkr) = \X$ for some congruence level $k \geq r$.
Let $\mathcal{Z}$ be the closed subset associated to the blow-up $\phi : \X' \to \X$ and $V = \X \backslash \mathcal{Z}$.
Then $\phi_{|V} : \phi^{-1}(V) \to V$ is an isomorphism. It follows that $(\Mkrp)_{|\phi^{-1}(V)} \simeq (\Mkr)_{|V}$ and $\supp(\Mkrp) \supset \phi^{-1}(V)$.
Since the open subset $V$ is dense in $\X$ (recall that the formal curve $\X$ is irreducible) and the blow-up $\phi : \X' \to \X$ is a closed map, we have $\X = \overline{V} = \phi\left( \overline{\phi^{-1}(V)}\right)$.
By definition, the support $\supp(\Mkrp)$ is closed.
Thus, $\overline{\phi^{-1}(V)} \subset \supp(\Mkrp)$ and $\phi(\supp(\Mkrp)) = \X$.
To conclude, we have proved that $\phi(\suppi(\Mp)) = \X = \suppi(\M_\X)$.
\end{proof}

We deduce from the equality $\phi(\suppi(\Mp)) = \suppi(\M_\X)$ the following proposition saying that integrable connections are exactly the coadmissible $\Dip$-modules whose infinite support is empty.

\begin{prop}\label{prop3.8}
Let $\Mp$ be a coadmissible $\Dip$-module. Then $\suppi(\Mp) = \emptyset$ if and only $\Mp$ is a coherent $\O_{\X' , \Q}$-module.
In this situation, $\Mp$ is also a locally free $\O_{\X' , \Q}$-module of finite rank.
\end{prop}
\begin{proof}
Let us note $\M_\X = \phi_* \Mp$ and assume first that $\Mp$ is a coherent $\O_{\X' , \Q}$-module.
It follows from \cite[proposition 2.3.1]{huyghe} that $\M_\X = \phi_* \Mp$ is a coherent $\O_{\X , \Q}$-module.
Thus, again by \cite[proposition 6.8]{hallopeau3}, we know that $\suppi(\M_\X) = \emptyset$.
Since $\phi(\suppi(\Mp)) = \suppi(\M_\X)$, we have $\suppi(\Mp) = \emptyset$.
Conversely, we suppose that $\suppi(\Mp) = \emptyset$.
Then we have $\phi(\suppi(\Mp)) = \phi(\emptyset) = \suppi(\M_\X) = \emptyset$.
In other words, $\M_\X$ is an integrable connection and $\M_\X$ is a coherent $\O_{\X , \Q}$-module always by \cite[proposition 6.8]{hallopeau3}.
More precisely, we know that $\M_\X$ is locally a free $\O_{\X , \Q}$-module of finite rank $n \in \N$.
This comes from the fact that $\Mk$, where $\M_\X = \varprojlim_k \Mk$, is locally of the form $\Dkq / P$ with $P$ a dominant and unitary finite differential operator of order $n$.
It follows that $\Mkp \simeq \phi^! \Mk$ is then locally of the form $\Dkqp / P$.
As a consequence, $\Mkp$ is locally a free $\O_{\X' , \Q}$-module of rank $n$ by \cite[lemma 3.23]{hallopeau1}.
Thus, $\Mp \simeq \varprojlim_k \Mkp$ is also locally a free $\O_{\X' , \Q}$-module of rank $n$ by \cite[lemma 4.13]{hallopeau1}.
\end{proof}

Le us emphasize the following important point: this proposition works only at the level of the whole formal curve $\X'$.
More precisely, let $\Mp$ be a sub-holonomic coadmissible $\Dip$-module and let us note $V = \X' \backslash \suppi(\Mp)$.
In general, the image $\phi(V)$ of $V$ is larger than $\X \backslash \phi(\suppi(\Mp))$ and the restriction of $\phi : \X' \to \X$ to $V$ does not induce a map $\phi_{|V} : V \to \X \backslash \phi(\suppi(\Mp))$.
So we cannot directly apply this proposition to the coadmissible $\D_{V , \infty}$-module $(\Mp)_{|V}$.
In particular, the fact that $(\Mp)_{|V}$ is a coherent $\O_{V , \Q}$-module is not clear.
Nevertheless, we proved in lemma \ref{lemma 3.12} that any sub-holonomic $\Dip$-module $\Mp$ is generically an integrable connection, which is sufficient.
We recall that the horizontal multiplicity $m_0(\Mp)$ of $\Mp$ is the maximal rank of this connection.

\begin{prop}
A coadmissible $\Dip$-module is sub-holonomic if and only it is generically an integrable connection.
\end{prop}
\begin{proof}
The direct implication comes from lemma \ref{lemma 3.12}.
Let $\Mp$ be a coadmissible $\Dip$-module which is generically an integrable connection.
Assume that $\phi : \X' \to \X$ is the blow-up with respect to a closed subset $\mathcal{Z}$ of $\X$ and note $U = \X \backslash \mathcal{Z}$.
Then $\phi_{|\phi^{-1}(U)} : \phi^{-1}(U) \to U$ is an isomorphism.
It follows hat $(\M_\X)_{|U} \simeq (\Mp)_{|\phi^{-1}(U)}$ is generically a connection.
In particular, this shows that $\M_\X$ is sub-holonomic.
Finally, theorem \ref{theorem3.14} implies that $\Mp$ is also sub-holonomic.
\end{proof}

Let $\Mp$ be again a sub-holonomic coadmissible $\Dip$-module.
The x-axis of any vertical irreducible component $C'$ of $\Char(\Mp)$ is a closed point $x' \in \suppi(\Mp)$ ; we also denote by $m_{C'}(\Mp)$ the multiplicity $m_{x'}(\Mp)$ of $\Mp$ associated to the point $x'$.
Let $\Irr(\Mp)$ be the set of vertical irreducible component of $\Char(\Mp)$.
The characteristic cycle of $\Mp$ is defined by the following formal sum:
\[ \CC(\Mp) := m_0(\Mp) \cdot \X' + \sum_{C' \in \Irr(\Mp)} m_{C'}(\Mp) \cdot C' . \]

\begin{prop}\label{prop3.20}
This characteristic cycle is additive on short exact sequences of sub-holonomic $\Dip$-modules.
Moreover, let $\Mp$ be a sub-holonomic $\Dip$-module.
Then $m_0(\Mp) = m_0(\phi_*\Mp)$ and  $\Mp$ is zero if and only if $\CC(\Mp) = 0$.
\end{prop}
\begin{proof}
The additivity of characteristic cycles comes from proposition \ref{prop3.4b} together with the additivity of multiplicities.
We now prove that $m_0(\Mp) = m_0(\phi_*\Mp)$.
We can assume that $\X$ is affine and that $\Mk = \Dkq / \I_k$ where $\phi_*\Mp = \varprojlim_k \Mk$ as a coadmissible $\Di$-module.
Thanks to \cite[lemma 6.7]{hallopeau3} and \cite[corollary 3.11]{hallopeau1}, the multiplicity $m_0(\Mp)$ coincides with the integer $\Nb(\I_k) :=  \max\{ \Nb(P), \, P \in \I_k(\X) \}$ for $k$ large enough (where $\Nb(P)$ is the order of $(P \mod \varpi)$ after normalization).
The same holds for $\Mp = \varprojlim_k \Mkp$ when working on the open subset composed of smooth points of $\X'$.
Using \cite[theorem 2.3.8, theorem 2.3.12]{huyghe}, ie the equivalence of categories between coadmissible $\Dip$-modules and coadmissible $\Di$-modules and theorem A for coherent $\Dkqp$-modules, we get $\Mkp(\X') \simeq \Dkqp(\X') / \I_k(\X)$.
Since the order $\Nb$ of differential operators is a global notion, it follows that $m_0(\Mp) = \Nb(\I_k) = m_0(\M_\X)$.
Finally, assume that $\CC(\Mp) = 0$.
This is equivalent to have $\suppi(\Mp) = \emptyset$ and $m_0(\Mp) = 0$. This implies that $\Mp$ is locally a free $\O_{\X' , \Q}$-module of rank less or equal to $m_0(\Mp) = 0$.
Thus, $\Mp = 0$.
\end{proof}

We define the length $\ell(\Mp) \in \N$ of a sub-holonomic $\Di$-module $\Mp$ to be the length of its characteristic cycle, ie the sum of all its multiplicities.
With the notations above, we have $\ell(\Mp) = m(\Mp) + m_0(\Mp)$. It follows immediately that $\ell(\Mp) = \ell(\phi_*\Mp)$.
Moreover, $\Mp =0$ if and only if $\ell(\Mp) = 0$. Proposition \ref{prop3.20} implies the following.

\begin{cor}\label{cor3.21}
Any sub-holonomic $\Dip$-module $\Mp$ is of finite length less or equal to the length $\ell(\Mp)$.
\end{cor}

\begin{example}
Assume as in example \ref{example2.1} that $\X = \Spf(\V\langle x \rangle)$ and that $\phi : \X' \to \X$ is the admissible blow-up of $\X$ defined by the ideal $I = (x , \varpi)$ of $\V\langle x \rangle$. We have $\kp = 1$.
A standard open covering of $\X'$ is given by $U_1 = \Spf \left( \V \left\langle \frac{x}{\varpi} \right\rangle \right)$ and $U_2 = \Spf \left( \V \left\langle \frac{\varpi}{x} , x \right\rangle \right)$.
Consider $P = a(x)\cdot \partial^d + a_{d-1}(x)\cdot \partial^{d-1} + \dots + a_0(x) \in \Di(\X) \simeq \Dip(\X')$.
Since $P$ is a finite order differential operator, the coadmissible $\Di$-module $\M = \Di / P$ is sub-holonomic.
If $x_1 , \dots , x_s$ denote the zeroes of $a$ in $\X$, then we recall that $\suppi(\M) = \{x_1 , \dots, x_s\}$ and the characteristic variety of $\M$ is given by picture \ref{figurechar}.
The module $\M' = \phi^* \M \simeq \Dip / P$ is also sub-holonomic. Its vertical irreducible components correspond to the vertical line passing through the points of the infinite support $\suppi(\M')$ which coincides with the zeroes of $a$ inside the admissible blow-up $\X'$.
Over the open subset of smooth points of $\X' \backslash \suppi(\M')$, the coadmissible $\Dip$-module $\M'$ is locally a free $\O_{\X' , Q}$-module of order $d$, exactly as $\M$ on the open subset $W_\M = \X \backslash \{x_1 , \dots , x_s\}$ of $\X$.
We now assume that $a(x) = (x - \varpi)\cdot(x - \varpi^2) \in I = (x , \varpi) \subset \V\langle x \rangle$.
We have $\bar{a}(x) = x^2$ in $X = \Spec(\kappa[x])$ and $Z(\bar{a}) = \{x = 0\}$.
Thus, the infinite support $\suppi(\M)$ consists of the point $\{x = 0\}$ with multiplicity two.
Let us note $t = \frac{x}{\varpi}$ and $b(t)$ the normalized function $a$ seen as an element of $\O_{\X' , \Q}(\X') \simeq \O_{\X , \Q}(\X)$.
We have $a_{|U_1} = (\varpi t - \varpi)\cdot (\varpi t - \varpi^2) = \varpi^2 \cdot (t - 1) \cdot (t - \varpi)$, $b_{|U_1} = (t - 1) \cdot (t - \varpi)$ and $\bar{b}_{|U_1} = t \cdot (t - 1)$.
Similarly, one has $b_{|U_2} = (t^{-1} - 1) \cdot (t^{-1} - \varpi)$.
In particular, $Z(\bar{b})$ is composed of two points and $\suppi(\M')$ contains two closed points of multiplicity one.
\end{example}

\section{Sub-holonomic $\Dzr$-modules}\label{section5}

Let us denote by $\G$ the set of all admissible formal blow-ups $\X' \to \X$ of $\X$.
For another admissible formal blow-up $\X'' \to \X$ of $\X$, we say that $\X'' \geq \X'$ if the blow-up $\X'' \to \X$ factors through $\X' \to \X$.
The corresponding morphism $\pp : \X'' \to \X'$ is then an admissible blow-up which is uniquely determined by the universal property of blow-ups.
As a consequence, the set $\G$ of admissible blow-ups of $\X$ is partially ordered and directed, ie any two elements have a common upper bound.
The Zariski-Riemann space associated with $\X$ is defined as the projective limit of all admissible blow-ups:
\[ \zr := \varprojlim_{\X' \in \G} \X' .\]
Since the formal curve $\X$ is connected and quasi-compact, the Zariski Riemann space $\zr$ is also connected and quasi-compact.
This is proved for example in \cite[section II.3.1]{rigid}.
For an admissible blow-up $\X' \in \G$, we denote by $\sp : \zr \to \X'$ the canonical projection. This map is quasi-compact, closed and projective.
If $\X'' \geq \X'$, then we have $\sp = \pp \circ \, \spp$, where $\pp : \X'' \to \X'$ is the associated admissible blow-up.

\subsection{Coadmissible $\Dzr$-modules}

For any admissible blow-up $\pp : \X'' \to \X'$ in $\G$, we dispose respectively of the sheaves $\Dip$ and $\mathcal{\D}_{\X'' , \infty}$ of rapidly converging differential operators.
It was proved in \cite[proposition 3.1.4]{huyghe} that there exists a canonical isomorphism $(\pp)_* \mathcal{\D}_{\X'' , \infty} \simeq \Dip$.
We obtain from the adjunction formula $\pp^{-1} \circ (\pp)_* \to \id$ a transition homomorphism of sheaves $\sp^{-1} \Dip \to \spp^{-1}\D_{\X'' , \infty}$ commuting with the partial order of $\G$.
The sheaf $\Dzr$ of differential operators over the Zariski-Riemann space $\zr$ is defined in \cite{huyghe} by
\[ \Dzr := \varinjlim_{\X' \in \G} \sp^{-1} \Dip. \]

\begin{definition}
A coadmissible $\Dzr$-module is a $\Dzr$-module $\Mzr$ such that for any admissible blow-up $\X'$ of $\X$, there exists a coadmissible $\Dip$-modules $\Mp$ satisfying the following conditions:
\begin{enumerate}
\item
If $\X'' \geq \X'$, then there exists an isomorphism $\varphi_{\X'' , \X'} : (\pp)_* \M_{\X''} \overset{\simeq}{\longrightarrow} \Mp$.
\item
Whenever $\X''' \geq \X'' \geq \X'$, we have $\varphi_{\X'' , \X'} \circ (\pp)_*( \varphi_{\X''' , \X''}) = \varphi_{\X''' , \X'}$.
\item
$\Mzr$ is isomorphic to the inductive limit $\varinjlim_{\X' \in \G} \sp^{-1} (\Mp)$ as $\Dzr$-module.
\end{enumerate}
\end{definition}

Thanks to \cite[proposition 3.2.5]{huyghe}, we know that the category $\mathcal{C}_{\zr}$ of coadmissible $\Dzr$-module is abelian and equivalent to the category of coadmissible $\Dip$-modules for any admissible blow-up $\X'$ of $\X$.
More precisely, let us briefly recall how one can associate a coadmissible $\Dzr$-module to a coadmissible $\Dip$-module $\Mp$.
For any admissible blow-up $\pp : \X'' \to \X'$ of $\X'$, the $\Dip$-module $\M_{\X''} := \pp^! \Mp$ is coadmissible.
Moreover, we know by \cite[theorem 3.1.12]{huyghe} that we have a natural isomorphism $(\pp)_* \M_{\X''} \simeq \Mp$.
Thus, we get a transition homomorphism $\sp^{-1} \Mp \to \spp^{-1} \M_{\X''}$ and we define from the module $\Mp$ the coadmissible $\Dzr$-module $\Mzr := \varinjlim_{\X'' \in \mathcal{BL}_{\X'}} \spp^{-1} \M_{\X''}$.
The quasi-inverse functor is simply given by pushforward via the specialisation map $\sp : \zr \to \X'$.
Thus, we get an equivalence of abelian categories $(\sp)_* : \mathcal{C}_{\zr} \overset{\simeq}{\longrightarrow} \mathcal{C}_{\X'}$, where $\mathcal{C}_{\X'}$ is the category of coadmissible $\Dip$-modules for any admissible blow-up $\X' \in \G$.
Similarly, one can define by proposition \ref{propwh} the right coadmissible $\Dzr$-module
\[ \Ext^d_{\Dzr}(\Mzr , \Dzr) := \varinjlim_{\X' \in \G} \sp^{-1} \left(\Ext^d_{\Dip}(\Mp , \Dip)\right) .\]
The coadmissible module $\Mzr$ is said to be weakly holonomic if $\Ext^d_{\Dzr}(\Mzr , \Dzr) = 0$ for all $d \neq \dim\X = 1$.
Since $(\sp)_*\left(\Ext^d_{\Dzr}(\Mzr , \Dzr)\right) \simeq \Ext^d_{\Dip}(\Mp , \Dip)$, weakly holonomicity of $\Mzr$ implies that $\Mp$ is also weakly holonomic.
The converse comes from proposition \ref{propwh}, ie $\Mzr$ is weakly holonomic as soon as $\Mp$ is for some admissible blow-up $\X'$ of $\X$.
We recall that the dual of a coadmissible $\Dip$-module $\Mp  = \varprojlim_k \Mkp$ is defined by the left coadmissible $\Dip$-module
\[ \Mp^\vee := \varprojlim_k \left( \Ext^1_{\Dkqp}(\Mkp , \Dkqp) \otimes_{\O_{\X' , \Q}} \omega_{\X' , \Q}^{-1} \right) .\]
The dual of $\Mzr$ is then the left coadmissible $\Dzr$-module $\Mzr^\vee := \varinjlim_{\X' \in \G} \sp^{-1} \Mp^\vee$.
We deduce from proposition \ref{propdual} that $(\Mzr^\vee)^\vee \simeq \Mzr$ when $\Mzr$ is weakly holonomic.
Next proposition summarizes all this.

\begin{prop}
Let $\Mzr = \varinjlim_{\X' \in \G} \sp^{-1} \M_{\X'}$ be a coadmissible $\Dzr$-module.
Then $\Mzr$ is weakly holonomic if and only if $\Mp$ is for some admissible blow-up $\X' \in \G$.
In this case, all the coadmissible modules $\Mp$ are weakly holonomic and the dual $\Mzr^\vee$ of $\Mzr$ satisfies a biduality isomorphism.
\end{prop}

Let $\Mzr = \varinjlim_{\X' \in \G} \sp^{-1} \M_{\X'}$ be a coadmissible $\Dzr$-module.
One can check that proposition \ref{prop2.7} holds for any admissible blow-up $\pp : \X'' \to \X'$ in $\G$ ; the proof is exactly the same.
Thus, we have canonical isomorphisms $(\pp)_* \M_{\X'' , k , r} \simeq \Mkrp$ for all integers $k \geq r > \max\{\kp , k_{\X''}\}$.
We deduce from theorem \ref{theorem3.14} that the coadmissible module $\M_{\X"}$ is sub-holonomic if and only if the module $\Mp$ is.
In this situation, the support $\suppi(\M_{\X''})$ is finite and we deduce exactly as in lemmas \ref{lemma3.9} and \ref{lemma3.10} that $\pp(\suppi (\M_{\X''})) = \suppi(\Mp)$.
Working on irreducible components of $\X'$, we prove as in proposition \ref{prop3.13} that $\pp(\suppi (\M_{\X''})) = \suppi(\Mp)$.
As a consequence, we obtain the following proposition.

\begin{prop}
Let $\Mzr = \varinjlim_{\X' \in \G} \sp^{-1} (\Mp)$ be a coadmissible $\Dzr$-module and we fix a blow-up $\pp : \X'' \to \X'$ for some $\X'' \geq \X'$ in $\G$.
\begin{enumerate}
\item
For any integers $k \geq r > \max\{\kp , k_{\X''} \}$, there is a canonical isomorphism of coherent $\Fkrp$-modules $(\pp)_* \M_{\X'' , k , r} \simeq \Mkrp$.
\item
We have $\pp(\suppi (\M_{\X''})) = \suppi(\Mp)$.
\item
The coadmissible $\D_{\X'' , \infty}$-module $\M_{\X''}$ is sub-holonomic if and only if the coadmissible $\Dip$-module $\Mp$ is.
\end{enumerate}
\end{prop}

Let again $\Mzr = \varinjlim_{\X' \in \G} \sp^{-1} (\Mp)$ be a coadmissible $\Dzr$-module.
To resume, the coadmissible $\Dip$-modules $\Mp$ are all holonomic or none at all.
Let us recall that for any admissible blow-up $\pp : \X'' \to \X'$ in $\G$, $\sp = \pp \circ \spp$. It follows that $\spp^{-1}(\suppi \M_{\X''}) \subset \sp^{-1}(\suppi \Mp)$.

\begin{definition}
We associate to any coadmissible $\Dzr$-module $\Mzr = \varinjlim_{\X' \in \G} \sp^{-1} (\Mp)$ its infinite support defined by
\[ \suppi(\Mzr) := \bigcap_{\X' \in \G} \sp^{-1}(\suppi \Mp) \subset \zr .\]
\end{definition}

Let $\X'$ be in $\G$.
Since the specialisation map $\sp : \zr \to \X'$ is surjective, we have $\sp(\sp^{-1}(\suppi (\Mp))) = \suppi (\Mp)$.
Moreover, the equality $\pp(\suppi (\M_{\X''})) = \suppi(\Mp)$ also implies that for any $x' \in \sp^{-1}(\suppi \Mp)$, there exists a point $x'' \in \spp^{-1}(\suppi \M_{\X''})$ such that $\sp(x') = \sp(x")$.
We deduce from these two facts that $\sp(\suppi(\Mzr)) = \suppi (\Mp)$.
As a consequence, the coadmissible $\Dip$-module $\Mp$ is sub-holonomic for any admissible blow-up $\X' \in \G$ as soon as the support $\suppi(\Mzr)$ is finite.

\begin{prop}\label{prop4.3}
Let $0 \to \Nn_{\zr} \to \Mzr \to \L_{\zr} \to 0$ be a short exact sequence of coadmissible $\Dzr$-modules.
Then $\suppi (\Mzr) = \suppi (\Nn_{\zr}) \cup \suppi (\L_{\zr})$.
\end{prop}
\begin{proof}
Let $\X'$ be an admissible blow-up of $\X$.
Since the functor $(\sp)_*$ is exact for coadmissible $\Dzr$-modules by \cite[proposition 3.2.5]{huyghe}, we get a short exact sequence $0 \to \Nn_{\X'} \to \Mp \to \L_{\X'} \to 0$ of coadmissible $\Dip$-modules.
Thanks to proposition \ref{prop3.4b}, we know that $\suppi (\Mp) = \suppi (\Nn_{\X'}) \cup \suppi (\L_{\X'})$.
Thus,
\[ \sp^{-1}(\suppi (\Mp)) = \sp^{-1}(\suppi (\Nn_{\X'})) \cup \sp^{-1}(\suppi (\L_{\X'})) . \]
We immediately deduce that
\[ \suppi(\Mzr) = \bigcap_{\X' \in \G} \sp^{-1}(\suppi \Mp) \supset \suppi (\Nn_{\zr}) \cup \suppi (\L_{\zr}) . \]
Conversely, let $x \in \suppi(\Mzr) = \bigcap_{\X' \in \G} \left( \sp^{-1}(\suppi (\Nn_{\X'})) \cup \sp^{-1}(\suppi (\L_{\X'})) \right)$.
Then $x$ belongs for example to the set $\sp^{-1}(\suppi (\Nn_{\X'}))$ for an infinite number of admissible blow-ups $\X'$ of $\X$.
Recall that the sequence $(\sp^{-1}(\suppi (\Nn_{\X'})))_{\X' \in \G}$ is decreasing. As a consequence, $x \in \sp^{-1}(\suppi (\Nn_{\X'}))$ for any $\X' \in \G$.
Then $x \in \suppi (\Nn_{\zr}) = \bigcap_{\X' \in \G} \sp^{-1}(\Nn_{\X'})$.
Thus, $\suppi(\Mzr) \subset \suppi (\Nn_{\zr}) \,\cup \, \suppi (\L_{\zr})$.
\end{proof}

\subsection{Sub-holonomic $\Dzr$-modules}

We naturally define sub-holonomic coadmissible $\Dzr$-modules by asking there infinite supports to be finite.

\begin{definition}
Let $\Mzr$ be a coadmissible $\Dzr$-module.
\begin{enumerate}
\item
The module $\Mzr$ is said to be sub-holonomic if the support $\suppi(\Mzr))$ is finite.
\item
The module $\Mzr$ is called an integrable connection if this is a coherent $\O_{\zr , \Q}$-module.
\end{enumerate}
\end{definition}

Let $\X'$ be an admissible blow-up of $\X$ and $\sp : \zr \to \X'$ the specialization map.
Since $\sp(\suppi(\Mzr)) = \suppi (\Mp)$, the fact that $\Mzr = \varinjlim_{\X' \in \G} \sp^{-1} (\Mp)$ is sub-holonomic implies that the coadmissible $\Dip$-modules $\Mp$ are all sub-holonomic.
As a consequence, any sub-holonomic $\Dzr$-module is weakly holonomic.
Next example implies that for any integrable connection $\Mzr$, we have $\suppi(\Mzr) = \emptyset$.
In particular, integrable connections are sub-holonomic.

\begin{example}\label{example4.5}
The coadmissible $\Dzr$-module $\O_{\zr , \Q} := \varinjlim_{\X' \in \G} \sp^{-1} \O_{\X' , \Q}$ is by definition an integrable connection.
Moreover, since all the coadmissible $\Dip$-modules $\O_{\X' , \Q}$ are integrable connections, we have $\suppi(\O_{\X' , \Q}) = \emptyset$.
It follows that
\[ \suppi(\O_{\zr , \Q}) = \bigcap_{\X' \in \G} \sp^{-1}(\suppi \Mp) = \emptyset . \]
Similarly, the coadmissible $\Dzr$-modules $(\O_{\zr , \Q})^n \simeq \varinjlim_{\X' \in \G} \sp^{-1} (\O_{\X' , \Q})^n$ are integrable connections and satisfy $\suppi((\O_{\zr , \Q})^n) = \emptyset$.
\end{example}

\begin{prop}\label{prop4.6}
A coadmissible $\Dzr$-module $\Mzr$ is an integrable connection if and only if $\suppi(\Mzr) = \emptyset$.
In this case, $\Mzr$ is locally a free $\O_{\zr , \Q}$-module of finite rank.
\end{prop}
\begin{proof}
The direct implication comes from example \ref{example4.5}.
Let us assume that the coadmissible module $\Mzr = \varinjlim_{\X' \in \G} \sp^{-1} (\Mp)$ has an empty infinite support, ie $\suppi(\Mzr) = \emptyset$.
Since $\mathrm{sp}_\X(\suppi(\Mzr)) = \suppi (\M_\X)$, it follows that $\suppi (\M_\X) = \emptyset$.
Then \cite[proposition 6.8]{hallopeau3} implies that the coadmissible $\Di$-module $\M_\X$ is locally a free $\O_{\X , \Q}$-module of finite rank.
Let $U$ be an affine open subset of $\X$ such that $(\M_\X)_{|U} \simeq (\O_{\X , \Q})_{|U}^n$ as $\Di$-module.
Then by \cite[theorem 3.2.6]{huyghe} (theorem A for coadmissibles $\Dzr$-modules) we deduce that $(\Mzr)_{|\mathrm{sp}_\X^{-1}(U)} \simeq  (\O_{\zr , \Q})_{|\mathrm{sp}_\X^{-1}(U)}^n$.
In this way, we can obtain an open covering of $\zr$ for which $\Mzr$ is locally a free $\O_{\zr , \Q}$-module of finite rank.
\end{proof}

We now prove the equivalence between holonomicity over the formal smooth model $\X$ and the Zariski-Riemann space $\zr$.

\begin{theorem}\label{theorem4.7}
Let $\Mzr = \varinjlim_{\X' \in \G} \sp^{-1} (\Mp)$ be a coadmissible $\Dzr$-module.
\begin{enumerate}
\item
The module $\Mzr$ is sub-holonomic if and only if one of the coadmissible modules $\Mp$ is sub-holonomic for some admissible blow-up $\X'$ of $\X$.
In this case, all the coadmissible $\Dip$-modules $\Mp$ are sub-holonomic.
\item
If $\Mzr$ is sub-holonomic, then it has finite multiplicities at the points of its infinite support $\suppi(\Mzr)$.
Moreover, these multiplicities are additive on short exact sequences of sub-holonomic coadmissible $\Dzr$-modules.
\end{enumerate}
\end{theorem}
\begin{proof}
We have already seen that if $\Mzr$ is sub-holonomic, then all the modules $\Mp$ are.
It follows from the fact that $\sp(\suppi(\Mzr)) = \suppi (\Mp)$.
Thanks to theorem \ref{theorem3.14}, this is equivalent to have one of the $\Mp$ sub-holonomic.
Assume that all the coadmissible $\Dip$-modules $\Mp$ are sub-holonomic.
The fact that $\Mzr$ is sub-holonomic is not totally obvious.
Indeed, we only know that all the supports $\suppi(\Mp)$ are finite.
But potentially the intersection $\suppi(\Mzr) = \cap_{\X' \in \G} \sp^{-1}(\suppi \Mp)$ is not.
Let us recall that if $\pp : \X'' \to \X'$ are admissible blow-ups of $\X$, then $\pp(\suppi(\M_{\X''})) = \suppi(\Mp)$.
As a consequence, $\#\suppi(\M_{\X''}) \geq \# \suppi(\Mp)$.
The module $\Mzr$ is sub-holonomic only if these cardinals stabilize. 
The closed subsets $(\suppi \Mp)_{\X' \in \G}$ form a compatible system in the projective limit $\zr = \varprojlim_{\X' \in \G} \X'$.
Since $\sp = \pp \circ \spp$, we observe that
\[ \varprojlim_{\X' \in \G} \suppi (\Mp) = \bigcap_{\X' \in \G} \sp^{-1}(\suppi \Mp) = \suppi(\Mzr) . \]
We saw in section \ref{section3.2}, relation \ref{eq2}, that
\[ m(\M_\X) = \sum_{x \in \suppi(\M_\X)} m_x(\M_\X)  = \sum_{x' \in \suppi(\Mp)} m_{x'}(\Mp) . \]
Since by definition, $m_{x'}(\Mp) \geq 1$ for all point $x'$ of the support $\suppi(\Mp)$, the closed set $\suppi(\Mp)$ contains at most $m(\M_\X)$ points.
For any blow-up $\pp : \X'' \to \X'$, $\pp(\suppi(\M_{\X''})) = \suppi(\Mp)$ and $\#\suppi(\M_{\X''}) \geq \# \suppi(\Mp)$.
As a consequence, there exists an admissible blow-up $\X'$ of $\X$ such that for any blow-up $\X'' \to \X'$,
\[ n := \#\suppi(\M_{\X''}) = \# \suppi(\Mp) \leq m(\M_\X) .\]
It follows that the map $\pp$ induces a bijection $\suppi(\M_{\X''}) \overset{\simeq}{\longrightarrow} \suppi(\Mp)$.
We deduce that the infinite support $\suppi(\Mzr) = \varprojlim_{\X' \in \G} \suppi (\Mp)$ is composed of $n$ elements.
Write $\suppi(\Mzr) = \{ x_1 , \dots , x_n \}$, $x'_i = \sp(x_i)$ and $x_i'' = \spp(x_i)$. The map $\pp : \suppi(\M_{\X''}) \to \suppi(\Mp), ~ x_i'' \to x_i'$ is then a bijection.
We know that $m_{x_i''}(\M_{\X''}) \leq m_{x_i'}(\Mp)$ and $\sum_{i=1}^n m_{x_i'}(\Mp)  = \sum_{i=1}^n m_{x_i''}(\M_{\X''})$.
As a consequence, for any $i \in \{ 1 , \dots , n \}$ we have $m_{x_i''}(\M_{\X"}) = m_{x_i'}(\Mp)$.
We set $m_{x_i}(\Mzr) := m_{x_i'}(\Mp) \in \N^*$. This integer does not depend on the admissible blow-up $\X'$.
The additivity of these multiplicities comes from the additivity of the multiplicities for sub-holonomic coadmissibles $\Dip$-modules.
 \end{proof}
 
Let us denote by $\mathcal{SH}_{\zr}$ the full sub-category of $\mathcal{C}_{\zr}$ composed of sub-holonomic $\Dzr$-modules.
Proposition \ref{prop4.3} implies that the category $\mathcal{SH}_{\zr}$ is abelian.
Since the functor $(\sp)_* : \mathcal{C}_{\zr} \to \mathcal{C}_{\X'}$ is exact, theorem \ref{theorem4.7} together with \cite[Corollary 6.24]{hallopeau3} imply that sub-holonomic coadmissible $\Dzr$-modules are of finite length.
To resume, we have the following result.

\begin{cor}\label{cor4.9}
Let $\X' \in \G$.
The equivalence $(\sp)_* : \mathcal{C}_{\zr} \overset{\simeq}{\longrightarrow} \mathcal{C}_{\X'}$ of \cite[proposition 3.2.5]{huyghe} induces an equivalence of abelian categories $(\sp)_* : \mathcal{SH}_{\zr} \overset{\simeq}{\longrightarrow} \mathcal{SH}_{\X'}$.
As a consequence, sub-holonomic coadmissible $\Dzr$-modules are of finite length.
\end{cor}

\begin{example}~
We assume that $\X$ is affine with a local coordinate. Let $P \in \Di(\X)$ be a finite differential operator.
Then we know by \cite[proposition 6.2]{hallopeau3} that the coadmissible $\Di$-module $\M_\X := \Di / P$ is sub-holonomic.
It follows that the coadmissible $\Dzr$-module $\Mzr$ associated to $\M_\X$ is sub-holonomic.


\end{example}

We end this section by proving that any sub-holonomic $\Dzr$-modules is generically an integrable connection.

\begin{prop}\label{prop4.11}
A coadmissible $\Dzr$-module $\Mzr$ is sub-holonomic if and only if it is generically an integrable connection: there exists an open dense subset $W$ of $\zr$ such that $(\Mzr)_{|W}$ is locally a free $\O_{W , \Q}$-module of finite rank.
Moreover, this rank coincides with the horizontal multiplicity $m_0(\Mp)$ of the coadmissible sub-holonomic $\Dip$-modules $\Mp$.
\end{prop}
\begin{proof}
Assume firstly that the coadmissible $\Dzr$-module $\Mzr = \varinjlim_{\X' \in \G} \sp^{-1} (\Mp)$ is generically an integrable connection and let $W$ be as in the statement.
Then by definition, $(\Mzr)_{|W}$ is locally a coherent $\O_{\zr , \Q}$-module.
Since the Zariski-Riemann space $\zr$ is quasi-compact and a basis of open subsets is given by $\left\{ \sp^{-1}(V), \, \X' \in \G, \, V\subset \X' \, \text{open} \right\}$, we can assume that $W = \sp^{-1}(W')$ for some admissible blow-up $\X' \in \G$ with $W'$ an open subset of $\X'$.
It follows that the coadmissible module $(\Mp)_{|W'} \simeq (\sp)_* (\Mzr)_{|W}$ is a coherent $\O_{W' , \Q}$-module, so an integrable connection.
We deduce from the fact that the specialisation map $\sp : \zr \to \X'$ is both surjective and closed that $W'$ is also dense in $\X'$.
As a consequence, the module $\Mp$ is generically an integrable connection and then sub-holonomic.
Finally, the coadmissible $\Dzr$-module $\Mzr$ is sub-holonomic thanks to corollary \ref{cor4.9}.

Conversely, assume that $\Mzr$ is sub-holomomic. Let $W_{\Mp}$ be the maximal dense open subset of $\X'$ such that $(\Mp)_{W_{\Mp}}$ is a locally free $\O_{\X' , \Q}$-module of finite rank.
Thanks to proposition \ref{prop3.20}, we know that this rank is an integer $m_0$ that does not depend on the choice of the admissible blow-up $\X'$.
Let $\pp : \X'' \to \X'$ be an admissible blow-up in $\G$. Then we have $\pp^{-1}(W_{\Mp}) \subset W_{\M_\X''}$.
Indeed, let us choose a point $x' \in W_{\Mp}$. There exists an affine open neighbourhood $V_{x'}$ of $x'$ such that $(\Mp)_{|V_{x'}} \simeq (\O_{V_{x'} , \Q})^{m_0}$.
We deduce that $(\M_{\X''})_{|\pp^{-1}(V_{x'})} \simeq (\O_{\pp^{-1}(V_{x'}) , \Q})^{m_0}$. Thus, $\pp^{-1}(x') \subset W_{\M_{\X''}}$.
In other words, $(W_{\Mp})_{\X' \in \G}$ is a compatible sequence of open subsets inside $\zr = \varprojlim_{\X' \in \G} \X'$.
Let us introduce $W := \varprojlim_{\X' \in \G} W_{\Mp} = \bigcup_{\X' \in \G} \sp^{-1}(W_{\Mp})$ which is an open subset of the Zariski-Riemann space $\zr$ by definition of its topology.
We now check that $W$ is dense. Let $x \in \zr$ and $V_x = \sp^{-1}(V')$ be an open neighbourhood of $x$.
Then $W_{\Mp} \cap V' \neq \emptyset$ because $W_{\Mp}$ is dense in $\X'$.
As a consequence, $W \cap V_x \neq \emptyset$.
Since the open subsets of the form $V_x$ give rise to a basis of open neighbourhood of $x$, this proves that $W$ is dense in $\zr$.
Finally, we verify that $(\Mzr)_{|W}$ is locally a free $\O_{\zr , \Q}$-module of finite rank $m_0$.
Indeed, for any point $x$ of $W$, let us choose an open neighbourhood $V_x = \sp^{-1}(V')$ of $x$ such that $(\Mp)_{|V'} \simeq (\O_{V' , \Q})^{\m_0}$.
The equivalence of categories given between coadmissible modules over $\zr$ and $\X'$ by the specialization map $\sp$ implies that $(\Mzr)_{| V_x} \simeq (\O_{V_x , \Q})^{\m_0}$.
\end{proof}

\section{Sub-holonomic $\widehat{\D}_{\X_K}$-modules}\label{section6}

In this last section, we associate to coadmissible $\Dcap$-modules a characteristic variety which is a closed subset of the cotangent space $T^* \X_K$. We deduce a notion of sub-holonomicity for these modules.

\subsection{Equivalence of categories between $\mathcal{C}_{\X_K}$ and $\mathcal{C}_{\zr}$}

We assume in this subsection that $\X$ is a smooth and separated formal $\V$-scheme of any dimension $d \in \N$.
We denote by $\X_K$ the rigid analytic space over $K$ associated to the formal smooth scheme $\X$ and by $\tilde{\text{sp}}_{\X'} : \X_K \to \X'$ the specialization map for any admissible blow-up $\X' \in \G$.
These maps induce an arrow $\text{sp} : \X_K \to \zr$ which is injective with dense image for the constructible topology on the Zariski-Riemann space $\zr$.
Let us mention that $\tilde{\text{sp}}_{\X'}$ is then the composition of the specialization map $\sp : \zr \to \X'$ by $\text{sp} : \X_K \to \zr$.
We dispose of the sheaf $\Dcap$ of rapidly converging differential operators introduced in \cite{ABW1} by Ardakov-Wadsley over the smooth rigid analytic space $\X_K$ together with a notion of coadmissible $\Dcap$-modules.
Let us quickly detail the construction of these.
For simplicity, we assume that $\X = \Spf \mathcal{A}$ is affine with a system of local coordinates.
Then $\X_K = \mathrm{Sp} A$ is affinoid with $A = \mathcal{A}\otimes_\V K$ such that its tangent sheaf $\tau_{\X_K}$ if a free $\O_{\X_K}$-module of rank $d$.
We denote by $\partial_1, \dots, \partial_d$ the basis of $T_{\X_K}(\X_K)$ and note $\tau := \tau_\X(\X) \simeq \bigoplus_{i=1}^d \mathcal{A}\cdot \partial_i$.
The sheaf $\Dcap$ can be written as a projective limit of coherent sheaves of differential operators $\D_n$ defined over the sites $\X_K(\varpi^n \tau)$ of $\varpi^n \tau$-admissible subdomains of $\X_K$, see \cite[section 4.6-4.7]{ABW1}.
More precisely, let $Y$ be an admissible open subset of $\X_K$. Then $Y$ is $\varpi^n \tau$-admissible for $n$ large enough, ie $Y \in \X_K(\varpi^n \tau)$.
Thus, $\D_n(Y)$ is well defined for $n$ sufficiently large and $\Dcap(Y) := \varprojlim_n \D_n(Y)$ is a Fr\'echet-Stein algebra.
One has
\[ \D_n(\X_K) = \left\{ \sum_{\alpha \in \N^d} a_\alpha \cdot \partial^\alpha,~~ a_\alpha \in A ~,  ~~\lim_{|\alpha| \to \infty} |\varpi|^{-n |\alpha|} \cdot |a_\alpha| = 0 \right\},\]
\[ \Dcap(\X_K) = \left\{ \sum_{\alpha \in \N^d} a_\alpha \cdot \partial^\alpha, \, a_\alpha\in A, \, \lim_{|\alpha| \to \infty} |\varpi|^{-n |\alpha|} \cdot |a_\alpha| = 0 \hspace{0.3cm} \forall n \in \N \right\} . \]
A $\Dcap$-module $\M$ is said to be coadmissible if there exists an admissible affinoid covering $(Y_i)$ of $\X_K$ such that $\M(Y_i)$ is a coadmissible $\Dcap(Y_i)$-module and such that for any affinoid subdomain $Z$ of $Y_i$, the natural map $\Dcap(Z) \widehat{\otimes}_{\Dcap(Y_i)} \M(Y_i) \to \M(Z)$ is an isomorphism.
Let us recall that a $\Dcap(Y_i)$-module is coadmissible if it is isomorphic to a projective limit of coherent $\D_n(Y_i)$-modules $\M_n(Y_i)$ such that $\D_n(Y_i) \otimes_{\D_{n+1}(Y_i)} \M_{n+1}(Y_i) \simeq \M_n(Y_i)$.

\begin{prop}\label{prop5.1}
Let $\X$ be a smooth and separated formal scheme of dimension $d$ over $\V$. 
The map $\mathrm{sp} : \X_K \to \zr$ induces an exact equivalence of categories between coadmissible $\widehat{\D}_{\X_K}$-modules and coadmissible $\Dzr$-modules.
\end{prop}
\begin{proof}
Let $U = \Spf \mathcal{A}$ be an affine open subset of $\X$ with a system of local coordinates and we keep the notations used above: $U_K = \mathrm{Sp} A$ with $A = \mathcal{A}\otimes_\V K$, its tangent sheaf $\tau_{U_K}$ if a free $\O_{U_K}$-module with basis $\partial_1, \dots, \partial_d$ and $\tau := \tau_\X(U) \simeq \bigoplus_{i=1}^d \mathcal{A}\cdot \partial_i$.
We have
\[ \Dcap(U_K) = \left\{ \sum_{\alpha \in \N^d} a_\alpha \cdot \partial^\alpha, \, a_\alpha\in A, \, \lim_{|\alpha| \to \infty} |\varpi|^{-n |\alpha|} \cdot |a_\alpha| = 0 \hspace{0.3cm} \forall n \in \N \right\} \simeq \Di(U) .\]
Recall that for any admissible blow-up $\phi : \X' \to \X$, $\Di(U) \simeq \Dip(\phi^{-1}(U))$.
Let $Y$ be a quasi-compact admissible open subset of $\X_K$ contained in the inverse image of some affine open subset of $\X$ admitting local coordinates.
We denote by $\text{sp}_*(Y) := \bigcup W$, where the union is taken over all open subsets $W$ of $\zr$ such that $W \cap \text{sp}(\X_K) \subset \text{sp}(Y)$, the open subset of $\zr$ induced by $Y$.
Since the open subset $Y$ is assumed to be quasi-compact, it is represented by an open subset $\mathcal{Y}'$ of some admissible blow-up $\X'$ of $\X$.
In this case, we have $\text{sp}_*(Y) = \sp^{-1}(\mathcal{Y}')$ where $\sp : \zr \to \X'$ is the specialization map. As before, we see that $\Dcap(Y) \simeq \Dip(\mathcal{Y}')$.
For any admissible blow-up $\pp : \X'' \to \X'$, we have $\text{sp}_*(Y) = \spp^{-1}(\mathcal{Y}'')$ with $\mathcal{Y}'' = \pp^{-1}(\mathcal{Y})$.
Recall that $\Dzr := \varinjlim_{\X' \in \G} \sp^{-1} \Dip$. Thus, we get an isomorphism of $K$-algebras $\Dcap(Y) \simeq \Dzr(\text{sp}_*(Y))$.
The set of quasi-compact admissible open subsets $Y$ is a basis of open subsets for $\X_K$ and the open subsets of the form $\text{sp}_*(Y) = \sp^{-1}(\mathcal{Y}')$ define a basis for the topology on $\zr$.
As a consequence, we obtain an isomorphism between the sheaves $\Dcap$ and $\Dzr$ associating $\Dcap(Y)$ to $\Dzr(\text{sp}_*(Y))$.
Let us emphazise that at this point we do not prove that this isomorphism preserves the Fr\'echet-Stein structures.
Similarly, we associate to a $\Dcap$-module $\M$ a $\Dzr$-module $\Mzr$ via the local association $\Mzr(\text{sp}_*(Y)) := \M(Y)$ and vice-versa.
It remains to check that coadmissibility is preserved on both sides.
By definition, there exists an admissible affinoid open covering $(Y_i = \mathrm{Sp} A_i)_i$ of $\X_K$ such that $\M(Y_i)$ is a coadmissible $\Dcap(Y_i)$-module.
For all $i$, let us fix an admissible blow-up $\X_i$ of $\X$ such that $Y_i$ is represented by an affine open subset $\mathcal{Y}_i = \Spf \mathcal{A}_i$ of $\X_i$.
The affinoid $Y_i$ is $\varpi^n \tau$-admissible for $n \in \N$ large enough, say $n \geq n_0$. We can assume that $n_0 \geq k_{\X_i}$.
Then $\M(Y_i) \simeq \varprojlim_{n \geq n_0} \M_n(Y_i)$ where $M_n(Y_i)$ is a coherent $\D_n(Y_i)$-module.
We have 
\[ \D_n(Y_i) = \left\{ \sum_{\alpha \in \N^d} a_\alpha \cdot \partial^\alpha,~~ a_\alpha \in A_i ~,  ~~\lim_{|\alpha| \to \infty} |\varpi|^{-n |\alpha|} \cdot |a_\alpha| = 0 \right\} \simeq \widehat{\D}^{(0)}_{\X_i , n , \Q}(\mathcal{Y}_i) .\]
In particular, $\M_{\mathcal{Y}_i , n} (\mathcal{Y}_i) := \M_n(Y_i)$ is a coherent $\widehat{\D}^{(0)}_{\X_i , n , \Q}(\mathcal{Y}_i)$-module and the projective limit $\M_{\mathcal{Y}_i}(\mathcal{Y}_i) := \M(Y_i) \simeq \varprojlim_{n \geq n_0} \M_{\mathcal{Y}_i , n} (\mathcal{Y}_i)$ is a coadmissible $\D_{\X_i , \infty}(\mathcal{Y}_i)$-module.
Thanks to \cite[theorem 3.1.13]{huyghe}, $\M_{\mathcal{Y}_i}(\mathcal{Y}_i)$ uniquely determines a coadmissible $\D_{\mathcal{Y}_i , \infty}$-module $\M_{\mathcal{Y}_i}$ over $\mathcal{Y}_i$.
Let $\Mzr$ be the $\Dzr$-module associated to $\M$.
By \cite[proposition 3.2.5]{huyghe}, the functor $(\text{sp}_{\mathcal{Y}_i})_* : \text{sp}_{\X_i}^{-1}(\mathcal{Y}_i) \to \mathcal{Y}_i$ induces an equivalence of categories between coadmissible $\D_{\text{sp}_{\X_i}^{-1}(\mathcal{Y}_i)}$-modules and coadmissible $\D_{\mathcal{Y}_i , \infty}$-modules.
It follows that $(\Mzr)_{|\text{sp}_{\X_i}^{-1}(\mathcal{Y}_i)}$ is coadmissible since the module $(\text{sp}_{\mathcal{Y}_i})_*(\Mzr)_{|\text{sp}_{\X_i}^{-1}(\mathcal{Y}_i)} \simeq \M_{\mathcal{Y}_i}$ is.
The open subsets $\text{sp}_{\X_i}^{-1}(\mathcal{Y}_i) = \text{sp}_*(Y_i)$ define an open covering of $\zr$. Thus, the module $\Mzr$ is locally coadmissible.
Moreover, this coadmissibily is compatible with restrictions to open subsets because $\M$ is also compatible with such restictions.
As a consequence, $\Mzr$ is a coadmissible $\Dzr$-module.
In particular, this also implies that the isomorphism between $\Dcap$ and $\Dzr$ induced by the map $\text{sp} : \X_K \to \zr$ is compatible with the Fr\'echet-Stein structure.
Similarly, we verify that the $\Dcap$-module associated to a coadmissible $\Dzr$-module is also coadmissible.
\end{proof}

\subsection{Sub-holonomic $\Dcap$-modules}

\paragraph{Quasi-compact situation}

We now return to the case of a smooth, quasi-compact and connected formal curve $\X$.
Let $\Mzr = \varinjlim_{\X' \in \G} \sp^{-1} (\Mp)$ be a sub-holonomic $\Dzr$-module.
Then its infinite support $\suppi(\Mzr) = \{x_1 , \dots , x_s\}$ consists of a finite set of closed points in bijection with $\suppi(\Mp) = \{x_1' , \dots , x_s'\}$ for some admissible blow-up $\X'$ of $\X$.
In particular, since the specialization map $\text{sp} : \X_K \to \X'$ is surjective onto the closed points of $\X'$, the points of $\suppi(\Mzr)$ corresponds bijectively to points of the rigid analytic space $\X_K$.
In other words, there exists unique points $\tilde{x}_1, \dots , \tilde{x}_s \in \X_K$ such that $\sp(\tilde{x}_i) = x_i$.
Thus, we identify $\suppi(\Mzr)$ with points of $\X_K$ via the specialization map $\text{sp} : \X_K \to \zr$.
Let us mention that $\suppi(\Mzr)$ is also in bijection with $\suppi(\M_{\X''})$ for any admissible blow-up $\X'' \to \X'$, so the choice of the admissible blow-up does not matter.

\begin{definition}
A coadmissible $\Dcap$-module $\M$ is said to be sub-holonomic if its associated coadmissible $\Dzr$-module $\Mzr$ is sub-holonomic.
We define its infinite support by $\suppi(\M) := \mathrm{sp}^{-1}(\suppi(\Mzr))$.
\end{definition}

We denote by $\mathcal{SH}_{\X_K}$ the full subcategory of coadmissible $\Dcap$-modules consisting of sub-holonomic objects.
By definition of sub-holonomic $\Dcap$-modules, the equivalence of categories between coadmissible $\Dcap$-modules and $\Dzr$-modules given by the specialization map $\text{sp} : \X_K \to \zr$ induces an equivalence of categories between $\mathcal{SH}_{\X_K}$ and $\mathcal{SH}_{\zr}$.
In particular, the category $\mathcal{SH}_{\X_K}$ of sub-holonomic $\Dcap$-modules is abelian.
Let $\M$ be a sub-holonomic $\Dcap$-module. Its infinite support $\suppi(\M)$ is in bijection with $\suppi(\Mzr)$.
In particular, we can add to the points of $\suppi(\M)$ multiplicities coming from the ones of $\suppi(\Mzr)$: for any $x \in \suppi(\M)$, $m_x(\M) := m_{\text{sp}(x)}(\Mzr)$.
The following result immediately comes from theorem \ref{theorem4.7}, propositions \ref{prop4.3} and \ref{prop5.1}.

\begin{prop}\label{prop5.3}
Let $0 \to \M \to \Nn \to \L \to 0$ be an exact sequence of coadmissible $\Dcap$-modules.
Then $\suppi(\Nn) = \suppi(\M) \cup \suppi(\L)$. In particular, $\Nn$ is sub-holonomic if and only if the modules $\M$ and $\L$ are.
In this case, we have $m_x(\Nn) = m_x(\M) + m_x(\L)$ for any $x\in \suppi(\Nn)$.
\end{prop}

It remains to obtain an horizontal multiplicity for sub-holonomic $\Dcap$-modules.
For that, we prove that such a module is generically an integrable connection.
We verify exactly as in proposition \ref{propconnections} that coherent $\O_{\X_K}$-module endowed with an integrable connection coincide with coadmissible $\Dcap$-modules which are also $\O_{\X_K}$-coherent.
This result was also demonstrated by Ardakov-Wadsley in \cite[theorem B]{ardakov2}.
Again, we identify integrable connections with $\O_{\X_K}$-coherent coadmissible $\Dcap$-modules.

\begin{prop}\label{prop5.4}
A coadmissible $\Dcap$-module $\M$ is sub-holonomic if and only if there exists an admissible open dense subset $W$ of $\X_K$ obtained by removing a finite number of points (one can choose $W = \X_K \backslash \suppi(\M)$) such that $\M_{|W}$ is locally a free $\O_{\X_K}$-module of finite rank.
Moreover, this rank coincides with the horizontal multiplicity $m_0(\Mzr)$ of its associated coadmissible $\Dzr$-module $\Mzr$.
\end{prop}
\begin{proof}
Let $\M$ be a sub-holonomic $\Dcap$-module and $Y$ an affinoid open subset of $\X_K$ inside $\X_K \backslash \suppi(\M)$.
We fix an admissible blow-up $\X'$ of $\X$ such that $Y$ comes from an affine open subset $\mathcal{Y}'$ of $\X'$ and such that $\suppi(\M) \simeq \suppi(\Mp)$, where $\Mzr = \varinjlim_{\X' \in \G} \sp^{-1} (\Mp)$ is the coadmissible $\Dzr$-module corresponding to $\M$.
The affinoid open subset $Y$ is $\varpi^k \tau$-admissible for $k \in \N$ large enough, say $k \geq k_0 \geq \kp$.
We have $\M(Y) \simeq \varprojlim_{k \geq k_0} \M_k(Y)$ as a coadmissible $\Dcap(Y) = \varprojlim_{k \geq k_0}  \D_k(Y)$-module.
We have seen that $\M_k(Y) \simeq \Mkp(\mathcal{Y}')$ as a coherent $\D_k(Y) \simeq \Dkqp(\mathcal{Y}')$-module.
The fact that $Y \subset \X_K \backslash \suppi(\M)$ implies that $\mathcal{Y}' \subset \X' \backslash \suppi(\Mp)$.
It then follows by lemma \ref{lemma3.8} that $\M_k(Y)$ is a coherent $\O_{\X_K}(Y)$-module.
Since $\X_K$ is a smooth rigid analytic curve, the affinoid $K$-algebra $\O_{\X_K}(Y)$ is a principal ideal domain.
The structure of $\D_k(Y)$-module provides a connection on $\M_k(Y)$.
In particular, this implies that $\M_k(Y)$ is torsion free as $\O_{\X_K}(Y)$-module and we deduce as in \cite[section 1.3.10]{adriano} that $\M_k(Y)$ is free over $\O_{\X_K}(Y)$ of finite rank.
Necessarily, this rank equals the rank $m_0(\Mp)$ of $\Mkp$ where this is a locally free $\O_{\X' , \Q}$-module.
In particular, we have proved that for $k$ large enough, $\M_k(Y)$ is a free $\O_{\X_K}(Y)$-module of finite rank $m_0(\Mp)$ independent of $k$.
It follows that $\M(Y) \simeq \varprojlim_{k \geq k_0} \M_k(Y)$ is also a free $\O_{\X_K}(Y)$-module of rank $m_0(\Mp)$.
We recall that $m_0(\Mzr) = m_0(\Mp)$. Thus, $\M$ is locally on affinoid open subsets of $\X_K \backslash \suppi(\M)$ a free $\O_{\X_K}$-module of rank $m_0(\Mzr)$.
Conversely, assume that there exists an admissible open subset $W$ of $\X$ obtained by removing a finite set of points such that $\M_{|W}$ is a coherent $\O_{\X_K}$-module.
Then $\suppi(\M_{|W}) = \emptyset$ and $\suppi(\M) \subset \X_K \backslash W$ is finite.
Thus, the coadmissible $\Dcap$-module $\M$ is sub-holonomic.
\end{proof}

We define the horizontal multiplicity $m_0(\M)$ of a sub-holonomic $\Dcap$-module $\M$ to be this rank.
This is clearly additive on short exacte sequences of sub-holonomic $\Dcap$-modules.

\begin{example}
Let $\X = \Spf (\V\langle x \rangle)$ be the affine formal line over $\V$.
We identify the Zariski Riemann space $\zr$ with the rigid analytic unit disc $\X_K = \mathrm{Sp}(K\langle x \rangle)$ and we identify $\Dcap$ with the sheaf $\Dzr$ via proposition \ref{prop5.1}.
Let $\X_K^\circ = \X_K \backslash \{x = 0\}$ and $j : \X_K^\circ \hookrightarrow \X_K$ be the natural inclusion.
Then $j_* \O_{\X_K^\circ}$ is a coadmissible $\D_{\X_K}$-module by  \cite[proposition 6.2]{huyghe2}.
Moreover, this is sub-holonomic with $\suppi(j_* \O_{\X_K^\circ}) = \{ x=0 \}$.
Indeed, the coadmissible module $j_* \O_{\X_K^\circ}$ is a connection over $\X_K^\circ$ (where this is exactly $\O_{\X_K^\circ}$).
Thus, it has an horizontal multiplicity equal to one.
A formal model for $j_* \O_{\X_K^\circ}$ is for example the coadmissible $\Di$-module $\M_\X = \Di / (x\cdot \partial -1)$.
We have $\suppi(\M_\X) = \{ x = 0 \}$ with multiply one at $x = 0$.
Thus, for any admissible blow-up $\phi : \X' \to \X$, $\suppi(\phi^! \M_\X) = \{ x = 0 \}$ with multiplicity one.
As a consequence, the infinite support of the coadmissible module $j_* \O_{\X_K^\circ}$ is a point with multiplicity one at this point.
To conclude, $j_* \O_{\X_K^\circ}$ is a sub-holonomic $\Dcap$-module with two non-zero multiplicities each one equal to one.
\end{example}

\begin{remark}
Let $X_K$ be a smooth, quasi-compact and connected rigid analytic $K$-variety.
One can define directly the microlocalization $K$-algebras $\F_{\infty , r}(Y_i)$ of $\wideparen{\D}_{X_K}(Y_i)$ by choosing an admissible open cover $(Y_i)$of $X_K$ ; the integer $r$ for which this is possible depends on this cover.
In particular, for any coadmissible $\wideparen{\D}_{X_K}$-module $\M$, we can define its infinite support $\suppi(\M)$ by fixing such a cover.
Since this infinite support does not depend on $r$, $\suppi(\M)$ will be independent on the choice of the cover.
Thus, we can define in this setting a notion of sub-holonomicity satisfying the desire properties even when we do not have a smooth formal model.
\end{remark}

\paragraph{General situation}

We now assume that $\X$ is a smooth, separated and connected curve which is potentially not quasi-compact.
We denote by $T^*\X_K$ the cotangent space of $\X_K$ and we identify $\X_K$ with its zero section.

\begin{definition}
Let $\M$ be a coadmissible $\Dcap$-module and $(Y_i)$ an admissible quasi-compact open covering of $\X_K$.
The infinite support of $\M$ is defined by $\suppi(\M) := \bigcup \supp(\M_{|Y_i})$.
\end{definition}

This subset of $T^* \X_K$ does not depend on the choice of the admissible covering.
In particular, the fact that all the coadmissible $\wideparen{\D}_{Y_i}$-modules $\M_{|Y_i}$ are sub-holonomic is independent from the choice of the covering.
Moreover, this is closed being locally closed by definition.

\begin{definition}
A coadmissible $\Dcap$-module $\M$ is said to be sub-holonomic if there exists an admissible quasi-compact open covering $(Y_i)$ of $\X_K$ such that all the $\M_{|Y_i}$ are sub-holonomic.
\end{definition}

Let $\M$ be a sub-holonomic $\Dcap$-module. Then $\M_{|Y}$ is sub-holonomic in the sense of the preceding paragraph for any admissible quasi-compact open subset $Y$ of $\X_K$.
As a consequence of proposition \ref{prop5.4}, we deduce that any sub-holonomic $\Dcap$-module $\M$ is generically an integrable connection with finite rank $m_0(\M)$ given by the rank $m_0(\M_{|Y})$ for any admissible quasi-compact open subset $Y$ of $\X_K$.
Let $z \in \suppi(\M)$ and $Y$ be an admissible quasi-compact open neighbourhood of $z$ in $\X_K$. The multiplicity $m_z(\M)$ of $\M$ at $z$ is simply given by the multiplicity $m_z(\M_{|Y})$ as defined before.

\begin{definition}
We define the characteristic variety $\Char(\M)$ of a sub-holonomic $\Dcap$-module $\M$ to be the subset of $T^*\X_K$ consisting of an horizontal component $\X_K$ of multiplicity $m_0(\M)$ and of vertical components $(x = z)$ which are vertical lines with $x$-axis the points $z$ of $\suppi(\M)$ with multiplicities $m_z(\M)$.
The corresponding characteristic cycle is the formal sum
\[ \CC(\M) := m_0(\M) \cdot \X_K + \sum_{z \in \suppi(\M)} m_z(\M) \cdot (x = z) .\]
For a non sub-holonomic $\Dcap$-module, we set $\Char(\M) := T^*\X_K$.
\end{definition}

Let us point out the fact that this characteristic cycle may be infinite ; but this is finite over admissible quasi-compact open subsets of $\X_K$. 
Next proposition is a consequence of the equivalence of categories between sub-holonomic $\Dcap$-modules and sub-holonomic $\Dzr$-modules together with proposition \ref{prop5.3}.

\begin{prop}
Let $0 \to \M \to \Nn \to \L \to 0$ be an exact sequence of coadmissible $\Dcap$-modules.
Then $\Char(\Nn) = \Char(\M) \cup \Char(\L)$ and $\suppi(\Nn) = \suppi(\M) \cup \suppi(\L)$.
In particular, $\Nn$ is sub-holonomic if and only if $\M$ and $\L$ are. In this case, we have $\CC(\Nn) = \CC(\M) + \CC(\L)$.
\end{prop}

As a consequence, we deduce that the category $\mathcal{SH}_{\Dcap}$ of sub-holonomic $\Dcap$-module is abelian.
Thanks to the isomorphism \ref{eqext} and proposition \ref{propwh}, we get that any sub-holonomic $\Dcap$-module is weakly holonomic in the sense of \cite{ABW2}.
In particular, $\mathcal{SH}_{\Dcap}$ defines a full sub-category of the one of weakly holonomic $\Dcap$-modules.

\begin{remark}\,
\begin{enumerate}
\item
Let us mention that sub-holonomic $\Dcap$-modules are not stable by direct image.
We recall the following example of \cite{bitoun}.
Let $\X_K = \mathrm{Sp}(K \langle x \rangle)$, $U_K = \X_K \backslash \{0\}$ the open subset obtained by removing the origin and $j : U_K \hookrightarrow \X_K$ the natural embedding.
Note $\partial = \frac{d}{dx}$ and consider the differential operator $P_\lambda = x\cdot\partial - \lambda \in \Dcap(\X_K)$ with $\lambda \in K$.
The $\wideparen{\D}_{U_K}$-module $\M_\lambda = \wideparen{\D}_{U_K} / P_\lambda \simeq \O_{U_K} \cdot x^\lambda$ is coadmissible.
Indeed, this is an integrable connection since $x$ is invertible on $U$.
But \cite[theorem 1.1]{bitoun} tells us that the $\wideparen{\D}_{\X_K}$-module $j_*\M_{U_K}$ is not coadmissible when $\lambda$ is not of positive type.
\item
The category of sub-holonomic $\Dcap$-modules should contain Bode's holonomic modules introduced in \cite{bodehol} in the case of dimension one.
\end{enumerate}
\end{remark}

A coadmissible $\Dcap$-module $\M$ is zero if and only if $\Char(\M) = 0$.
In particular, a sub-holonomic $\Dcap$-module is zero if and only if its characteristic cycle is null.
The following corollary is then a consequence of the fact that on any admissible quasi-compact open subset $Y$ of $\X_K$, the characteristic cycle $\CC(\M_{|Y})$ is finite.

\begin{cor}
Let $\M$ be a sub-holonomic $\Dcap$-module and $Y$ an admissible quasi-compact open subset  of $\X_K$.
Then the module $\M_{|Y}$ has finite length less or equal to $\ell(\CC(\M_{|Y})) = m_0(\M) + \sum_{z \in Y \cap \suppi(\M)} m_z(\M) \in \N$.
\end{cor}

\begin{example}\,
\begin{enumerate}
\item
Let $\X$ be a smooth, separated and connected curve.
We consider a non-empty open subset $U$ of $\X_K$ and $\M$ an integrable connection on $\X_K$.
We denote by $j : U \hookrightarrow \X_K$ the natural embedding.
Then $j_* (\M_{|U})$ is a coadmissible $\Dcap$-modules by \cite[section 10.5]{ABW2}.
Since this module is generically an integrable connection, this is sub-holonomic.
More generally, the coadmissible $\Dcap$-modules $R^i j_* (\M_{|U})$ are also sub-holonomic.
\item
Let $\Arig$ be the rigid analytic line over $K$. For any $r \in |K^*|^\Q$, we denote by $\mathbb{B}(r) = \mathrm{Sp}(\O(r))$ the closed rigid analytic ball of radius $r$ with
\[ \O(r) = \left\{ \sum_{\ell \in \N} \alpha_\ell \cdot t^\ell, \, \alpha_\ell \in K :  |\alpha_\ell| \cdot r^\ell \to 0 \right\} .\]
We have on $\mathbb{B}(r)$ the Fr\'echet-Stein $K$-algebra of rapidly converging differential operators
\[ \wideparen{\D}(r) := \wideparen{\D}_{\Arig}(\mathbb{B}(r)) = \left\{ \sum_{n \in \N} a_n \cdot \partial^n : \, a_n \in \O(r), \, \forall R > 0, \, R^n \cdot |a_n| \to 0 \right\} = \varprojlim_k \D_k(r)\]
where $\D_k(r) := \left\{ \sum_{n \in \N} a_n \cdot (\varpi^k\partial)^n, : \, a_n \in \O(r), \, |a_n| \to 0 \right\}$ is a Banach $K$-algebra.
We now fix an increasing sequence of elements $(r_m)_m$ of $|K^*|^\Q$.
In this case, $\bigcup_m \mathbb{B}(r_m)$ is an admissible cover of $\Arig$.
We have $\O(\Arig) = \bigcap_m \O(r_m)$ and $\wideparen{\D}_{\Arig}(\Arig) = \bigcap_m \wideparen{\D}(r_m)$.
Moreover, we easily check, using properties of the coadmissible $\wideparen{\D}$-modules of \cite{ardakov}, that a $\widehat{\D}_{\Arig}$-module $\M$ is coadmissible if and only if for all $m \in \N$, $\M(r_m) := \M(\wideparen{\D}(r_m))$ is a coadmissible $\wideparen{\D}(r_m)$-module and for any $m' \geq m$,
\[ \M(r_m) \simeq \wideparen{\D}(r_m) \wideparen{\otimes}_{\widehat{\D}(r_{m'})} \M(r_{m'}) .\]
For $m' \geq m$, we have $\suppi(\M(r_m)) \subset \suppi(\M(r_{m'}))$.
Then $\suppi(\M) := \bigcup_m \suppi(\M(r_m))$ is an increasing union.
By definition, a coadmissible $\wideparen{\D}_{\Arig}$-module $\M$ is sub-holonomic if all the coadmissible modules $\M(r_m)$ are sub-holonomic, ie if the infinite supports $\suppi(\M(r_m))$ are all finite.
Moreover, if $\M(r_m)$ is sub-holonomic, so are the modules $\M(r_n)$ for any $n \leq m$.

\item
Let us detail one example coming from \cite[section 6.4]{patel}.
From now on, $K$ is a finite extension of $\Q_p$, $\Lambda = \V^{\oplus 2}$ a lattice of $K^{\oplus 2}$ and $\X$ the formal completion of the smooth model $\mathbb{X} = \mathbb{P}\text{roj}(\Lambda)$ of $\mathbb{P}_K^1 = \mathbb{P}\text{roj}(K^{\oplus 2})$ associated to $\Lambda$.
We denote by $\X_n$ the formal completion of the semi-stable model $\mathbb{X}_n$ of $\mathbb{P}_\V^1$ defined in \cite[Section 2.1]{patel} : we get $\mathbb{X}_1$ by blowing up $\mathbb{X}$ along the $\mathbb{F}_q$-rational points of its special fiber and $\mathbb{X}_{n+1}$ is obtained by blowing up the
smooth $\mathbb{F}_q$-rational points of the special fiber of $\mathbb{X}_n$.
Let us note $\X_n^\circ$ the open subscheme of $\X_n $ obtained by removing its smooth $\mathbb{F}_q$-rational points $\X_n^{\text{sm}}(\mathbb{F}_q)$.
The inductive limit $\widehat{\Sigma}_{0 , K} := \varinjlim_n \X_n^\circ$ is then a formal model for the $p$-adic upper half plane $\Sigma_{0 , K} := \mathbb{P}_K^{1 , \text{rig}} \backslash \mathbb{P}_K^1(K)$ over $K$.
The group $G_0 := \text{GL}(2,\V)$ acts on the formal scheme $\X_n$ and there is an induced diagonal left action of $G_0$ on the sheaves $\widehat{\D}^{(m)}_{\X_n , k , \Q}$ and $\widetilde{\D}^\dagger_{n , \Q} := \varinjlim_m \widehat{\D}^{(m)}_{\X_n , n , \Q}$.
We consider the overconvergent isocrystal $\mathcal{L}\text{oc}_n^\dagger(\O(\Sigma_{0 , K})_n)$ on $\X_n$ given by the functions which are regular on $\X_n^\circ$ and have overconvergent singularities along the closed subset $\X_n^{\text{sm}}(\mathbb{F}_q)$.
By \cite[remark 6.4.5]{patel}, this sheaf is a $G_0$-equivariant coherent $\widetilde{\D}^\dagger_{n , \Q}$-module.
Moreover, we dispose of the following description given in \cite[corollary 6.4.4]{patel}.
There exists a short exact sequence $0 \to \O_{\X_n , \Q} \to (\mathcal{L}\text{oc}_n^\dagger(\O(\Sigma_{0 , K})_n)) \to \M \to 0$ of $G_0$-equivariant $\widetilde{\D}^\dagger_{n , \Q}$-modules, where $\M$ is a skyscraper sheaf supported at the smooth $\mathbb{F}_q$-rational points $\X_n^{\text{sm}}(\mathbb{F}_q)$ of $\X_n$.
Let $\M = \mathcal{L}\text{oc}^\dagger(\O(\Sigma_{0 , K}))$ be the associated coadmissible $\wideparen{\D}_{\mathbb{P}_K^{1 , \text{rig}}}$-module.
Thanks to \cite[proposition 3.1.10]{huyghe}, the module $\M$ can be written as a projective limit of coadmissible $\D_{\X_n , \infty}$-modules $\M_{\X_n} = \varprojlim_k \M_{\X_n , k}$.
The coherent $\widehat{\D}^{(0)}_{\X_n , k , \Q}$-module $\M_{\X_n , k}$ admits a short exact sequence similar to the preceding one for $k$ large enough.
Hence, this is holonomic and $\supp(\M_{\X_n , k , r}) = \X_n^{\text{sm}}(\mathbb{F}_q)$ for $k \geq r$ sufficiently large.
In particular, $\M_{\X_n}$ is sub-holomomic with $\suppi(\M_{\X_n}) = \X_n^{\text{sm}}(\mathbb{F}_q)$.
More precisely, its characteristic variety contains the zero section with horizontal multiplicity exactly one and vertical components with x-axis at $\X_n^{\text{sm}}(\mathbb{F}_q)$.
It follows that $\M= \mathcal{L}\text{oc}^\dagger(\O(\Sigma_{0 , K}))$ is sub-holonomic with $\suppi(\M) = \varprojlim_n \text{sp}_{\X_n}^{-1}(\X_n^{\text{sm}}(\mathbb{F}_q))$, where $\text{sp}_{\X_n} : \mathbb{P}_K^{1 , \text{rig}} \to \X_n$ is the specialization map.
Moreover, this is an integrable connection of rank one on the $p$-adic upper half plane $\Sigma_{0 , K}$.
\end{enumerate}
\end{example}

\bibliographystyle{plain}
\bibliography{biblio.bib}

\begin{thebibliography}{10}

\bibitem{adriano}
Tomoyuki Abe and Adriano Marmora.
\newblock On $p$-adic product formula for epsilon factors.
\newblock {\em J. Inst. Math. Jussieu}, 14:275--377, 2015.

\bibitem{ardakov2}
K~Ardakov and S~Wadsley.
\newblock D-modules on rigid analytic spaces ii: Kashiwara's equivalence.
\newblock {\em Journal of Algebraic Geometry}, 27:647--701, 2018.

\bibitem{ABW2}
Konstantin Ardakov, Andreas Bode, and Simon Wadsley.
\newblock $\wideparen{D}$-modules on rigid analytic spaces {III}: Weak
  holonomicity and operations.
\newblock {\em Compositio Mathematica}, 157(12):2553--2584, 2021.

\bibitem{ardakov}
Konstantin Ardakov and Simon Wadsley.
\newblock $\wideparen{D}$-modules on rigid analytic spaces \text{I}, 2015.

\bibitem{ABW1}
Konstantin Ardakov and Simon~J Wadsley.
\newblock $\wideparen{D}$-modules on rigid analytic spaces {I}.
\newblock {\em Journal f{\"u}r die reine und angewandte Mathematik (Crelles
  Journal)}, 2019(747):221--275, 2019.

\bibitem{berthelot1}
Pierre Berthelot.
\newblock $\mathcal{D}$-modules arithmétiques {I}. {O}pérateurs différentiels
  de niveau fini.
\newblock In {\em Annales scientifiques de l'Ecole normale supérieure},
  volume~29, pages 185--272, 1996.

\bibitem{bitoun}
Thomas Bitoun and Andreas Bode.
\newblock Extending meromorphic connections to coadmissible
  $\wideparen{D}$-modules.
\newblock {\em Journal f{\"u}r die reine und angewandte Mathematik (Crelles
  Journal)}, 2021(778):97--118, 2021.

\bibitem{bodehol}
Andreas Bode.
\newblock {Holonomic D-cap-modules on rigid analytic spaces}, 2025.

\bibitem{rigid}
Kazuhiro Fujiwara and Fumiharu Kato.
\newblock Foundations of rigid geometry i, 2017.

\bibitem{garnier}
Laurent Garnier.
\newblock Théorèmes de division sur $\widehat{D}^{(0)}$ et applications.
\newblock {\em Bulletin de la Société Mathématique de France}, 123(4):547--589,
  1995.

\bibitem{hallopeau1}
Raoul Hallopeau.
\newblock $\widehat{D}^{(0)}_{k}$-modules holonomes sur une courbe formelle.
\newblock {\em Journal de th\'eorie des nombres de Bordeaux}, 36(3):869--917,
  2024.

\bibitem{hallopeau3}
Raoul Hallopeau.
\newblock {Cycle caract{\'e}ristique pour les $\mathcal{D}$-modules
  coadmissibles sur une courbe formelle}.
\newblock working paper or preprint, March 2025.

\bibitem{huyghe}
Christine Huyghe, Tobias Schmidt, and Matthias Strauch.
\newblock Arithmetic structures for differential operators on formal schemes.
\newblock {\em Nagoya Mathematical Journal}, 243:157--204, 2021.

\bibitem{huyghe2}
Christine Huyghe, Tobias Schmidt, and Matthias Strauch.
\newblock Arithmetic differential operators with congruence level structures:
  First results and examples.
\newblock {\em Journal of Number Theory}, 237:332--352, 2022.

\bibitem{patel}
Deepam Patel, Tobias Schmidt, and Matthias Strauch.
\newblock Locally analytic representations of $\text{GL}(2,\text{L})$ via
  semistable models of $\mathbb{P}^1$.
\newblock {\em Journal of the Institute of Mathematics of Jussieu}, 18:125 --
  187, 2014.

\bibitem{peter}
Michael Temkin Annette~Werner Peter~Schneider, Peter~Scholze.
\newblock Non-archimedean geometry and applications.
\newblock {\em Oberwolfach Rep}, 19 no.1:231--302, 2022.

\bibitem{stacks-project}
The {Stacks Project Authors}.
\newblock \textit{Stacks Project}.
\newblock \url{https://stacks.math.columbia.edu}, 2018.

\bibitem{virrion}
Anne Virrion.
\newblock Dualit{\'e} locale et holonomie pour les $\mathcal{\D}$-modules
  arithm{\'e}tiques.
\newblock {\em Bulletin de la Soci{\'e}t{\'e} Math{\'e}matique de France},
  128(1):1--68, 2000.

\bibitem{zab}
Gergely Z\'{a}br\'{a}di.
\newblock Generalized {R}obba rings.
\newblock {\em Israel J. Math.}, 191(2):817--887, 2012.
\newblock With an appendix by Peter Schneider.

\end{thebibliography}

\end{document}